\documentclass[12pt,letter]{amsart}
\usepackage{amscd}
\usepackage{amssymb}
\usepackage[centering,text={15.5cm,23cm}]{geometry} 
\usepackage{graphicx,color,caption}
\usepackage{comment}
\usepackage{wrapfig}
\usepackage[all]{xy}
\usepackage{mathrsfs}
\usepackage{mathtools}
\usepackage{marvosym}
\usepackage{stmaryrd}
\usepackage{srcltx}
\usepackage{hyperref}
\usepackage{upgreek} 
\usepackage[inline]{enumitem} 

\definecolor{shadecolor}{rgb}{1,0.9,0.7}

\newtheorem{theorem}{Theorem}[section]
\newtheorem{lemma}[theorem]{Lemma}
\newtheorem{lemma-definition}[theorem]{Lemma-Definition}
\newtheorem{proposition}[theorem]{Proposition}
\newtheorem{corollary}[theorem]{Corollary}

\theoremstyle{definition}

\newtheorem{definition}[theorem]{Definition}
\newtheorem{construction}[theorem]{Construction}
\newtheorem{assumption}[theorem]{Assumption}

\theoremstyle{remark}
\newtheorem{remark}[theorem]{Remark}

\numberwithin{equation}{section}
\numberwithin{figure}{section}

\newcommand {\lfor} {\llbracket}
\newcommand {\rfor} {\rrbracket}

\newcommand{\DD} {\mathbb{D}}

\newcommand{\NN} {\mathbb{N}}
\newcommand{\ZZ} {\mathbb{Z}}
\newcommand{\QQ} {\mathbb{Q}}
\newcommand{\RR} {\mathbb{R}}
\newcommand{\CC} {\mathbb{C}}

\newcommand{\PP} {\mathbb{P}}

\newcommand{\GG} {\mathbb{G}}

\newcommand {\shA}  {\mathcal{A}}
\newcommand {\shAff} {\mathcal{A}\text{\textit{ff}}}

\newcommand {\shD}  {\mathcal{D}}

\newcommand {\shF}  {\mathcal{F}}

\newcommand {\shHom} {\mathcal{H}\!\text{\textit{om}}}

\newcommand {\shL}  {\mathcal{L}}
\newcommand {\shM}  {\mathcal{M}}
\newcommand {\shMPA} {\mathcal{MP\!A}}

\newcommand {\shN}  {\mathcal{N}}
\newcommand {\shO}  {\mathcal{O}}

\newcommand {\shPL} {\mathcal{PL}}

\newcommand {\shR}  {\mathcal{R}}
\newcommand {\shS}  {\mathcal{S}}
\newcommand {\shT}  {\mathcal{T}}

\newcommand {\shP}  {\mathcal{P}}

\newcommand {\shX}  {\mathcal{X}}
\newcommand {\shY}  {\mathcal{Y}}

\newcommand {\foD}  {\mathfrak{D}}

\newcommand {\foS}  {\mathfrak{S}}

\newcommand {\foX}  {\mathfrak{X}}

\newcommand {\foZ}  {\mathfrak{Z}}

\newcommand {\fob}  {\mathfrak{b}}

\newcommand {\fom}  {\mathfrak{m}}

\newcommand {\fop}  {\mathfrak{p}}

\newcommand {\fou}  {\mathfrak{u}}

\newcommand {\an}  {\mathrm{an}}

\newcommand {\dlog} {\operatorname{dlog}}

\newcommand {\eps}  {\varepsilon}

\newcommand {\gp}  {{\operatorname{gp}}}
\newcommand {\Gr}  {\operatorname{Gr}}

\newcommand {\Hom}  {\operatorname{Hom}}

\newcommand {\hra} {\hookrightarrow}

\newcommand {\id}  {\operatorname{id}}
\newcommand {\im}  {\operatorname{im}}

\newcommand {\Int}  {\operatorname{Int}}

\renewcommand {\ker } {\operatorname{ker}}

\newcommand {\la}  {\leftarrow}

\newcommand {\lra}  {\longrightarrow}
\newcommand {\ls}  {\dagger}

\newcommand {\M} {\mathcal{M}}

\newcommand {\maxid} {\mathfrak{m}}

\newcommand {\MPA} {\operatorname{MPA}}

\renewcommand{\O}  {\mathcal{O}}

\newcommand {\ol} {\overline}

\renewcommand{\P}  {\mathscr{P}}

\newcommand {\Proj} {\operatorname{Proj}}

\newcommand {\ra}  {\to}

\newcommand {\res}  {\operatorname{res}}
\newcommand {\restxt}  {\mathrm{res}}

\newcommand {\scrS}  {\mathscr{S}}

\newcommand {\sing} {\mathrm{sing}}

\newcommand {\Spec} {\operatorname{Spec}}

\newcommand {\Spf}  {\operatorname{Spf}}

\newcommand {\trop}  {{\operatorname{trop}}}

\newcommand {\ul} {\underline}

\newcommand {\val}  {{\operatorname{val}}}

\newcommand {\D} {\mathfrak D}

\newcommand {\X} {\mathfrak X}

\def\mydate{\ifcase\month \or January\or February\or March\or
April\or May\or June\or July\or August\or September\or October\or 
November\or December\fi \space\number\day,\space\number\year}

\makeatletter
\let\oldcite\cite
\renewcommand{\cite}{\@ifnextchar[{\@newcite}{\oldcite}}
\def\@newcite[#1]#2{\oldcite{#2},\,#1}
\makeatother

\begin{document}

\title[Period integrals from wall structures]
{Period integrals from wall structures via tropical cycles, canonical
coordinates in mirror symmetry and analyticity of toric degenerations}
\author{Helge Ruddat, Bernd Siebert}

\address{\tiny JGU Mainz, Institut f\"ur Mathematik, Staudingerweg 9, 55128 Mainz, Germany}
\email{ruddat@uni-mainz.de}

\address{\tiny Department of Mathematics, University of Texas at Austin,
2515 Speedway Stop, Austin, TX 78712, USA}
\email{siebert@math.utexas.edu}
\thanks{H.~Ruddat was partially supported by the Carl-Zeiss foundation
and DFG grant RU 1629/4--1}

\begin{abstract}
We give a simple expression for the integral of the canonical holomorphic volume
form in degenerating families of varieties constructed from wall structures and
with central fiber a union of toric varieties. The cycles to integrate over are
constructed from tropical 1-cycles in the intersection complex of the central
fiber.

One application is a proof that the mirror map for the canonical
formal families of Calabi-Yau varieties constructed by Gross and the second
author is trivial. We also show that these families are the completion of an
analytic family, without reparametrization, and that they are formally versal as
deformations of logarithmic schemes. Other applications include canonical
one-parameter type III degenerations of K3 surfaces with prescribed Picard
groups.

As a technical result of independent interest we develop a theory of period
integrals with logarithmic poles on finite order deformations of normal crossing
analytic spaces.
\end{abstract}

\maketitle
\tableofcontents

\setcounter{tocdepth}{1}
\section{Introduction}
A \emph{period} of a complex manifold $X$ is the integral $\int_\beta\alpha$ of
a holomorphic differential $k$-form $\alpha$ over a singular $k$-cycle $\beta$
on $X$. The classical example is the elliptic integral $\int
dx/\sqrt{x^3+ax+b}$, an integral over a closed curve of the holomorphic one form
$y^{-1}dx$ on the elliptic curve $y^2=x^3+ax+b$. More modern accounts emphasize
the interpretation of periods in terms of Hodge theory and their dependence on
varying $X$ and $\alpha$ in a holomorphic family \cite{griffithsBullAMS}. This
interpretation is of fundamental importance in the study of moduli spaces
\cite{CMSP}. Another fascinating aspect of period integrals is the countable set
of values obtained for algebraic varieties defined over $\QQ$ \cite{periods}.

The main result of this paper gives a simple closed formula of a class
of period integrals for families of complex manifolds naturally arising in mirror
symmetry and in the study of cluster varieties. The families $\shX\to S$
considered have a special fiber $X_0$ that is a union of projective toric
varieties of dimension $n$, glued pairwise along toric divisors. In particular,
$X_0$ is normal crossings outside a subset of codimension two. The special fiber
is conveniently represented by the union of momentum polytopes, glued pairwise
along their facets according to the gluing of the irreducible components of
$X_0$, thus forming a cell complex $\P$ with underlying topological space $B$ a
pseudo-manifold, possibly with boundary. Outside codimension two, the
family $\shX\to S$ is built from toric pieces via special isomorphisms encoded
in a \emph{wall structure} on $B$. The special isomorphisms respect the toric
holomorphic differential $n$-forms $z_1^{-1}dz_1\wedge\ldots\wedge z_n^{-1}d
z_n$, which thus define a global relative holomorphic $n$-form $\Omega$ for
$\shX$ over the parameter space $S$.

For example, for any Laurent polynomial
$f\in\CC[u_1^{\pm1},\ldots,u_{n-1}^{\pm1}]$, the family of subvarieties
\begin{equation}
\label{Eqn: Basic equation}
zw=f\cdot t^\kappa
\end{equation}
of $\CC^2\times (\CC^*)^{n-1}$ parametrized by $t\in\CC$ is of this form. Such
families arise as mirrors to local Calabi-Yau manifolds \cite{CKYZ},\cite{ICM},
and in mirror symmetry for varieties of general type \cite{GKR},\cite{AAK}. If
the wall structure is not locally finite, this picture is accurate only at
finite orders in the deformation parameter. A careful treatment of period
integrals with logarithmic poles at $t=0$ in this setup is given in
Appendix~\ref{Sect: Period integrals}.

In the simplest versions \cite{logmirror1},\cite{affinecomplex}, $S$ is the
spectrum of a discrete valuation ring, or a disk analytically, but in any case,
$S$ is an open subset of an affine toric variety or its completion along a toric
divisor \cite{GHK},\cite{theta}. Thus there is a well-defined notion of monomial
function on $S$.

For the domain of integration, we consider continuous deformations $\beta_t$ of
a class of $n$-cycles $\beta$ on $X_0$ that generically fiber as a real
$(n-1)$-torus bundle over a graph in $B$ and which intersect the singular locus
of $X_0$ transversely in some sense. In a nutshell, our main result says that in
the best cases, which include \cite{affinecomplex} and \cite{GHK}, there are
constants $c\in \CC$ and a monomial $t^q$ on $S$ with
\[
\frac{1}{(2\pi\sqrt{-1})^{n-1}}\int_{\beta_t} \Omega_{X_t}= c+\log t^q,
\]
as a holomorphic function outside $t^q=0$ and up to multiples of
$2\pi\sqrt{-1}$. This result is highly remarkable since for algebraically
parametrized families, period integrals of this form typically lead to
transcendental functions. In fact, replacing $t$ by $h\cdot t$ for some
invertible analytic function $h$ changes the right-hand side by a summand $\log
h$. Thus while the logarithmic monomial behavior can be expected for cycles of
our form, the fact that $c$ is a constant is very special to the particular
construction of the family $\shX\to S$ via a wall structure.

The most obvious application of this result is to \emph{mirror symmetry}. On the
complex side of mirror symmetry, one is looking at families $\shX\to S$ of
Calabi-Yau varieties with topological monodromy around the critical locus
unipotent of maximally possible exponent \cite{CdGP91}, \cite{De93}, \cite{Mo93}.
In this situation, the limiting mixed Hodge structure on the cohomology of a
nearby smooth fiber turns out to be of Hodge-Tate type \cite{De93}, and
exponentials of the kind of period integrals studied here provide a
distinguished set of holomorphic functions on the parameter space. These
functions provide a set of coordinates at points where the family is
semi-universal, that is, where the Kodaira-Spencer map is an isomorphism. Since
these functions only depend on discrete choices, they are known as
\emph{canonical coordinates} in mirror symmetry. The identification of the
complexified K\"ahler moduli space of the mirror with complex moduli works by
canonical coordinates. Thus our result says that the mirror map is monomial on
the subspace generated by our cycles. In favorable situations one obtains
full-dimensional pieces of the complexified K\"ahler cone on the mirror side.

As another important application, we prove a strong \emph{analyticity} result
for the canonical toric degenerations of \cite{affinecomplex} and their
universal refinement in \cite{theta}, Theorem~A.7. This result should be
important for the symplectic study of canonical toric degenerations.


\subsection{Toric degenerations from wall structures}
\label{Par: Toric degenerations}
For more precise statements we now give more details on the setup and
construction. We work in the general setting of \cite{theta} and fix a finitely
generated $\CC$-algebra $A$ and some $k\in\NN$. The algebra $A$ provides moduli for
the construction and may be assumed to be $\CC$ at first reading. The base ring
of our degeneration is $A_k=A[t]/(t^{k+1})$, so $k$ determines the order of
deformation to be considered.
\medskip

\setcounter{tocdepth}{1}
\subsubsection{Polyhedral affine manifolds \texorpdfstring{$(B,\P)$}{(B,P)}}
The basic arena of all constructions is a cell complex $\P$ of integral
polyhedra with underlying topological space $B$ an $n$-dimensional
pseudo-manifold with possibly empty boundary (\cite{theta}, Definition~1.1). All
constructions happen away from codimension two. We reserve the letter $\sigma$
for maximal cells and $\rho$ for codimension one cells of $\P$, respectively, possibly
adorned. For a cell $\tau\in\P$ we denote by $\Lambda_\tau\simeq \ZZ^{\dim\tau}$
the group of integral tangent vector fields on the interior $\Int\tau$ of
$\tau$. We also need maximal cells of the barycentric subdivision of a
codimension one cell $\rho$, and these are denoted $\ul\rho$. By writing
$\ul\rho$ it is understood that $\rho$ is the codimension one cell of $\P$
containing $\ul\rho$.

Denote by $\Delta\subset B$ the union of all $(n-2)$-cells of the barycentric
subdivision that lie in the $(n-1)$-skeleton of $\P$, that is, which are
disjoint from the interiors of maximal cells. On $B\setminus\Delta$ we assume
given an integral affine structure that restricts to the usual integral affine
structure on the interiors of maximal cells. Note that this amounts merely to
specifying, for each $\ul\rho$ not contained in $\partial B$, the parallel
transport through $\ul\rho$ of a primitive integral vector complementary to
$\Lambda_\rho$ on one of the two neighboring maximal cells $\sigma$ of $\rho$
to the other neighboring cell $\sigma'$. The polyhedral complex $\P$ along with
the affine structure on $B=|\P|$ away from $\Delta$ is what we call a
\emph{polyhedral affine pseudo-manifold}. We use the notation
$\Delta_2\subset\Delta$ for the smaller set defined by the $(n-2)$-skeleton of $\P$.
\medskip

\subsubsection{Kinks $\kappa_{\ul\rho}$ and multivalued piecewise affine
function $\varphi$}
The second piece of data is the collection of exponents
$\kappa\in\NN\setminus\{0\}$ appearing in the local models \eqref{Eqn: Basic
equation} in codimension one. There is one such exponent for each $\ul\rho$, so
these exponents may vary\footnote{In \cite{affinecomplex} and \cite{GHK}, kinks
depend only on $\rho$, but they do depend on $\ul\rho\subset\rho$ in some proofs
of \cite{GHK}.} along a codimension one cell $\rho$. As a matter of notation, we
denote the collection of all $\kappa_{\ul\rho}$ by the associated multivalued
piecewise affine function $\varphi$ (\cite{theta}, Definition~1.8).
\medskip

\subsubsection{Wall structures}
The third piece of data is a wall structure $\scrS$ on our polyhedral affine
pseudo-manifold, as defined in \cite{theta}, Definition~1.22. The wall structure
consists of a finite collection of \emph{walls}, each wall being an
$(n-1)$-dimensional rational polyhedron $\fop$ contained in some cell of $\P$,
along with an algebraic function $f_\fop$. The walls define an
$(n-1)$-dimensional cell complex, assumed to cover all $(n-1)$-cells $\rho\in\P$
and to subdivide each maximal cell of $\P$ into (closed) convex \emph{chambers},
denoted $\fou$. There are thus two kinds of walls, depending on whether the
minimal cell of $\P$ containing $\fop$ is a maximal cell $\sigma$ or a
codimension one cell $\rho$. In the first case, \emph{walls of codimension
zero}, $f_\fop$ is of the form\footnote{The definition in \cite{theta} admitted
$f_\fop$ of the more general form $1+\sum_i a_i z^{m_i}t^{\ell_i}$ with
$\ell_i>0$ and $m_i\in\Lambda_\sigma$ tangent to $\fop$. Such a wall can be
decomposed into walls of the more restrictive form considered here, provided
$f_\fop$ can be written as a product $\prod_i(1+a'_i z^{m'_i}t^{\ell'_i})$. This
is the case iff the (finite) Taylor series expansion of $\log f_\fop$ at $1\in
A$ has no terms that are pure powers of $t$. This property is crucial for walls
of codimension~$0$ not to contribute to the period integral. It is fulfilled in
all known cases.}
\[
f_\fop=1+a z^m t^\ell,
\]
with $a\in A$, $\ell>0$ and $z^m$ the monomial in the Laurent polynomial ring
$\CC[\Lambda_\sigma] \simeq\CC[z_1^{\pm1},\ldots,z_n^{\pm1}]$ defined by some
$m\in \Lambda_\sigma$ \emph{tangent to $\fop$}. The second case, \emph{walls
of codimension one}, cover the sources of the inductive construction of the wall
structure. Such a wall is therefore also called \emph{slab} and written with a
different symbol $\fob$ instead of $\fop$ for easier distinction. In this case,
there are no conditions on $f_\fob$ other than that the exponents of monomials
be tangent to $\rho$, that is,
\[
f_\fob\in A[\Lambda_\rho][t].
\]
Here $\Lambda_\rho\simeq\ZZ^{n-1}$ denotes the group of integral tangent vector
fields on $\rho$.
\medskip

\subsubsection{Construction of the scheme $X^\circ_k/A_k$} \label{subsec-make-Xk}
From the wall structure we build a scheme $X_k^\circ$ over $A_k=A[t]/(t^{k+1})$
assuming a \emph{consistency condition}, by taking one copy $\Spec R_\fou$ with
$R_\fou= A_k[\Lambda_\sigma]$ for each chamber $\fou\subseteq\sigma$ and one copy of
$\Spec R_\fob$ with
\begin{equation}\label{Eqn: ring in codim one}
R_\fob=A_k[\Lambda_\rho][Z_+,Z_-]/(Z_+ Z_- - f_\fob t^{\kappa_{\ul\rho}})
\end{equation}
for each slab $\fob$. A wall $\fop$ of codimension zero defines a \emph{wall
crossing automorphism} of $A_k[\Lambda_\sigma]$, for $\sigma$ the maximal cell
containing $\fop$, see \eqref{Eqn: Wall crossing formula} below and
\cite{theta}, \S2.3. The consistency condition in codimension zero
(\cite{theta}, Definition~2.13) is equivalent to saying that sequences of such
automorphisms identify all $\Spec R_\fou$ for chambers contained in the same
maximal cell $\sigma$ in a consistent fashion.

If a slab $\fob\subseteq\ul\rho$ is a facet of a chamber $\fou\subseteq\sigma$,
there is an open embedding
\begin{equation}\label{Eqn: Canonical open emb}
\Spec R_\fou\lra \Spec R_\fob
\end{equation}
defined by the inclusion $\Lambda_\rho\subset \Lambda_\sigma$ and by identifying
$Z_+$ with $z^\zeta$ for $\zeta\in\Lambda_\sigma$ a generator of
$\Lambda_\sigma/\Lambda_\rho$ pointing from $\rho$ into $\sigma$. For the other
chamber $\fou'$ containing $\fob$, contained in the maximal cell $\sigma'$ with
$\sigma\cap\sigma'=\rho$, the corresponding homomorphism $R_\fob\to R_{\fou'}$
maps $Z_-$ to $z^{\zeta'}$ with $\zeta'$ the parallel transport of $-\zeta$
through $\ul\rho$. In this procedure there is a choice of co-orientation of
$\rho$ that determines which maximal cell to take for $\sigma$, and a choice of
$\zeta\in\Lambda_\sigma$, but any two choices lead to isomorphic results.
Consistency in codimension one provides the necessary cocycle condition to
assure the existence of a scheme $X_k^\circ$ with open embeddings of all $\Spec
R_\fou$ and $\Spec R_\fob$ compatible with all wall crossing automorphisms and
all open embeddings \eqref{Eqn: Canonical open emb}.

If $\partial B\neq\emptyset$, the codimension one cells $\ul\rho$ contained in
$\partial B$ require a slightly different treatment that turns $\partial B$ into a
divisor in $X_k^\circ$. We do not review this construction here because all our
arguments take place on the complement of $\partial B$.
\medskip

\subsubsection{Codimension two locus, partial completion and theta functions}
The fiber $X_0^\circ$ of $X_k^\circ$ over $t=0$ is a product of $\Spec A$ with a
union of toric varieties, one for each maximal cell $\sigma$, glued pairwise
canonically along toric divisors as prescribed by the combinatorics of $\P$. By
construction, the toric varieties do not contain any toric strata of codimension
larger than one. For a maximal cell $\sigma$, the fan of the corresponding toric
variety is the $1$-skeleton of the normal fan of $\sigma$, so consists only of
the origin and the rays. While it is always possible to add the codimension two
strata to $X_0^\circ$ to arrive at a scheme $X_0$, the extension $X_k$ of the
flat deformation $X^\circ_k$ of $X_0^\circ$ to $X_0$ is a lot more subtle and in
particular, requires a consistency condition in codimension two. The approach
taken in \cite{GHK} and \cite{theta} to produce $X_k$ is to rely on the
construction of a canonical $A_k$-module basis of the homogeneous coordinate
ring, consisting of \emph{(generalized) theta functions}. For $B=(S^1)^n$ these
functions indeed agree with Riemannian theta functions. Theta functions will
only be used once in this paper, for the construction of the degenerate momentum
map in Proposition~\ref{Prop: momentum map}. There is one generalized theta
function $\vartheta_m$ for each integral point $m$ of $B$. We refer to
\cite{theta} for details. Our periods are computed entirely on $X_k^\circ$ and
hence the extension from $X_k^\circ$ to $X_k$ is largely irrelevant here.
\medskip

\subsubsection{Gluing data}
\label{Par: Gluing data}
One obvious way to introduce parameters in the construction is to compose the
open embeddings $\Spec R_\fou\to \Spec R_\fob$ from \eqref{Eqn: Canonical open
emb} with an $A_k$-linear toric automorphism of $\Spec R_\fou$. For
$R_\fou=A_k[\Lambda_\sigma]$ such an automorphism is given by a homomorphism
$\Lambda_\sigma\to A^\times$. The choices $s_{\sigma\ul\rho}\in
\Hom(\Lambda_\sigma,A^\times)$ for each $\ul\rho,\sigma$ with
$\ul\rho\subset\sigma$ is called \emph{(open) gluing data}. All previous notions
generalize, with consistency in codimension one and two now checked with the
open embeddings \eqref{Eqn: Canonical open emb} twisted by the given gluing
data. Gluing data may spoil projectivity or even the existence of the completed
central fiber $X_0\supset X_0^\circ$. Since the details of this extension are
not relevant for the present paper, we refer the interested reader to
\cite{theta}, Section~5. Gluing data change the period integral and will play an
important role in the application to analyticity, hence have to be taken into
consideration.

For simplicity of notation we write $X_0$ and $X_k$ instead of $X_0^\circ$ and
$X_k^\circ$ in the following discussion, but work only away from strata of codimension
larger than one.


\subsection{Singular cycles on $X_0$ from tropical $1$-cycles}
The $n$-cycles considered are defined from $n$-cycles $\beta$ on $X_0$ that
generically fiber as a finite union of real $(n-1)$-torus bundles over a graph
$\beta_\trop$ in $B$. The torus fiber over a non-vertex point of $\beta_\trop$
in the interior of a maximal cell $\sigma\in\P$ is an orbit under the conormal
torus $\xi^\perp\otimes U(1)\simeq U(1)^{n-1}$ inside the real torus
$\Hom(\Lambda_\sigma,\ZZ)\otimes_\ZZ U(1)\simeq U(1)^n$ acting on the toric
irreducible component $X_\sigma\subseteq X_0$ defined by $\sigma$. The matching
of the various orbits at a vertex amounts to the local vanishing of the boundary
of $\beta_\trop$ as a singular $1$-cycle with twisted coefficients\footnote{See
\cite{Br97}, \S{VI.12}, for singular homology with coefficients in a sheaf.} in
the local system $\Lambda$.
\medskip

\subsubsection{The degenerate momentum map $\mu:X_0\to B$}
To globalize we observe that each maximal cell $\sigma$ comes with a momentum
map $\mu_\sigma: X_\sigma\to \sigma$ of the corresponding irreducible component
$X_\sigma\subseteq X_0$. For trivial gluing data (all $s_{\sigma\ul\rho}=1$), the
$\mu_\sigma$ agree on codimension one strata to define a degenerate momentum map
$\mu:X_0\to B$. This map should be viewed as a limiting SYZ-fibration
\cite{SYZ}. There is a partial collapse of torus fibers over the deeper strata
of $X_0$ described explicitly by the Kato-Nakayama space of $X_0$ as a log
space, see \cite{KNreal} for some details. For non-trivial gluing data, the
$\mu_\sigma$ have to be composed with diffeomorphisms of the maximal cells
$\sigma$ to make them match over common strata. In the
projective setting one can use generalized theta functions for a canonical
construction, otherwise there may be obstructions to the existence of $\mu$ in
codimension two. The following is Proposition~\ref{Prop: momentum map}.

\begin{proposition}
If $X_0$ is projective, there exists a degenerate momentum map $\mu:X_0\to B$.
Without the projectivity assumption, such a map exists at least on the complement
of the union of toric strata of $X_0$ of codimension two.
\end{proposition}
\medskip

\subsubsection{The log singular locus $Z\subset X_0$, its amoeba image
$\shA\subset B$ and the adapted affine structure on $B\setminus(\Delta_2\cup\shA)$}
For each codimension one cell $\rho\in\P$ and any slab $\fob\subset \rho$, the
closure of the zero locus of $f_\fob$ defines a hypersurface $Z_\rho$ in the
codimension one locus $X_\rho\subset X_0$. By consistency, this zero locus is
independent of the chosen slab on $\rho$. From the local equation in codimension
one \eqref{Eqn: Basic equation},\eqref{Eqn: ring in codim one}, it follows that
$Z_\rho$ is the locus where the degeneration is not semi-stable and is indeed
singular even from the logarithmic point of view. We define the \emph{log
singular locus},
\[
Z=\bigcup_\rho Z_\rho,
\]
a codimension two subset of $X_0$ lying in the singular locus of $X_0$.
The image of $Z$ under our degenerate momentum map,
\[
\shA=\mu(Z),
\]
is called its \emph{amoeba image}. In fact, for each codimension one cell,
$\shA\cap\Int\rho$ is a diffeomorphic image of the hypersurface amoeba of
$f_\fob$, for any slab $\fob\subset\rho$. If the base ring $A$ is higher
dimensional, we first take a base change $A\to\CC$ to restrict to a slice of the
deformation $X_k\to \Spec A_k$ or work with a small analytic subset of $\Spec A$
for otherwise $\shA$ may be too large to be useful.

For $x\in B\setminus(\Delta_2\cup\shA)$ and $\fob$ a slab containing $x$, the restriction of
$f_\fob$ to $\mu^{-1}(x)$ has no zeros. Thus there exists a unique $m_x\in
\Lambda_\rho$ with the restriction of $z^{-m_x} f_\fob: \mu^{-1}(x)\to \CC^*$
contractible. This means that the adapted local equation
\begin{equation}
\label{Eqn: adapted local eqn}
Z_+ \big(z^{-m_x}Z_-\big) =\big(z^{-m_x}f_\fob\big) t^{\kappa_{\ul\rho}}
\end{equation}
locally analytically describes the toric normal crossings degeneration
\[
zw = t^{\kappa_{\ul\rho}}
\]
by taking $z=Z_+ $, $w=Z_-/f_\fob$. This observation motivates the
definition of an adapted integral affine structure on $B\setminus(\Delta_2\cup\shA)$ that
defines the parallel transport of $-\zeta'\in \Lambda_{\sigma'}$ through $x$ to
be $\zeta-m_x$ instead of $\zeta$, the integral tangent vector chosen in
connection with the gluing \eqref{Eqn: Canonical open emb}.

The set of integral tangent vectors for the adapted affine structure now defines
a local system $\Lambda$ on $B\setminus(\Delta_2\cup\shA)$ of free abelian groups of rank
$n$. The dual local system $\shHom(\Lambda,\ul\ZZ)$ is denoted $\check\Lambda$.
Note that if $x\in B\setminus(\Delta_2\cup\shA)$ lies in a maximal cell $\sigma$, we have a
canonical isomorphism of the stalk $\Lambda_x$ with $\Lambda_\sigma$.
\medskip

\subsubsection{Tropical $1$-cycles}
With the adapted affine structure on $B\setminus(\Delta_2\cup\shA)$ we are now in a position to
define the affine geometric data representing our singular $n$-cycles on $X_0$.

\begin{definition}
\label{Def: Tropical $1$-cycles}
A \emph{tropical one-cycle} $\beta_\trop$ is a twisted singular one-cycle on
$B\setminus(\Delta_2\cup\shA)$ with coefficients in the sheaf of integral
tangent vectors $\Lambda$, that is, $\beta_\trop\in
Z_1(B\setminus(\Delta_2\cup\shA),\Lambda)$.
\end{definition}

Thus a tropical one cycle is an oriented graph $\Gamma$ together with a map
$h:\Gamma\to B\setminus(\Delta_2\cup\shA)$ and for each edge $e\subseteq\Gamma$ a
section $\xi_e$ of $(h|_e)^*\Lambda$ such that for each vertex $v$ the cycle
condition $\sum_{e\ni v} \pm\xi_e=0$ holds, with sign depending on $e$ being
oriented toward or away from $v$. We typically assume without restriction that
$\xi_e\neq0$ for all $e$. One way to obtain such cycles is from a tropical curve
with a chosen orientation on each edge; the tangent vector for an edge $e$ is
then given by the oriented integral generator of the tangent space of $e$
multiplied by the weight of the edge. The balancing (cocycle) condition for tropical
curves implies that the associated twisted singular chain is a
cycle. We may thus think of twisted singular cycles carrying integral tangent
vectors as flabby versions of tropical curves. This motivates the use of the
word ``tropical''.
\medskip

\subsubsection{The singular cycle $\beta$ associated to a tropical $1$-cycle
$\beta_\trop$}
\label{Sub: construction of beta}
Fix a parameter value $a\in\Spec(A)_\an$ and let $X_0(a)$ be the fiber of $X_0$
over $a$. Let us now assume for simplicity that $\beta_\trop$ is transverse to
the $(n-1)$-skeleton of $\P$ and that each of its edges $e$ is embedded into a
single maximal cell $\sigma$. For each edge $e$, choose a section $S_e\subset
X_\sigma$ of $\mu_\sigma: X_\sigma\to \sigma$ over $e$, chosen compatibly over
vertices. Then define a chain $\beta_e$ over $e$ as the orbit of $S_e$ under the
subgroup of $\Hom(\Lambda_\sigma, U(1))\simeq U(1)^n$ mapping $\xi_e$ to $1$. If
$\xi_e$ is primitive, this subgroup equals $\xi^\perp\otimes U(1)\simeq
U(1)^{n-1}$, otherwise it is the product of this $(n-1)$-torus with $\ZZ/m_e\ZZ$
for $m_e\in\NN\setminus\{0\}$ the index of divisibility of $\xi_e$. At a vertex
$v$ of $\beta_\trop$ the cycle condition $\sum_{e\ni v} \pm \xi_e=0$, with signs
adjusting for the orientation of the edges at $v$, implies that the boundaries
over $v$ of the chains $\beta_e$ bound an $n$-chain $\Gamma_v$ over $v$. The
singular $n$-cycle associated to $\beta_\trop$ is now defined as
\[
\beta=\sum_e \beta_e+\sum_v\Gamma_v.
\]
The section $S_e$ is only unique in homology up to adding closed circles in
fibers; the orbit of such a circle yields a copy of the fiber class of
$\mu_\sigma$, which is homologically trivial in $X_0(a)$, so we obtain the
following.

\begin{lemma} \label{lemma-homo-trop-to-ordinary}
The association $\beta_\trop\mapsto \beta$ induces a well-defined homomorphism
\[
H_1(B\setminus(\Delta_2\cup\shA),\Lambda)\ra H_n(X_0(a),\ZZ).
\]
\end{lemma}

\begin{remark} \label{rem-Symington-CBM}
For $n=2$, the construction of $n$-cycles from tropical one-cycles was done
before in \cite{Symington} with a minor variation: Symington's tropical cycles
have boundary in the amoeba image $\shA$ of the affine structure, that is, they
are relative cycles in $H_1(B,\shA;\iota_*\Lambda)$ for
$\iota:B\setminus\shA\hra B$ the inclusion. Note that in this dimension,
$\shA\subset B$ is a finite set. Symington's alternative definition gives
nothing new compared to our cycles since relative one-cycles with boundary on
$\shA$ are homologous to cycles with little loops around $\shA$ and this process
lifts to singular cycles on $X_0(a)$. Similar relative cycles in higher
dimension are more peculiar; one needs to replace $\iota_*\Lambda$ by a subsheaf
of $\iota_*\Lambda$, see \cite{affinecoh} for details.

If $\beta_\trop$ is the tropical cycle associated to a tropical curve and the
section $S$ the restriction of the positive real locus in a real degeneration
situation, as discussed in \cite{KNreal}, then $\beta$ is a Lagrangian cell
complex. A related situation for $n=3$ arises in \cite{tropLag} for $n=3$.

\sloppy
More generally, a similar procedure produces singular cycles in
$H_{n-p+q}(X_0(a),\ZZ)$ from cycles in $H_{q}(B,\iota_*\bigwedge^p\Lambda)$,
well-defined up to adding cycles constructed from
$H_{q-1}(B,\iota_*\bigwedge^{p-1}\Lambda)$. For tropical cycles with boundary in
$\shA$, more care needs to be taken. See \cite{CBM13} for an application of
relative tropical 2-cycles to conifold transitions, see also \cite{affinecoh}.
\end{remark}

\fussy
The point of using the adapted affine structure on $B\setminus(\Delta_2\cup\shA)$ is as
follows. For any analytic family $\shX\to D$ with central fiber
$X_0(a)$ and given by \eqref{Eqn: ring in codim one} locally in codimension one,
there is a continuous family of $n$-cycles $\beta(t)$ for $t\in D\setminus
\RR_{<0}$ with $\beta(0)=\beta$. The reason for having to remove $\RR_{<0}$
in this statement is the topological monodromy action on $\beta(t)$ for varying
$t$ in a loop around the origin.

At a vertex $v$ of $\beta_\trop$ on a slab $\fob\subseteq\rho$, the local
situation is as follows. If $\xi_e\in\Lambda_\rho$, then, in adapted coordinates,
$zw=t^{\kappa_{\ul\rho}}$ describes $\shX$ locally, and the cycle $\beta$ is locally a product
of an $(n-2)$-chain $\gamma\approx U(1)^{n-2}$ with the union of two disks
$|z|\le 1$, $|w|\le 1$. In this case, $\beta(t)$ equals $\gamma$ times the
cylinder $zw=t^{\kappa_{\ul\rho}}$, $|z|,|w|\le 1$ and the local topological
monodromy is trivial. Otherwise, $\beta$ is locally the product of an
$(n-1)$-chain $\gamma \approx U(1)^{n-1}$ with a curve $\iota$ connecting
$z=1,w=0$ with $z=0, w=1$. In deforming to $t\neq 0$ we can again leave $\gamma$
untouched, but the curve $\iota$ deforms to a curve $\iota(t)$ on the cylinder
connecting $(z,w)=(1,1/t^{\kappa_{\ul\rho}})$ to $(z,w)=
(1/t^{\kappa_{\ul\rho}},1)$. The topological monodromy acts on $\iota(t)$ by a
$\kappa_{\ul\rho}$-fold Dehn twist. These Dehn-twists leave an expected
ambiguity of the construction of $\beta(t)$ by multiples of the vanishing cycle
$\alpha\approx (S^1)^n$. Note that there are also continuous families of cycles
homologous to $\alpha$ that converge to a fiber of the degenerate momentum map
$\mu$. In particular, $\alpha$ can be interpreted as a fiber of the SYZ
fibration. Note also that $\alpha$ can be viewed as constructed from a generator
of $H_0(B\setminus (\Delta_2\cup\shA),\iota_*\bigwedge^0\Lambda)$ in the
generalized construction mentioned in Remark~\ref{rem-Symington-CBM}.
\medskip

\subsubsection{Picard-Lefschetz transformation and $c_1(\varphi)$}
\label{Par: PL trf}
The effect of Picard-Lefschetz transformations on our singular cycles can be
written down purely in terms of affine geometry. Since this expression appears
in our period integrals, we review it here. The multivalued piecewise affine
function $\varphi$ defines a cohomology class in $H^1(B\setminus
(\Delta_2\cup\shA),\check\Lambda)$ denoted $c_1(\varphi)$, see \cite{theta}, \S1.2. Cap
product then defines an integer valued pairing with tropical cycles, that we
denote
\begin{equation}\label{Eqn: pairing with c_1}
\langle c_1(\varphi),\beta_\trop\rangle \in\ZZ.
\end{equation}
Explicitly, this pairing can be computed as follows. Assume without restriction
that $\beta_\trop$ is transverse to the $(n-1)$-skeleton of $\P$. Then for a
vertex $v$ of $\beta_\trop$ on a slab $\fob\subseteq\ul\rho$, let $e,e'$ be the
adjacent edges following the orientation of $\beta_\trop$. Denoting $\sigma$ the
maximal cell containing $e$, let $\check d_e\in\check\Lambda_\sigma$ be the
primitive generator of $\Lambda_\rho^\perp$ evaluating positively on tangent
vectors pointing from $\rho$ into $\sigma$. Then $v$ contributes the summand
$\langle \check d_e,\xi_e\rangle\cdot\kappa_{\ul\rho}$ to $\langle
c_1(\varphi),\beta_\trop\rangle$, the sum taken over all vertices of $\beta_\trop$ on slabs.

The following is Proposition~\ref{Lem: monodromy action}.

\begin{proposition}
\label{Prop: Picard-Lefschetz transformation}
Let $\beta_\trop\in Z_1(B\setminus(\Delta_2\cup\shA),\Lambda)$ be a tropical one-cycle and let
$\beta\in H_n(X_0(a),\ZZ)$ be the associated singular $n$-cycle. Then the
Picard-Lefschetz transformation of the deformation $\beta_t$ of $\beta$ to an
analytic smoothing $X_t$ of $X_0(a)$ acts by
\[
\beta_t\longmapsto \beta_t+\langle c_1(\varphi),\beta_\trop\rangle\cdot \alpha.
\]
Here $\alpha\in H_n(X_t,\ZZ)$ is the vanishing cycle.
\end{proposition}


\subsection{Statements of main results}
We need two more ingredients before being able to state the main theorem.
\medskip

\subsubsection{Pairing $\beta$ with gluing data}
Our gluing data $s=(s_{\sigma\ul\rho})$ also produces a first cohomology class,
this time in $H^1(B\setminus(\Delta_2\cup\shA),\check\Lambda\otimes A^\times)$. Just as
$c_1(\varphi)$, this cohomology class can be evaluated on tropical cycles via
the cap product to obtain an element of $A^\times$. We write
\begin{equation}
\label{Eqn: pairing with gluing data}
\langle s,\beta_\trop\rangle\in A^\times
\end{equation}
for this pairing. In the notation used for $c_1(\varphi)$ above, a vertex $v$ of
$\beta_\trop$ on a slab contained in $\ul\rho$ now contributes
$\big(s_{\sigma'\ul\rho}/ s_{\sigma\ul\rho}\big)^{\langle \check
d_e,\xi_e\rangle}$ as a multiplicative factor in the definition of $\langle
s,\beta_\trop\rangle$.
\medskip

\subsubsection{The complex Ronkin function}
\label{Sub: Ronkin fct}
For each value of the parameter space $\Spec A$, each slab function $f_\fob$
defines a holomorphic function on $\Hom(\Lambda_\rho,\CC^*)\simeq
(\CC^*)^{n-1}$. Such a holomorphic function $f$ has an associated \emph{Ronkin
function} \cite{Ro02} on $\RR^{n-1}$, defined by
\[
N_f(x)=\frac{1}{(2\pi\sqrt{-1})^{n-1}}\int_{\operatorname{Log}^{-1}(x)}
\frac{\log|f(z_1,\ldots,z_{n-1})|}{z_1\cdots z_{n-1}}dz_1\ldots dz_{n-1},
\]
with $\operatorname{Log}(z_1,\ldots,z_n)= (\log |z_1|,\ldots,\log |z_n|)$. This
function is piecewise affine on the complement of the hypersurface amoebae
$\shA_f=\operatorname{Log}(f=0)$ and otherwise continuous and strictly convex. It
plays a fundamental role in the study of amoebas \cite{PassareRullgard}. The
derivative of $N_f$ at a point $x\in \RR^n\setminus\shA_f$ is the homology class
of the restriction of $f$ to $\operatorname{Log}^{-1}(x)$, as a map
$U(1)^{n-1}\to\CC^*$. In particular, $N_f$ is locally constant near $x$ if and
only if this map is contractible.

Our period integrals involve the complex version of the Ronkin function for our
slab functions $f_\fob$. Let $x\in\Int\fob$ and $m_x\in \Lambda_\rho$ be as in
the definition of the adapted affine structure above. Taking $z_1,\ldots,z_{n-1}$
any toric coordinates on $\Spec\CC[\Lambda_\rho]$, we define the \emph{complex
Ronkin function} of $f_\fob$ at $x$ by
\begin{equation}
\label{Eqn: complex Ronkin fct}
\shR(z^{-m_x}f_\fob,x):= \frac{1}{(2\pi\sqrt{-1})^{n-1}} \int_{\mu^{-1}(x)}
\frac{\log \big(z^{-m_x} f_\fob(z_1,\ldots,z_{n-1})\big)}{z_1\cdots
z_{n-1}}dz_1\ldots dz_{n-1}\in\CC\lfor t\rfor.
\end{equation}
Under variation of parameters, that is, changing $A\to\CC$, the log singular
locus $Z$ and in turn the image $\shA$ moves. But as long as $x\not\in\shA$, the
complex Ronkin function varies analytically with the parameters, hence defines a
holomorphic function on appropriate open subsets of $\Spec(A)_\an$. Note also
that the real part of $\shR(z^{-m_x}f_\fob,x)$ equals $N_{z^{-m_x}f_\fob}$. But
$z^{-m_x}f_\fob$ is topologically contractible by the definition of $m_x$ and
hence $N_{z^{-m_x}f_\fob}$ is locally constant. In turn, the complex Ronkin
function is also locally constant, so does not depend on the choice of $x$
inside a connected component of $\fob\setminus\shA$. Reference
\cite{PassareRullgard} contains some more results on the complex Ronkin
function, notably a power series expansion in terms of the coefficients of
$f_\fob$. In general the information captured by the complex Ronkin function
does not seem to be well-understood.

Given a tropical cycle $\beta_\trop$ as before, we weight the complex Ronkin
function at a vertex $v$ of $\beta_\trop$ on a slab $\fob$ by $\langle \check
d_e,\xi_e\rangle\shR(z^{-m_v} f_\fob,v)$, notations as above. The sum of all
these contributions is denoted
\begin{equation}
\label{Eqn: pairing with Ronkin function}
\shR(\beta_\trop)\in \O(U)\lfor t\rfor,
\end{equation}
for $U\subset \Spec(A)_\an$ an open subset preserving the condition
$x\not\in\shA$ as discussed.

The complex Ronkin function $\shR(z^{-m_v} f_\fob,v)$ is trivial (constant $0$)
in one important situation.

\begin{proposition}
\label{Prop: Triviality of complex Ronkin function}
Assume that in a neighborhood of $\mu^{-1}(x)$ there is a convergent infinite product expansion
\[
z^{-m_x}f_\fob= \prod_{i=1}^\infty (1+a_i z^{m_i} t^{\ell_i})
\]
with $a_i\in A$ and all $m_i\neq0$. Then $\shR(z^{-m_x}f_\fob, x)= 0$.
\end{proposition}

\begin{proof}
By assumption we have a convergent Laurent expansion of $\log(z^{-m_x}f_\fob)$:
\[
\log\big(z^{-m_x}f_\fob\big)= \sum_i \log\big(1+a_i z^{m_i} t^{\ell_i}\big)
= \sum_i \sum_{j>0} \frac{(-1)^{j-1}}{j} a_i^j z^{jm_i}t^{j\ell_i}.
\]
The integral defining the complex Ronkin function can then be done term-wise.
Since $m_i\neq0$ for any $i$, the integral of $z^{jm_i}$ over the real torus
$\mu^{-1}(x)$ vanishes.
\end{proof}

For the slab functions appearing in the wall structures of \cite{affinecomplex},
the criterion of Proposition~\ref{Prop: Triviality of complex Ronkin function}
is fulfilled by the so-called \emph{normalization condition}, see \cite{affinecomplex}, \S3.6.
\medskip

\subsubsection{Period integrals}
Let us now assume that we have a polyhedral affine manifold $(B,\P)$, gluing
data $s=(s_{\sigma\ul\rho})$ and a wall structure on $B$ consistent in
codimension zero and one to order $k$, parametrized by a finitely generated
$\CC$-algebra $A$. We then obtain $X_k^\circ$, the flat deformation of
$X_0^\circ$ over $A_k=A[t]/(t^{k+1})$ and, for each point $a\in\Spec(A)_\an$,
the amoeba image $\shA=\mu(Z)\subset B$ of the log singular locus $Z$ in the
fiber $X_0^\circ(a)$ of $X_0^\circ$ over $a$. Let $\Omega$ be the canonical
relative holomorphic $n$-form on $X_k^\circ/A_k$ coming with the construction.

Here is the first main result of the paper, proved in \S\ref{Subsect: proof of
period thm}.

\begin{theorem}
\label{Thm: period thm}
Let $\beta_\trop\in Z_1(B\setminus(\Delta_2\cup\shA),\Lambda)$ be a tropical one-cycle and
$\beta$ an associated singular $n$-cycle on $X_0^\circ(a)$. Then using notations
introduced in \eqref{Eqn: pairing with c_1},\eqref{Eqn: pairing with gluing
data} and \eqref{Eqn: pairing with Ronkin function}, it holds
\begin{equation}
\label{Eqn: Main period formula}
\exp\bigg(\frac{1}{(2\pi\sqrt{-1})^{n-1}} \int_\beta\Omega \bigg)
= \exp\big(\shR(\beta_\trop)\big)\cdot\langle s,\beta_\trop\rangle \cdot t^{\langle c_1(\varphi),\beta_\trop\rangle}.
\end{equation}
According to Proposition~\ref{Prop: fiber integral well-defined} and
Proposition~\ref{holomorphic-periods-induce-formal-periods}, this result is
well-defined up to multiplication with $\exp(h\cdot t^{k+1})$ with $h\in \hat
A\lfor t\rfor$ for $\hat A$ the completion of $A$ at the maximal ideal
corresponding to $a$, and it agrees up to such changes with the corresponding
analytic integral for any flat analytic family $\shX\to U\times \DD$ over an
analytic open subset $U\times \DD\subset \Spec(A[t])_\an$ with reduction modulo $t^{k+1}$ equal to $X_k^\circ$. 
\end{theorem}

In the statement of Theorem~\ref{Thm: period thm}, the ambiguity of $\beta$ from
adding multiples of the vanishing cycle $\alpha$ disappears by exponentiation
since $\int_\alpha\Omega=(2\pi\sqrt{-1})^n$ (Lemma~\ref{lemma-intalpha}).

In practice one has a mutually compatible system of wall structures $\scrS_k$
consistent to increasing order $k$ and with an increasing, often unbounded
number of walls for $k\to\infty$. Theorem~\ref{Thm: period thm} then gives a
formula for the period integral as an element of the quotient field completion
$\hat A\lfor t\rfor$. In this formula only the complex Ronkin function
$\shR(\beta_\trop)$ potentially varies with $k$, capturing higher order
corrections to the slab functions $f_\fob$ as $k\to\infty$.

\begin{remark}
A straightforward generalization of Theorem~\ref{Thm: period thm} deals with
base spaces $A\lfor Q\rfor$ for $Q$ a toric monoid as in \cite{theta}. Then
$c_1(\varphi)\in H^1(B\setminus(\Delta_2\cup\shA),\check\Lambda\otimes Q^\gp)$ and $f_\fob\in
A[\Lambda_\rho][Q]$. Thus $\langle c_1(\varphi),\beta_\trop\rangle \in Q^\gp$
and the right-hand side of formula~\ref{Eqn: Main period formula} makes sense as
an element of $\hat A[Q^\gp]$ when writing the monomials of $\CC[Q^\gp]$ as
$t^q$ for $q\in Q^\gp$. This more general form follows easily from the stated
version by testing the statement on a dense set of $\Spf A\lfor Q\rfor$ by
base-changing via various morphisms $\Spf\big( \hat A\lfor t\rfor\big)\to
\Spf\big( \hat A\lfor Q\rfor\big)$.
\end{remark}

A particularly nice situation occurs when $B$ has simple singularities,
as introduced in \cite{logmirror1}, Definition~1.60. Morally, these are the
singularities that are indecomposable from the affine geometric point of view.
In dimension two, simple singularities lead to local models with slab functions
with at most one simple zero, that is $zw=(1+\lambda u)\cdot t^\kappa$ for some
$\lambda\in\CC$. In dimension three, the local models are $zw=(1+\mu u_1+\nu
u_2)\cdot t^\kappa$ or $xyz= (1+\lambda u)\cdot t^\kappa$ with
$\lambda,\mu,\nu\in\CC$. Then the algorithm of \cite{affinecomplex} produces a
canonical formal family $\foX\to \Spf \big(A\lfor t\rfor\big)$ with central
fiber $\Spec A$ classifying log Calabi-Yau spaces over the standard log point
with intersection complex $(B,\P)$, see \cite{theta}, Theorem~A7. Our second
main theorem says that locally this family is the completion of an analytic
family.

\begin{theorem}[Theorem~\ref{Thm: Analyticity of GS}]
\label{Thm: analyticity thm}
Assume that $(B,\P)$ has simple singularities, $B$ is orientable and $\partial
B$ again an affine manifold with singularities (e.g.\ empty). Then for each
closed point $a\in\Spec A$ there exists an analytic open neighborhood $U$ of
$a$ in $\Spec(A)_\an$, a disk $\DD\subset\CC$ and an analytic toric degeneration
\[
\shY\lra U\times \DD
\]
with completion at $(a,0)$ isomorphic, as a formal scheme over $A\lfor
t\rfor$, to the corresponding completion of the canonical toric degeneration $\foX\to \Spf
\big(A\lfor t\rfor\big)$ from \cite{theta}, Theorem~A.8.

Moreover, this completion is a hull for the logarithmic divisorial log
deformation functor defined in \cite{logmirror2}, Definition~2.7.
\end{theorem}

The proof occupies Section~\ref{Sect: analyticity}. The hard part of this
theorem is that it is not just an approximation result: the isomorphism of the
two formal families does not require any change of parameters. Thus the
canonical toric degenerations from \cite{affinecomplex} really are just an
algebraic order by order description of an analytic log-versal family with
monomial period integrals, that is, written in canonical coordinates.
\medskip

\noindent\emph{History of the results.}
The tropical construction of $n$-cycles and the main period computation in this
paper, for the case of \cite{affinecomplex} with trivial gluing data, has been
sketched by the second author in a talk at the conference ``Symplectic Geometry
and Physics'' at ETH~Z\"urich, September~3--7, 2007. Details have been worked
out in a first version of the paper in 2014 \cite{Version 1}. The present paper
is an essentially complete rewriting of that version, carefully treating period
integrals with logarithmic poles in finite order deformations, giving an
intrinsic formulation of all terms in the main period theorem (Theorem~\ref{Thm:
period thm}), including a treatment of non-normalized slab functions via the
Ronkin function and giving a proof of analyticity and versality of canonical toric
degenerations (Theorem~\ref{Thm: Analyticity of GS} and \S\ref{Par: Log versality}).


\subsection{Applications} \label{sec-applications}
We apply Theorem~\ref{Thm: period thm} in several interesting examples.
\begin{figure}
\captionsetup{width=1\linewidth}
\begin{center}
\includegraphics[width=.52\linewidth]{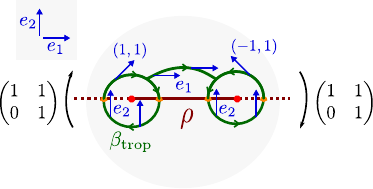}\qquad
\includegraphics[width=.38\linewidth]{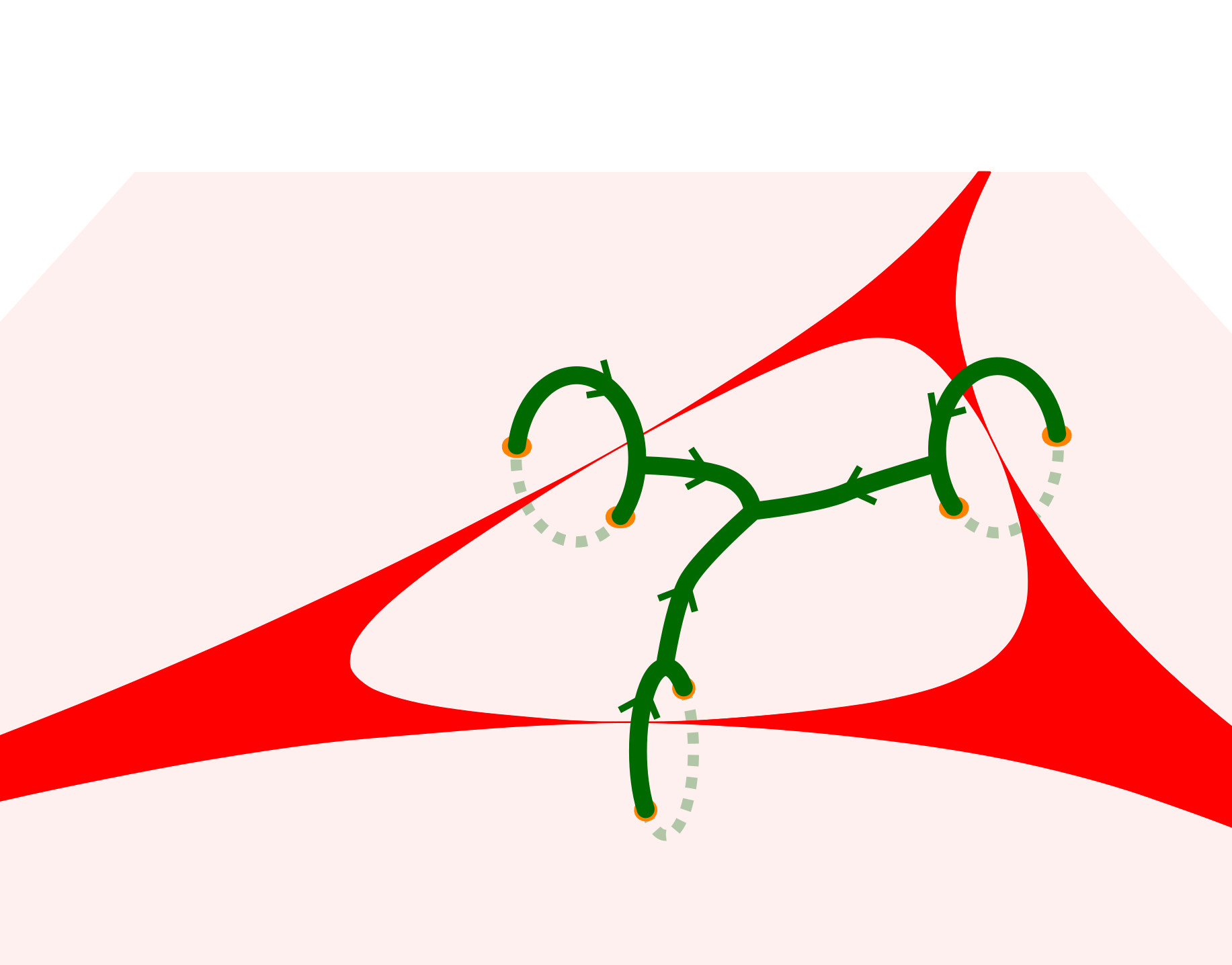} 
\end{center}
\caption{\small A one-dimensional slab $\rho$ with two focus-focus singularities
(red) on the left and a two-dimensional slab with amoeba image in red on
the right. Tropical one-cycles $\beta_\trop$ are given respectively in green and
their meeting points with the slabs in orange. Blue arrows on the left give the
sections of $\Lambda$ carried by edges of $\beta_\trop$.}
\label{fig-ex}
\end{figure}
\medskip

\subsubsection{Mirror dual of $K_{\PP^1}$}
\label{sec-localP1}
We consider \eqref{Eqn: Basic equation} alias \eqref{Eqn: ring in codim one} for
$f_\fob\in\CC[u,u^{-1}]$ featuring two zeros as follows
\begin{equation} \label{KP1-mirror}
zw=(au^{-1}+1)(1+bu)t^\kappa
\end{equation}
with $\kappa>0$ and $a,b\in\CC^\times$, $a\neq b$. The corresponding affine manifold $B$ is
shown on the left in Figure~\ref{fig-ex}, c.f. Figure 2.2 in \cite{ICM}. The two
focus-focus singularities are the images of the zeros of $f_\fob$ under the
momentum map. Figure~\ref{fig-ex} also shows a tropical one-cycle $\beta_\trop$,
and Theorem~\ref{Thm: period thm} yields
\begin{equation} \label{eq-mirror-KP1}
\exp\Big(\frac1{2\pi \sqrt{-1}}\int_\beta\Omega\Big) =a\cdot b,
\end{equation}
for $\beta$ the associated singular $2$-cycle. Indeed, one checks that the
contributions in $t$ arising from the two orange points on the same green circle
cancel. There are also no contributions from the trivial gluing data. By the
product expansion \eqref{KP1-mirror} and Proposition~\ref{Prop: Triviality of complex
Ronkin function}, the Ronkin term vanishes at the two inner crossing points
where this expansion is valid. However, for the two outer crossings, $f_\fob$ is
to be multiplied by $z^{-m_x}$, which is $u$ and $u^{-1}$, respectively. These
crossings produce the constant factors $a$ and $b$ in \eqref{eq-mirror-KP1}.

The geometry here arises as the mirror dual of $K_{\PP^1}$, a
smoothing of the $A_1$-singularity. Indeed, $a=b=1$ yields an $A_1$-singularity at
$x=y=u+1=0$. Our period computes the integral over the vanishing $2$-sphere. 
\medskip

\subsubsection{Mirror dual of $K_{\PP^2}$}
\label{sec-localP2}
We go up one dimension and consider a particular $f_\fob\in\CC[x^{\pm 1},y^{\pm
1}]$ whose zero-set is an elliptic curve so that \eqref{Eqn: Basic equation}
gives
\begin{equation} \label{eq-locP2}
zw=(1+x+y+s(xy)^{-1})t^\kappa
\end{equation}
for $s\in\CC^\times$ a parameter. This geometry arises as the mirror dual of
$K_{\PP^2}$, as studied before in \cite{CKYZ}, \S2.2 and from a toric
degeneration point of view in \cite{invitation}, Example~5.2 and \cite{ICM},
Figure~2.1. The corresponding real affine manifold is shown on the right of
Figure~\ref{fig-ex} with the amoeba of the elliptic curve the solid red area.
The amoeba complement in the slab has one bounded and three unbounded
components. The monomial $z^{-m_p}$ at a point $p$ inside one of the components
equals $1,x^{-1},y^{-1},xy$, respectively. To compute the contribution of the
Ronkin function, we write $z^{-m_p}\cdot f_\fob$ as an infinite product as
discussed in \cite{ICM}. Define integers $a_{ijk}$ by the identity
\begin{equation} \label{eq-factor-slab}
1+x+y+z=\prod_{i,j,k=0}^\infty (1+a_{ijk} x^iy^jz^{k})
\end{equation}
and then $h:=\prod_{k=1}^\infty (1+a_{kkk}s^k)$. Inserting $z=s(xy)^{-1}$ in
\eqref{eq-factor-slab} now yields a factorization of the slab function $f_\fob$
in \eqref{eq-locP2} as the product of $h$ and terms fulfilling the
hypothesis of Proposition~\ref{Prop: Triviality of complex Ronkin function}.
Thus the Ronkin term of the period integral at each of the three crossings of
the tropical cycle in the bounded center region of the amoeba complement equals
$\log(h)$. The Ronkin terms for crossings of the unbounded regions vanish
readily by Proposition~\ref{Prop: Triviality of complex Ronkin function}, except
for the constant term of $xy\cdot f_\fob$, which yields $\log(s)$.

Thus, by Theorem~\ref{Thm: period thm}, the exponentiated period integral for the
tropical cycle $\beta_\trop$ depicted in green in Figure~\ref{fig-ex} yields
\begin{equation} \label{eq-period-localP2}
\exp\Big(\frac1{(2\pi \sqrt{-1})^2}\int_\beta\Omega\Big) = h^3\cdot s.
\end{equation}

The smoothing algorithm of \cite{affinecomplex} replaces $f_\fob$ by $f_\fob+g$,
where $g=-2 s+5s^2-32s^3+\ldots$ is determined by the \emph{normalization condition}
saying that $\log(f_\fob+g)$ has no pure $s$-powers
(\cite{affinecomplex},\,\S3.6). See also \cite{GZ}, p.14 and \cite{ICM}, Example
3.1,(2). With the normalized slab function $f_\fob+g$, the factor $h^3$ disappears,
leaving only $s$ for the exponentiated period integral
\eqref{eq-period-localP2}. This illustrates the mechanism relating the
normalization condition and the fact that the exponentials of periods for the
canonical toric degenerations of \cite{affinecomplex} are monomials in the base
of the family, see also \S\ref{Par: Canonical coordinates} below. For a related
enumerative interpretation of the normalization condition see \cite{CCLT16},
Theorem~1.6.
\medskip

\subsubsection{Degenerations of K3 surfaces with prescribed Picard group} \label{example-K3}
For a K3 surface $Y$ with holomorphic volume form $\Omega$, an integral homology
class $\beta\in H_2(Y,\ZZ)$ is the first Chern class of a holomorphic line
bundle if and only if its Poincar\'e-dual class $\hat\beta\in H^2(Y,\CC)$ is of
type $(1,1)$. Since $\hat\beta$ is real, this condition can be detected by the
vanishing of $\int_\beta\Omega$:
\begin{equation}
\label{Eqn: (1,1)-condition}
\int_\beta \Omega= \int_\beta\overline\Omega=\int_Y \hat\beta\wedge\Omega \stackrel{!}{=}0.
\end{equation}
Combining this observation with Theorem~\ref{Thm: period thm} and
Theorem~\ref{Thm: analyticity thm}, we obtain the following computation of the
Picard lattice of general fibers of the canonical toric degenerations of K3
surfaces constructed in \cite{affinecomplex}. In this case $B$ is a $2$-sphere
and the amoeba locus $\shA$ consists of at most $24$ points. This number is
achieved if all singularities are of focus-focus type, that is, have local
affine monodromy conjugate to $\left(\begin{smallmatrix}1&
0\\1&1\end{smallmatrix}\right)$. This is the case iff all slab
functions have simple zeroes with pairwise different absolute values, for
example if $(B,\P)$ has simple singularities as explained before
Theorem~\ref{Thm: analyticity thm}. In any case, we assume that the affine
structure on $B$ extends over the vertices of $\P$, so we can disregard
$\Delta_2$. The result is then expressed in terms of the singular homology
group $H_1(B,\iota_*\Lambda)$ with $\iota:B\setminus \shA\to B$ the inclusion.
Before stating the result, we make some comments on this homology group and how
it relates to the K3 lattice.

First, if a singular $1$-cycle with coefficients in $\iota_*\Lambda$ passes
through a point of $\shA$, then the integral tangent vector carried at this
point is invariant under local integral affine monodromy and hence can be
perturbed away from $\shA$. In other words, push-forward by $\iota$ defines a
surjection
\begin{equation}
\label{Eqn: tropical cycles passing through shA}
\iota_*:H_1(B\setminus(\Delta_2\cup\shA),\Lambda)\lra H_1(B,\iota_*\Lambda).
\end{equation}
Second, for $\shY\ra U\times \DD$ the analytic family from Theorem~\ref{Thm:
analyticity thm} and $a\in U\subset \Spec(A)_\an$, $t\in \DD\setminus\{0\}$, let
$Y_t=Y_t(a)$ denote the fiber over $(a,t)$. Our construction of tropical cycles
defines a homomorphism
\[
H_1(B\setminus\shA,\Lambda)\lra H_2(Y_t,\ZZ)/\langle\alpha\rangle,
\quad \beta_\trop\longmapsto \beta,
\]
with $\alpha$ the vanishing cycle. This homomorphism is compatible with the respective
intersection pairings (cap product), as is the previous homomorphism \eqref{Eqn:
tropical cycles passing through shA}. Third, $H_1(B\setminus\shA,\Lambda)$
together with its intersection pairing only depends on the linear part of the
monodromy representation, hence can be computed in any model, by the classical
uniqueness result for this monodromy representation (\cite{LivneMoishezon},
Theorem on p.\,225). For one particular model, Symington in \cite{Symington},
\S11, has given a basis of tropical cycles\footnote{The cycles in
\cite{Symington} use a different construction for the cycles, but it is clear
how to obtain cycles homologous to hers in our fashion.} spanning the even
unimodular lattice $-E_8^{\oplus 2}\oplus H^{\oplus 2}$ of rank $20$ and
signature $(2,18)$, and these also map to a basis of $H_1(B,\iota_*\Lambda)$
under \eqref{Eqn: tropical cycles passing through shA} by unimodularity and a
rank computation. This lattice is the orthogonal complement of a hyperbolic
plane $H$ in the K3 lattice spanned by a fiber and a section of a K3 surface
fibered in Lagrangian tori over $B$. In our situation this fiber class is
$\alpha_t$ and we have identifications of lattices
\[
H_1(B,\iota_*\Lambda)=
\big\{\beta_t\in H_2(Y_t,\ZZ)\,\big|\, \beta_t\in\alpha_t^\perp\big\}\big/\langle \alpha\rangle
\simeq -E_8^{\oplus 2}\oplus H^{\oplus 2}.
\]
Our period integrals now identify the Picard lattice of $H_2(Y_t,\ZZ)$ inside this lattice.

\begin{corollary}
\label{Cor: K3 with Pic group}
Let $\pi:\shY\to U\times \DD$ be the analytic version from Theorem~\ref{Thm:
analyticity thm} of the canonical degenerating family of K3~surfaces defined by
a polyhedral affine structure $(B,\P)$ with underlying topological space $S^2$
and simple singularities, strictly convex multivalued piecewise affine function
$\varphi$ and gluing data $s\in H^1(B,\iota_*\check\Lambda\otimes\CC^*)$. Then the 
Picard group of a general fiber $Y_t$ of $\pi$ is canonically isomorphic to
\[
\big\{ \beta_\trop\in H_1(B,\iota_*\Lambda)\,\big|\, 
 \beta_\trop\in c_1(\varphi)^\perp,\, \langle s,\beta_\trop\rangle=1\big\}.
\]
\end{corollary}

\begin{proof}
Theorem~\ref{Thm: period thm} implies that $\int_{\beta_t}\Omega_{Y_t}$ can only
be constant if $\langle c_1(\varphi),\beta_\trop\rangle = 0$. If this is the
case then $\int_{\beta_t}\Omega_{Y_t}$ extends holomorphically over $t=0$ and
thus $\langle s,\beta_\trop\rangle=1$ is equivalent to
$\int_{\beta_t}\Omega_{Y_t}\in (2\pi\sqrt{-1})^n\ZZ$. Noting that
$\int_{\alpha_t}\Omega_{Y_t}=(2\pi\sqrt{-1})^n$ for $\alpha_t\in H_2(Y_t,\ZZ)$
the class of the vanishing cycle from Proposition~\ref{Prop: Picard-Lefschetz
transformation}, the equality in \eqref{Eqn: tropical cycles passing through
shA} implies that $\beta_t$ is the image of the Poincar\'e-dual of an integral
$(1,1)$ class under the quotient map $H_2(Y_t,\ZZ)\to H_2(Y_t,\ZZ)/\langle
\alpha_t\rangle$. Since $\alpha_t$ can be chosen Lagrangian, it can not be
Poincar\'e-dual to the class of a holomorphic line bundle. Hence the image of
$\beta_t$ in $H_2(Y_t,\ZZ)/\langle\alpha_t\rangle$ is enough to determine the
Picard lattice.
\end{proof}

Thus for trivial gluing data $s=1$, or $s$ of finite order in
$H^1(B,\iota_*\check\Lambda\otimes\CC^*)$, the Picard lattice of $Y_t$ has next
to largest rank $19$. We thus retrieve families studied intensely, see e.g.\
\cite{Mo84}, \cite{dolgachev}.

It is also possible to treat the more general families with non-simple
singularities from \cite{theta}, \S{A.4}, by treatment as a limit of a situation
with focus-focus singularities. Such models lead to one-parameter families with
$A_k$-singularities. For trivial gluing data, their resolution still provides
families of K3 surfaces with Picard rank~$19$. Non-simple singularities are
necessary for producing families of K3 surfaces with large Picard rank of low
degree. Further details will appear in \cite{GHKS}.

Related results from a more elementary perspective have been obtained in
\cite{Yamamoto}.
\medskip

\subsubsection{Degenerations of rational elliptic surfaces}
Another application is to toric degenerations of rational elliptic surfaces. In
this case, the formula for period integrals in Theorem~\ref{Thm: period thm} has
been used in the thesis of Lisa Bauer to prove a Torelli theorem for toric
degenerations of rational elliptic surfaces with simple singularities. We refer
to \cite{Bauer},\S5 for details.
\medskip

\subsubsection{Canonical Coordinates in Mirror Symmetry}
\label{Par: Canonical coordinates}
A canonical system of holomorphic coordinates on the base $V$ of a maximally
unipotent Calabi-Yau degeneration $\shY\ra V$ was first proposed in
\cite{Mo93}, \cite{De93} as follows. Let $\shY_0$ denote the maximally degenerate fiber,
$V\subset\CC^r$ a small neighborhood of $0$ and $\shY_t$ a regular fiber.
Assume the discriminant $D=D_1+\ldots+D_s\subset V$ is the intersection with $V$ of
a union of coordinate hyperplanes and let $T_i:H_n(\shY_t,\ZZ)\ra
H_n(\shY_t,\ZZ)$ be the monodromy transformation along a simple loop around
$D_i$. If $Y_0$ is reduced with simple normal crossings then the $T_i$ are
unipotent, so $N_i=-\log(T_i)$ is well-defined. Set $N=\sum_i \lambda_i N_i$ for
any $\lambda_i>0$.

The \emph{monodromy weight filtration} is the unique filtration $W_0\subseteq
W_1\subseteq W_2\subseteq \ldots\subseteq W_{2n}$ on $H_n(\shY_t,\QQ)$ with the
properties $N(W_i)\subseteq W_{i-2}$ and that $N^k:\Gr^W_{n+k}\ra \Gr^W_{n-k}$
is an isomorphism for $\Gr^W_{i}=W_{i}/W_{i-1}$. Schmid gave a decreasing
filtration $F^\bullet_{\lim}$ on $H^n(\shY_t,\CC)$ which combines with the
Poincar\'e dual filtration $\tilde W_i:=W^\perp_{2n-i}$ to give a mixed Hodge
structure. The degeneration $\shY\ra V$ is \emph{maximally unipotent} if this
mixed Hodge structure is \emph{Hodge-Tate}. The latter property implies that
$W_{2i}=W_{2i+1}$. If $\shY_t$ is Calabi-Yau, then $\dim_\QQ W_0=1$ and
$\dim_\QQ W_2/W_0=\dim H^1(\shY_t,\Theta_{\shY_t})$, so at least dimension-wise
it makes sense to expect that a set of cycles $\beta_1,\ldots,\beta_r\in
H_n(\shY_t,\ZZ)$ that descends to a basis of $W_2/W_0$ gives rise to a set of
coordinates $h_i:=\exp(\int_{\beta_i}\Omega)$ for $\Omega$ a suitably normalized
relative $n$-form of $\shY\ra V$. This was proved in \cite{Mo93}, \cite{De93}. These
coordinates are \emph{canonical} in the sense of being unique up to an
\emph{integral} change of basis of $W_2/W_0$. The $h_i$ also agree with
exponentials of flat coordinates for the special geometry on the Calabi-Yau
moduli space defined by the Weil-Petersson metric
\cite{Tian},\cite{Strominger},\cite{Freed}.

Motivated from \cite{SYZ}, the Leray filtration of the momentum map was found to
coincide with the above weight filtration for $n=3$ \cite{Gr98}\,\S4, so
generators for $W_2/W_0$ should be obtained from one-cycles in $B$ with values
in $\Lambda$, as also suggested in~\cite{KS06}\,\S7.4.1. Note however, this can
only work if $H_1(B,\iota_*\Lambda)$ has large enough rank and it is easy to
produce examples where this fails. A way to ensure the rank matches is by
requiring $B$ to be simple, see \cite{logmirror2}, \cite{affinecoh}. In the
simple situation, Corollary~\ref{Cor: monomial analytic periods} and \S\ref{Par:
Log versality} give directly that the exponentiated periods from cycles in
$H^1(B,\iota_*\Lambda)$ provide coordinates on a versal family. At least in the
case with simple singularities, it is expected that the image of the
homomorphism
\begin{equation}
\label{eq-trop-to-usual}
H_1(B\setminus(\Delta_2\cup\shA),\Lambda)\ra H_n(Y_t,\QQ)/\langle\alpha_t\rangle
\end{equation}
generates $W_2$. In general, $W_0=\im N^n$ and
\[
W_2=\big(\im N^{n-2}\cap \ker N\big)+\big(\im N^{n-1}\cap \ker N^2\big).
\]
By Proposition~\ref{Prop: Picard-Lefschetz transformation}, $\alpha_t\in\ker(N)$
and $\beta\in\ker(N^2)$ for every $\beta$ obtained from a tropical one-cycle. By
\cite{logmirror2}, $N$ can be identified with the Lefschetz operator on the
mirror. Thus, by the rotation of the Hodge diamond \cite{logmirror2} and up to
identifying the composition $Y_t\ra Y_0\stackrel\mu\ra B$ with a
compactification of $T^*_{B\setminus(\Delta_2\cup\shA)}/\check\Lambda\ra
B\setminus(\Delta_2\cup\shA)$ in the upcoming work \cite{RZ}, we find
$\alpha_t\in\im N^n$ and the image of \eqref{eq-trop-to-usual} indeed generates
$W_2$. The geometry over $B\setminus(\Delta_2\cup\shA)$ has been investigated
extensively in \cite{KNreal}.
\medskip

\subsection{Relation to other works}
Beyond algebraic curves, explicit computations of period integrals we found to
be quite rare in the literature. In higher dimensions, residue calculations can
sometimes be used to compute certain periods by the Griffiths-Dwork method of
reduction of pole order \cite{Dwork}, \cite{Griffiths}. Equation~(3.7) in
\cite{CdGP91} gives a famous example of such a computation in the context of
mirror symmetry. More recently the same type of period calculation became the
main protagonist in a prominent conjecture for the classification of Fano
manifolds \cite{CCGGK}. Other period integrals are often determined indirectly
as solutions of differential equations coming from the flatness of the
Gauss-Manin connection, usually at the expense of losing the connection to
topology, that is, to the integral structure. Even more recently, \cite{AGIS}
computed periods of a section of the SYZ fibration to small $t$-order in a
maximal degeneration as considered also in this article. Also worth mentioning
are the explicit computation of periods for local Calabi-Yau manifolds,
Proposition~3.5 in \cite{DK} and the numerical approximation of period integrals
over polyhedral cells carried out in \S2 in \cite{CS}. A particular local
situation similar to the example in \S\ref{sec-localP1} has been computed
independently by Sean Keel (unpublished).
\bigskip

\paragraph{\emph{Acknowledgments}}
It should be obvious that this papers would not have been possible without the
long term collaboration of the second author with Mark Gross on toric
degenerations and their use in mirror symmetry. The final shape of the result
owes a lot to the collaboration of the second author with M.~Gross, P.~Hacking
and S.~Keel on generalized theta functions and concerning compactified moduli
spaces of K3 surfaces \cite{theta},\cite{GHKS}. We also benefited from
discussions with H.~Arg\"uz, L.~Bauer, D.~Matessi, D.~van Straten and E.~Zaslow
on various aspects of this article.

%
%

\setcounter{tocdepth}{2}
\section{From tropical cycles to singular chains}
\label{section-tropical-to-homology}
Throughout let $B$ be an oriented tropical manifold possibly with boundary and
with polyhedral decomposition $\P$, convex MPL-function $\varphi$, open gluing
data $s$ and a consistent order $k$ structure for this data, as explained in
\S\ref{Par: Toric degenerations}. Let $X^\circ_k$ be the scheme over
$\CC[t]/(t^{k+1})$ obtained from this data by gluing the standard pieces $\Spec
R_\fob^k$ and $\Spec R_\fou$. Our computation of the period integrals is
entirely on $X^\circ_k$. For this computation in
Sections~\ref{section-tropical-to-homology} and \ref{section-compute-periods} we
therefore do not impose the additional consistency requirements needed to assume
the existence of the partial compactification $X_k$ or even the existence of
$X_0$, nor do we need an extension of $X^\circ_k$ to an analytic family. An
exception is the discussion of the degenerate momentum map $X_0\to B$, which is
of independent interest.

For simplicity of presentation we work here with fixed gluing data $s$, that is,
with $A=\CC$ as base ring in \S\ref{Par: Toric degenerations}. General $A$ can
easily be treated by either introducing analytic parameters in all formulas or,
for reduced $A$, by verifying the claimed period formula \eqref{Eqn: Main period
formula} on a dense set of gluing data.


\subsection{A generalized momentum map for \texorpdfstring{$X_0$}{X0}}

For trivial gluing data, $X_0$ exists as a projective variety with
irreducible components the toric varieties $X_\sigma$ with momentum polyhedra
the maximal cells $\sigma$ of $\P$. If $\sigma,\sigma'$ intersect in a
codimension one cell $\rho$ of $\P$ then the $(n-1)$-dimensional toric variety
$X_\rho$ is a joint toric prime divisor of $X_\sigma,X_{\sigma'}$, with the
identification toric, that is, mapping the distinguished points in the big cells
to one another. It is then not hard to see that the momentum maps $\mu_\sigma:
X_\sigma\to \sigma$ patch to define a generalized momentum map $\mu: X_0\to
B=\bigcup_{\sigma\in\P}\sigma$.

For general gluing data, the momentum maps $\mu_\sigma$ may not agree on joint
toric strata and it is not clear that $\mu$ exists. Assuming projectivity, we
present here a canonical construction of $\mu$ and otherwise prove the existence
of $\mu$ away from codimension two strata.

\begin{proposition}
\label{Prop: momentum map}
Assume that $X_0$ is projective.  Then there is a continuous map
\[
\mu: X_0\lra B
\]
which on each irreducible component $X_\sigma\subset X_0$ restricts
to a momentum map for the toric $U(1)^n$-action and some
$U(1)^n$-invariant K\"ahler form on $X_\sigma$.

Without the projectivity assumption, $\mu$ can be constructed on the complement
$X_0^\circ$ of the codimension two toric strata.
\end{proposition}

\begin{proof}
In the projective case, the central fiber $X_0$ can be constructed as $\Proj S$
with $S$ a graded $\CC$-algebra generated by one rational function $\vartheta_m$
on $X_0^\circ$ for each integral point $m$ of $B$, see \cite{theta},
\S5.2 (where $S$ is denoted $S[B](\tilde{ \bar{\mathbf{s}}})$). If $m$ lies in a
maximal cell $\sigma$ then $\vartheta_m$ restricts to a non-zero \emph{multiple} of the
monomial $z^m$ on $X_\sigma$ defined by toric geometry. Define the K\"ahler form
$\omega$ on $X_0$ as the pull-back of the Fubini-Study form $\omega_{\rm FS}$ on
projective space under the embedding $\Phi: X_0\to \PP^N$ defined by the
$\vartheta_m$, with $N+1$ the number of integral points on $B$. Denote by
$\mathfrak{t}$ the Lie algebra of the $n$-torus $T_\sigma$ acting on
$X_\sigma$ and write $\sigma$ as a momentum polytope in $\mathfrak{t}^*$. Now
define $\mu$ on $X_\sigma$ by
\begin{equation}\label{Eqn: mu_sigma}
\mu_\sigma: X_\sigma\lra\sigma\subset \mathfrak{t}^*,\quad
\mu_\sigma(z)= \frac{\sum_{m\in\sigma\cap\Lambda_\sigma}
|\vartheta_m(z)|^2\cdot m}{\sum_{m\in\sigma\cap\Lambda_\sigma}
|\vartheta_m(z)|^2}.
\end{equation}
We claim that $\mu_\sigma$ is a momentum map for the $U(1)^n$-action on
$X_\sigma$ with respect to $\omega|_{X_\sigma}$. Indeed, denote by
$\Delta_N$ the $N$-simplex with one vertex $v_m$ for each integral point $m\in
B(\ZZ)$. We view $\Delta_N$ as an integral polytope in $\mathfrak{\tilde t}^*$
with $\mathfrak{\tilde t}$ the Lie algebra of the torus $U(1)^{N+1}$ acting
diagonally on $\PP^N$. Let $\mu:\PP^N\to \mathfrak{\tilde t}^*$ be the usual
momentum map defined by a formula analogous to \eqref{Eqn: mu_sigma} with
$\vartheta_m$ replaced by the monomials of degree~$1$. Then there is an integral
affine map $\Delta_N\to \mathfrak{t}^*$, which for $m\in\sigma$ maps the
vertex $v_m$ to $m$ and the other vertices to arbitrary integral points. The
induced map $\mathfrak{\tilde t}^*\to\mathfrak{t}^*$ defines a morphism
of tori $\kappa:T_\sigma\to U(1)^{N+1}$ for which the composition $X_\sigma\to
X_0\stackrel{\Phi}{\to}\PP^N$ is equivariant.

Now $\mu_\sigma$ factors over the momentum map $\mu$ for $\PP^N$ as follows:
\[
\mu_\sigma: X_\sigma\lra X_0\stackrel{\Phi}{\lra} \PP^N
\stackrel{\mu}{\lra} \mathfrak{\tilde t}^*\stackrel{\kappa^*}{\lra}
\mathfrak{t}^*.
\]
Thus if $\xi\in\mathfrak{t}$ and $\tilde\xi$
is the induced vector field on $X_\sigma$, we can check the momentum
map property for $\mu_\sigma$ as follows:
\[
d(\xi\circ\mu_\sigma) =d(\xi\circ\kappa^*\circ\mu\circ\Phi)
= \Phi^*d(\kappa_*(\xi)\circ\mu)
= \Phi^*(\iota_{\Phi_*\tilde\xi}\omega_{\rm FS})
=\iota_{\tilde\xi}\omega.
\]

If $X_0$ is not projective, the complement $X_0^\circ$ of the codimension two
locus is the fibered sum of its irreducible components, with a toric divisor
$X^\circ_\rho$ contained in two components $X^\circ_\sigma$, $X^\circ_{\sigma'}$
identified via a toric automorphism, that is, by multiplication with an element
$g$ of the $(n-1)$-torus $T_\rho$ acting on $X_\rho$.\footnote{This fibered sum
is the description of $X_0^\circ$ in terms of closed gluing data discussed in
\cite{theta}, \S5.1. This description may not extend over the codimension two
locus.} Let $\mu_\sigma: X_\sigma\to\sigma$ and $\mu_{\sigma'}: X_{\sigma'}\to
\sigma'$ be the standard toric momentum maps. Since $\sigma\cap\sigma'=\rho$,
the restrictions of $\mu_\sigma$, $\mu_{\sigma'}$ to $X_\rho$, viewed as a toric
divisor in $X_\sigma$ and $X_{\sigma'}$, agree with the standard momentum map
$\mu_\rho: X_\rho\to\rho$. By equivariance of $\mu_\rho$ with respect to the
torus action there exists a diffeomorphism $\Phi:\rho\to\rho$ such that
\[
\mu_\rho(g\cdot z)= \Phi\big(\mu_\rho(z)\big)
\]
holds for any $z\in X_\rho$. Use $\Phi$ to change the identification of $\rho$
as a facet of $\sigma'$, but leave the embedding $\rho\to\sigma$ unchanged.
Repeating this construction for all $\rho$ leads to a directed system of all
polyhedra $\rho,\sigma\in \P$ of dimensions $n-1$ and $n$. After removing all
faces of dimension strictly less than $n-1$, a colimit of this directed system
in the category of topological spaces exists and is a topological manifold. It
is also not hard to see that this colimit is homeomorphic and cell-wise
diffeomorphic to the complement $B\setminus\Delta_2\subset B$ of the
$(n-2)$-skeleton of $\P$. Since we have a compatible description of $X_0^\circ$
as a colimit, we obtain the desired momentum map $\mu: X_0^\circ\lra B\setminus\Delta_2$ that
on $X_\sigma$ is the composition of $\mu_\sigma$ with the restriction $\sigma\to
B$ of the cell-wise diffeomorphism.
\end{proof}

Note that by \cite{delzant}, Th\'eor\`eme~2.1, two momentum maps on
a toric variety are related by a homeomorphism that is a
diffeomorphism at smooth points. In particular, for non-trivial
gluing data our global momentum map restricts on any irreducible
component $X_\sigma\subset X_0$ to the standard toric momentum map
$X_\sigma\to\sigma$ composed with a homeomorphism of $X_\sigma$ that
is a diffeomorphism away from strata of codimension at least
two. Note also that in the K\"ahler setting, the Hamiltonian vector field
on $X_\sigma$ defined by a co-vector $\xi\in \Lambda_\sigma^*$ is
given by the action of the algebraic subtorus $\GG_m= \Spec\CC[\ZZ]
\subseteq \Spec\CC[\Lambda_\sigma]$ defined by
$\xi:\Lambda_\sigma\to\ZZ$.


\subsection{Canonical affine structure on \texorpdfstring{$B\setminus(\Delta_2\cup\shA)$}{B\D2uA}}

Let $\mu: X_0^\circ\to B\setminus\Delta_2$ be a generalized momentum map as
produced by Proposition~\ref{Prop: momentum map}. Denote by $Z\subset X_0^\circ$
the log singular locus, an algebraic subset of dimension $n-2$ defined by the
vanishing of the slab functions. Further denote by $\Delta_k\subset B$ the
$(n-k)$-skeleton of $\P$, that is, the union of all cells of $\P$ of dimensions
at most $n-k$. The amoeba image $\mathcal A:=\mu(Z)$ is contained in the
$(n-1)$-skeleton $\Delta_1\subset B$. For $\rho\in\P$ an $(n-1)$-cell, $\mathcal
A\cap\Int\rho$ is diffeomorphic to the classical amoeba in $\RR^{n-1}$ defined by any of the slab
functions $f_{\fob}$ for $\fob\subset\rho$, viewed as an element of the ring of
Laurent polynomials $\CC[\Lambda_\rho]$. For the following discussion only the
reduction $f_{\ul\rho}$ of $f_\fob$ modulo $t$ is relevant. The notation
$f_{\ul\rho}$ is justified because by consistency in codimension one the
reduction of $f_\fob$ modulo $t$ only depends on the cell $\ul\rho\subset\rho$ of the
barycentric subdivision containing $\fob$. The fiber of $\mu$ over a point $x\in
\rho\setminus\mathcal A$ is the torus fiber of $X_\rho\to \rho$ over $x$. We
will now show that there is a natural extension of the integral affine structure
on $B\setminus\Delta_1$, the union of the interiors of maximal cells, to
$B\setminus (\Delta_2\cup\mathcal A)$ as follows.

\begin{construction}
\label{Affine structure on B minus A}
(Construction of the affine structure on $B\setminus
(\Delta_2\cup\mathcal A)$.) On the interior of a maximal cell $\sigma\subseteq
B$ define the integral affine structure by the Arnold-Liouville theorem for the
restriction of our momentum map from Proposition~\ref{Prop: momentum map}. For
$\rho$ an $(n-1)$-cell with $\sigma,\sigma'$ the adjacent maximal cells and
$\fob\subseteq\ul\rho$ a slab, recall from \S\ref{subsec-make-Xk} that the
defining equation
\begin{equation}
\label{Eqn: Std eqn codim 1}
Z_+ Z_-= f_\fob t^{\kappa_{\ul\rho}}
\end{equation}
involves monomials $Z_+= c_+ z^\zeta$, $Z_-= c_- z^{\zeta'}$ on the toric
varieties $X_\sigma, X_{\sigma'}$. Here $\zeta,\zeta'$ generate the normal spaces
$\Lambda_\sigma/ \Lambda_\rho$ and $\Lambda_{\sigma'}/\Lambda_\rho$,
respectively, and parallel transport through any point in $\ul\rho\supseteq\fob$
in the affine structure on $B\setminus\Delta$ carries $\zeta$ to $-\zeta'$. The
constants $c_\pm\in\CC^*$ are determined by gluing data, namely
$c_+=s_{\sigma\ul\rho}(\zeta)$, $c_-= s_{\sigma'\ul\rho}(\zeta')$. This equation
depends on the choice of $\ul\rho\subset\rho$, a cell of the subdivision of
$\rho$ defined by $\Delta\cap\rho$, but any other choice just leads to a
multiplication of the equation with a monomial $c z^{m_{\ul\rho\ul\rho'}}$ with
$m_{\ul\rho\ul\rho'}\in\Lambda_\rho$ and $c\in\CC$.

We now use these local models to define an adapted affine structure on $B$
outside $\Delta_2\cup\shA$. Since the integral affine structures on
$\sigma,\sigma'$ already agree on $\rho$, for the definition of a chart at
$x\in\Int(\ul\rho) \setminus\mathcal A$ it remains to declare the parallel
transport of $\zeta$ through $\Int\ul\rho$ as $-\zeta'+m_x$ for some
$m_x\in\Lambda_\rho$. The restriction of the reduction $f_{\ul\rho}$ of $f_\fob$
modulo $t$ to $\mu^{-1}(x)=\Hom(\Lambda_\rho, U(1)) \,\simeq\, (S^1)^{n-1}$
defines a map
\[
\Hom(\Lambda_\rho, U(1)) \lra \CC^*\stackrel{\arg}{\lra} U(1).
\]
The image of the first arrow lies in $\CC^*$ because
$f_{\ul\rho}|_{\mu^{-1}(x)}$ has no zeroes for $x\not\in\mathcal A$. The
positive generator of $H^1(U(1),\ZZ)\simeq\ZZ$ pulls back to the desired element
\begin{equation}
\label{Eqn: m_x}
m_x\in \Lambda_\rho = H^1\big( \Hom(\Lambda_\rho, U(1)),\ZZ\big).
\end{equation}
It is worthwhile noticing that $m_x$ agrees with the \emph{order}
of the amoeba complement selected by $x$, as defined in \cite{Forsbergetal},
Definition~2.1. In particular, $m_x$ is locally constant on
$\Int\rho\setminus\mathcal A$.
\end{construction}

\begin{remark}
\label{Rem: tilde Z_+ tilde Z_-}
With the definition of the affine structure in Construction~\ref{Affine
structure on B minus A} we are now in a position to rewrite the local equation
\eqref{Eqn: Std eqn codim 1} in a form suitable for the local construction of
$n$-cycles from tropical curves. For $x\in \Int(\ul\rho)\setminus\mathcal A$ let
$\tilde\zeta\in\Lambda_\sigma$ be any tangent vector generating
$\Lambda_\sigma/\Lambda_\rho$ and pointing from $\rho$ into $\sigma$. Then
$\tilde\zeta=\zeta+m$ for some $m\in\Lambda_\rho$. Thus defining
\[
\tilde Z_+= z^m\cdot Z_+,\quad \tilde Z_-= z^{-m-m_x} Z_-,
\]
Equation~\eqref{Eqn: Std eqn codim 1} can also be written as
\begin{equation}
\label{Eqn: Adapted equation codim 1}
\tilde Z_+ \tilde Z_-= (z^{-m_x}f_\fob) t^{\kappa_{\ul\rho}}.
\end{equation}
The point is that by the definition of $m_x$ in \eqref{Eqn: m_x}, this equation
differs from a standard normal crossings equation $zw=t^{\kappa_{\ul\rho}}$ by
the factor $z^{-m_x}f_\fob$. This factor is homotopically trivial as a map from
$\mu^{-1}(x)$ to $\CC^*$. This is a crucial property in the construction of an
$n$-cycle in Lemma~\ref{Lem: charts for cycle} below fulfilling the condition
(Cy~II) needed in our treatment of finite order period integrals in
Appendix~\ref{Sect: Period integrals}.

We emphasize that while these conventions look technical, our formula
\eqref{Eqn: Main period formula} for the period integral involves the Ronkin
function associated to $z^{-m_x} f_\fob$ and hence is sensitive to the
definition of $m_x$. See \S\ref{sec-localP1} and \S\ref{sec-localP2} for two
simple examples.
\end{remark}

\begin{remark}
\label{Rem: Comparison of shA and Delta} 
We defined an affine structure on $B\setminus(\Delta_2\cup\shA)$ in Construction~\ref{Affine
structure on B minus A}. On the other hand, \cite{theta} works with an affine
structure on $B\setminus\Delta$ for the formulation of the wall structure. These
two affine structures are related in the following way. Recall that
$\Delta\subset B$ is the $(n-2)$-skeleton of the barycentric subdivision of the
$(n-1)$-skeleton of $\P$.
\begin{figure}
\captionsetup{width=.8\linewidth}
\begin{center}
\includegraphics[width=0.27\textwidth]{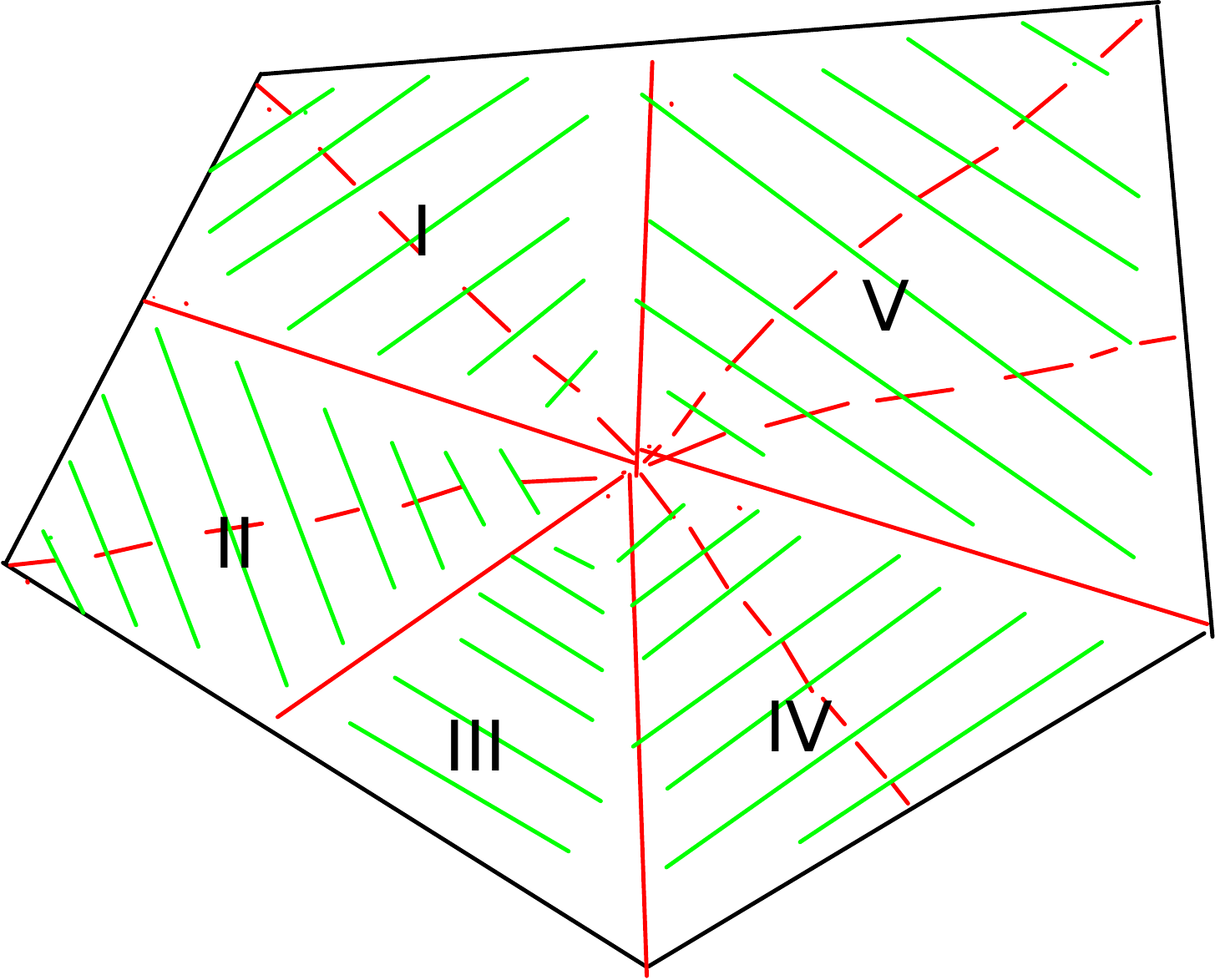}\hspace{5ex}
\includegraphics[width=0.27\textwidth]{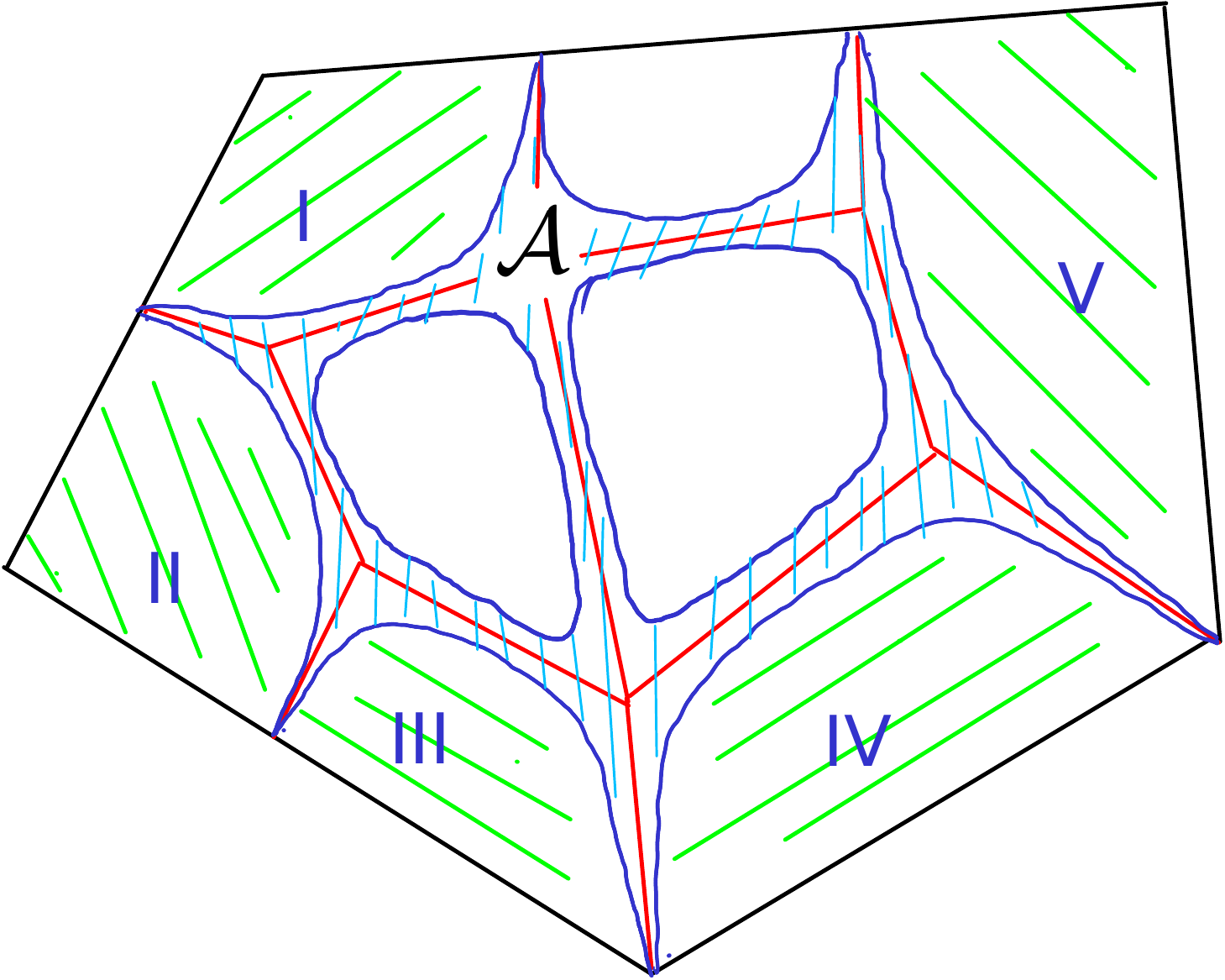}\hspace{5ex}
\includegraphics[width=0.27\textwidth]{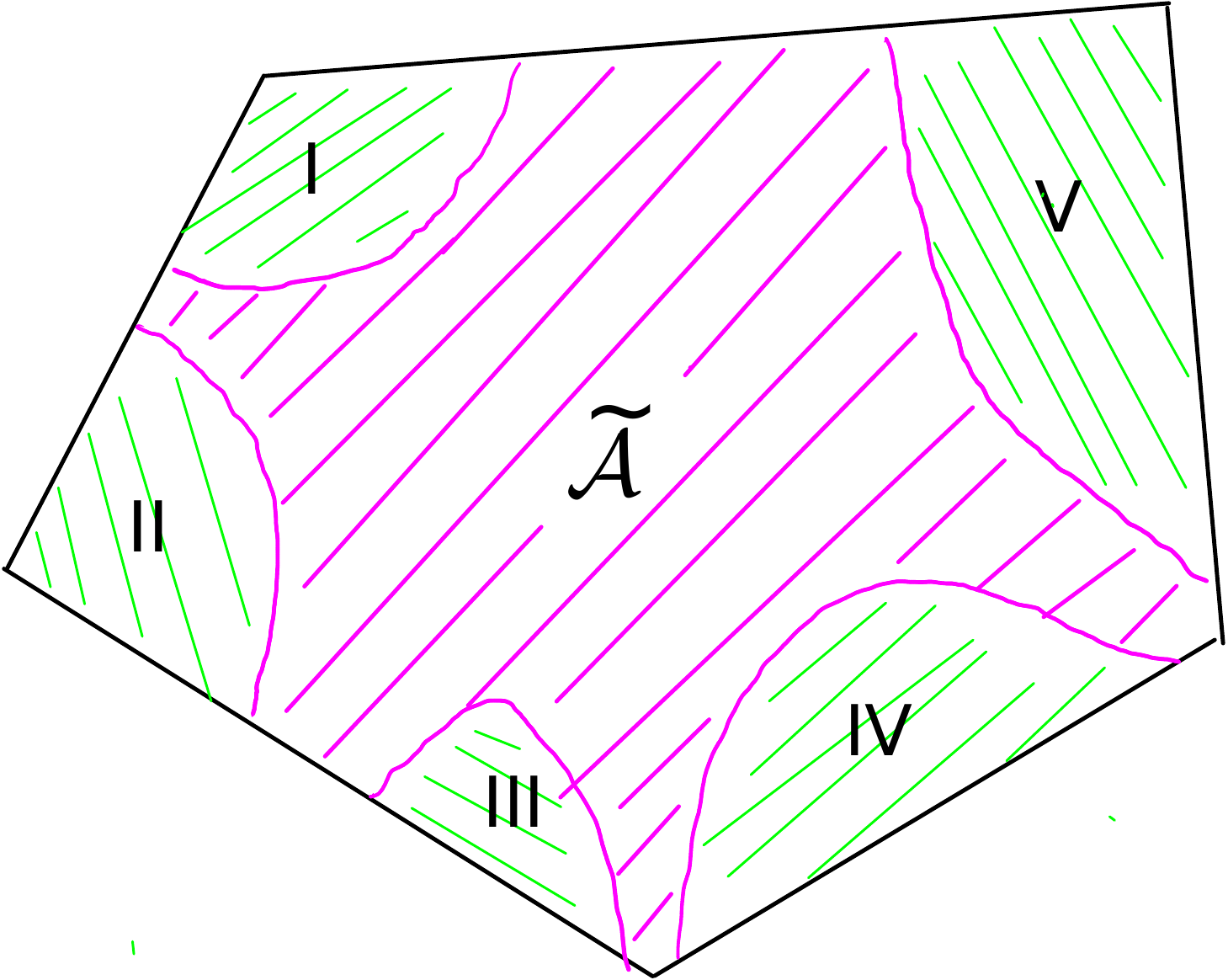}
\end{center}
\caption{Refinement of the affine structure of \protect\cite{theta} and the
common enlargement $\tilde\shA$ of $\shA$ and $\Delta$. Parallel transport
through the shaded areas with the same labels agrees.}
\label{Fig: tilde shA}
\end{figure}
In many cases, including \cite{affinecomplex} and \cite{GHK}, there is an
enlargement $\tilde\shA$ of $\Delta_2\cup\shA$ with $\Delta\subset \tilde\shA$
and such that there is a deformation retraction $\tilde\shA\to\Delta$. See
Figure~\ref{Fig: tilde shA} for a sketch of a typical situation. In particular,
tropical $1$-cycles on $B\setminus\Delta$ can be identified with tropical
$1$-cycles on $B\setminus\tilde\shA$. Replacing $\tilde\shA$ by $\shA$ then
potentially allows the consideration of further tropical cycles, those that are
not homologous to tropical cycles on $B\setminus\tilde\shA$, that is,
``passing through holes of $\shA$''. In this sense our affine structure on
$B\setminus \mathcal A$ is a refinement of the affine structure used in
\cite{theta}.
\end{remark}


\subsection{Construction of \texorpdfstring{$n$}{n}-cycles on
\texorpdfstring{$X_0^\circ$}{Xo} from tropical cycles on
\texorpdfstring{$B\setminus (\Delta_2\cup\shA)$}{B\D2uA}}

We consider tropical cycles $\beta_\trop$ as defined in
Definition~\ref{Def: Tropical $1$-cycles} for the integral affine structure from
Construction~\ref{Affine structure on B minus A}. The purpose of this section is
to construct an $n$-cycle $\beta$ on $X^\circ_0$ suitable for applying the results of
Appendix~\ref{Sect: Period integrals} for the computation of the period
integral $\int_\beta\Omega$ on $X^\circ_k$.

\begin{assumption}
\label{Ass: beta}
For our computation we make a few more assumptions on $\beta_\trop$ that with
hindsight can be imposed without restriction and with no influence on the period
integral.
\begin{enumerate}
\item[($\beta_{\mathrm{I}}$)]
Each point of intersection of $\beta_\trop$ with a wall is a vertex
of $\beta_\trop$ and an interior point of the wall. Any edge
contains at most one vertex contained in a wall.
\item[($\beta_{\mathrm{II}}$)]
Any vertex of $\beta_\trop$ of valency at least three is contained in
the interior of a chamber. 
\item[($\beta_{\mathrm{III}}$)]\label{ass}
Let $v$ be a vertex of $\beta_\trop$ contained in an $(n-1)$-cell $\rho\in\P$.
Denote by $e$, $e'$ the edges adjacent to $v$ with $\beta_\trop$ oriented from
$e$ to $e'$. Then $v$ is an interior point of a unique slab $\fob$, $v\in
\fob\setminus\mathcal A$, and all vertices of $e$ and $e'$ are bivalent. We also
assume that $e$ is contained in the image under the momentum map
$\mu:X_\sigma\to\sigma$ of the closure of an orbit of the one-dimensional
algebraic torus acting on $X_\sigma$ and fixing the toric divisor $X_\rho\subset
X_\sigma$ corresponding to $\rho$ point-wise, and similarly for $e'$ and
$\sigma'$.\footnote{In coordinates $\CC[\Lambda_\sigma]\simeq
\CC[z_1,\ldots,z_n]$ with $z_2,\ldots,z_n\in \CC[\Lambda_\rho]$, this
one-dimensional torus acts trivially on $z_2,\ldots,z_n$ and with weight~$\pm1$
on $z_1$. Hence the orbits are the level sets of $z_2,\ldots,z_n$ and their
$\mu$-image defines an integrable foliation of $\Int\sigma$ by real curves. Our
assumption says that locally $\beta_e$ maps to the closure of a leaf of this
foliation.}

Finally, we assume $\xi_e$ to be a primitive vector and if
$\xi_e\not\in \Lambda_\rho$ then $\xi_e$ generates
$\Lambda_\sigma/\Lambda_\rho$.
\end{enumerate}
\end{assumption}
All these assumptions can be realized without changing the class of
$\beta_\trop$ in $H_1(B\setminus(\Delta_2\cup\shA),\Lambda)$. For example, the last
assumption stated in ($\beta_{\mathrm{III}}$) can always be achieved as follows.
Write $\xi_e= a\zeta+ b m$ in $\Lambda_\sigma$ with $\zeta\in\Lambda_\sigma$ a
generator of $\Lambda_\sigma/\Lambda_\rho$, $a,b\in\ZZ$ and $m\in\Lambda_\rho$
primitive. Then replace both $e,e'$ by $a+b$ copies, with the first $a$ copies
carrying $\zeta$ and the last $b$ copies carrying $m$. This replacement
preserves the cycle property, that is, the balancing condition at vertices
\eqref{balancing}. A further subdivision of each new edge is needed to obtain
bivalent vertices for all edges intersecting the slab.

Note in particular that each edge $e$ of our tropical cycles $\beta_\trop$ is
considered a subset of $B$ and the map $h:\beta_\trop\to B$ mentioned after
Definition~\ref{Def: Tropical $1$-cycles} is defined on $e$ by the inclusion
$e\to B$.

\begin{construction}\label{Constr: beta}
(Construction of the $n$-cycle $\beta$ on $X_0^\circ$.)
For each edge $e$ of $\beta_\trop$ let $S(e)\subseteq X_0^\circ$ be a
differentiable section of $\mu$ over $e$, chosen compatibly over vertices. Note
that for trivial gluing data and $\sigma$ the maximal cell containing $e$, one
may choose for $S(e)$ the intersection of $\mu^{-1}(e)$ with the positive real
locus of $X_\sigma$, but in general there is no such canonical choice. For
arbitrary gluing data, we make an arbitrary choice, except if the edge $e$ has a
vertex $v$ on an $(n-1)$-cell $\rho$. In this case, if $\sigma$ denotes the
maximal cell containing $e$, we require $S(e)$ to be contained in the closure of
an orbit of the action of the one-dimensional subtorus of $\Spec
\CC[\Lambda_\sigma]$ fixing $X_\rho\subset X_\sigma$ point-wise. Note that this
condition is in agreement with the conditions imposed on $e$ in
($\beta_{\mathrm{III}}$) of Assumption~\ref{Ass: beta}. Note also that an
equivalent way to state this additional condition is to ask any monomial
$z^m\in\CC[\Lambda_\sigma]$ with $m\in\Lambda_\rho$ to take constant values on
$S(e)$.

For an edge $e$ of $\beta_\trop$ contained in a maximal cell $\sigma$ and
carrying the tangent vector $\xi_e\in\Lambda_\sigma$, define
$\ol\beta_e\subseteq X_\sigma$ as the orbit of $S(e)$ under the subgroup
\begin{equation}\label{tilde T_e} \tilde T_e=\big\{\phi\in\Hom(\Lambda_\sigma,
U(1))\,\big|\, \phi(\xi_e)=1\big\}
\end{equation}
of the real $n$-torus $T_\sigma= \Hom(\Lambda_\sigma, U(1))$ acting on
$X_\sigma$. {In coordinates, the orbits of
$\Hom(\Lambda_\sigma,U(1))\simeq U(1)^n$ are the loci of constant absolute
values of all monomials, that is, $|z^m|=\mathrm{const}$ for all
$m\in\Lambda_\sigma$. An orbit of $\tilde T_e$ inside such an orbit is then
given by $\arg(z^{\xi_e})=\mathrm{const}$. Note also that} if $\xi_e$ is an
$m_e$-fold multiple of a primitive vector $\bar\xi_e$ then $\tilde T_e$ is a
product of $\ZZ/m_e\ZZ$ with the real $(n-1)$-torus
\begin{equation}
\label{Eqn: torus T_e}
T_e= \big\{ \phi\in\Hom(\Lambda_\sigma, U(1))\,\big|\, \phi(\bar\xi_e)=1\big\}
= \xi_e^\perp\otimes_\ZZ U(1).
\end{equation}
The cyclic group $\ZZ/m_e\ZZ$ acts by multiplication by roots of unity on
$z^{\bar\xi_e}$. Thus if $e$ is disjoint from all $(n-1)$-cells, $\ol\beta_e$
is topologically a disjoint union of $m_e$ copies of the product of an interval
with $T_e$.

If one of the vertices $v$ of $e$ is contained in an $(n-1)$-cell
$\rho$, then over $v$ one has to replace $T_e$ by its image $T_{e,v}=
(\xi_e^\perp/\ZZ\check d_\rho)\otimes_\ZZ U(1)$ under the restriction map
\begin{equation}
\label{Eqn: restriction map of tori}
\Hom(\Lambda_\sigma, U(1))\lra \Hom(\Lambda_\rho, U(1)).
\end{equation}
Here $\check d_\rho\in\check\Lambda_\sigma= \Hom(\Lambda_\sigma,\ZZ)$ is a
primitive normal vector to $\Lambda_\rho\subset\Lambda_\sigma$. Note that
$T_e\to T_{e,v}$ is a finite cover of degree\footnote{By our last condition in
($\beta_{\mathrm{III}}$) of Assumption~\ref{Ass: beta}, $\xi_e$ generates
$\Lambda_\sigma/\Lambda_\rho$, so $\langle \check d_\rho,\xi_e\rangle=1$.}
$|\langle\check d_\rho,\xi_e\rangle|$ unless $\xi_e\in\Lambda_\rho$. In the
latter case $T_e\to T_{e,v}$ contracts the circle generated by $\check d_\rho$.

To define the orientation of $\ol\beta_e$ note that $\Lambda_\sigma$ is oriented
since $B$ is. Then $\xi_e$ induces a distinguished orientation on $T_{e}$ as
follows. Define a basis $\bar v_2,\ldots,\bar v_n$ of
$\Lambda_\sigma/\ZZ\bar\xi_e$ to be oriented if $\bar\xi_e,v_2,\ldots,v_n$ is an
oriented basis of $\Lambda_\sigma$, for any lift
$v_2,\ldots,v_n\in\Lambda_\sigma$ of $\bar v_2,\ldots,\bar v_n$. Then also
$\xi_e^\perp= (\Lambda_\sigma/\ZZ\bar\xi_e)^*$ is oriented and in turn $T_e$.
Now define the orientation of $\ol\beta_e$ by means of the identification
\[
\ol\beta_e\simeq S_e\times \tilde T_e\simeq S_e\times T_e\times(\ZZ/m_e\ZZ),
\]
with $S_e$ oriented by $e$. After triangulating we can view
$\ol\beta_e$ as a singular chain. If $e$ in the interior of a maximal
cell is oriented from vertex $v_-$ to $v_+$, the boundary
$\partial\ol\beta_e$ decomposes as follows:
\begin{equation}
\label{Eqn: Orientation of beta_e}
\partial\ol\beta_e= \partial_+\ol\beta_e-\partial_-\ol\beta_e,\quad
\partial_+\ol\beta_e=\{v_+\}\times \tilde T_e,\ 
\partial_-\ol\beta_e=\{v_-\}\times \tilde T_e.
\end{equation}
If $e$ intersects the codimension one cell $\rho$ in one of the vertices
$v_\pm$, the same formula holds with the factor $\tilde T_e$ in the boundary
component over $v_\pm$ replaced by the image $\tilde T_{e,{v_\pm}}$ under the
map \eqref{Eqn: restriction map of tori} above, with multiplicity $|\langle
\check d_\rho,\xi_e\rangle|$. In particular, for $\xi_e\in\Lambda_\rho$,
we have $\partial_\pm\ol\beta_e=0$ for the appropriate index $\pm$.

In any case, if $v\in\beta_\trop$ is a vertex of valency two with adjacent edges
$e,e'$ ordered according to the orientation of $\beta_\trop$, then
$\partial_+\ol\beta_e= \partial_-\ol\beta_{e'}$, so these two parts of the
boundary cancel in $\partial(\ol\beta_e+\ol\beta_{e'})$.

By ($\beta_{\textrm{II}}$), a vertex $v$ of valency at least three is contained
in the interior of a maximal cell $\sigma$. Denote by $S(v)$ the point of
intersection of $S(e)$ with $\mu^{-1}(v)$, for any edge $e$ adjacent to $v$. As
discussed above, $\tilde T_e\cdot S(v)$ is a union of translations of the
$(n-1)$-dimensional subtorus $T_e$ of $\mu^{-1}(v)= \Hom(\Lambda_\sigma, U(1))$.
The class of $\tilde T_e$ in $H_{n-1}(\mu^{-1}(v),\ZZ)$ is Poincar\'e-dual to
$\xi_e\in \Lambda_\sigma = H^1(\mu^{-1}(v),\ZZ)$. Define $\eps_{e,v}=1$ if
$\tilde T_e\cdot S(v)$ is positively oriented as part of the boundary of
$\ol\beta_e$ and $\eps_{e,v}=-1$ otherwise. By \eqref{Eqn: Orientation of
beta_e} we have $\eps_{e,v}=1$ if $e$ is oriented
toward $v$. Now the balancing condition for $\beta_\trop$ at $v$ says
\begin{equation}
\sum_{e\ni v} \eps_{e,v} \xi_e=0. \label{balancing} 
\end{equation}
Hence $\sum_{e\ni v} \eps_{e,v}\cdot[\tilde T_e\cdot S(v)] =0$ in
$H_{n-1}(\mu^{-1}(v),\ZZ)$. Thus there exists an $n$-chain
$\Gamma_v\subset\mu^{-1}(v)$ whose boundary equals the negative of this sum. The
chain $\Gamma_v$ is unique up to adding integral multiples of $\mu^{-1}(v)$. For
brevity of notation we define $\Gamma_v=0$ if $v$ is a bivalent vertex. By
construction, the sum of chains $\sum_e \ol\beta_e +\sum_v \Gamma_v$ defines an
$n$-cycle on $X_0^\circ$.

To arrive at a cycle of the form treated in Construction~\ref{Constr: finite
order periods}, we may need to adjust some of the boundary components of edges
adjacent to slabs, as will become clear in the proof of Lemma~\ref{Lem: charts
for cycle} below. To this end we admit the insertion of chains on some of such
edges as follows. Let $v$ be a vertex contained in a slab and $e$, $e'$ the
adjacent edges, with $\beta$ oriented from $e$ to $e'$. We now admit to subtract
from $\ol\beta_{e'}$ a chain $\Gamma_{e'}$, while adding $\Gamma_{e'}$ in the
chart $\Spec R^k_{\fou'}$ defined the chamber $\fou'$ containing $e'$ to the
collection of chains. For notational convenience we define
$\Gamma_e=0$ for all other edges $e$ and $\Gamma_v=0$ for all two-valent
vertices $v$. The resulting chain for any edge $e$ (modified or not) is now
denoted $\beta_e$. Thus we have $\beta_e=\ol\beta_e$ unless $e$ is oriented away
from a vertex $v_-$ lying on a slab.

Finally we define
\begin{equation}
\label{eqn:beta-sum-of-chains}
\beta:=\sum_e \beta_e + \sum_v \Gamma_v + \sum_e \Gamma_e.
\end{equation}
It follows from the construction that $\partial \beta=0$, so $\beta$ is
a singular cycle on $X_0^\circ$. Up to specifying the slab add-ins $\Gamma_{e'}$ in
Lemma~\ref{Lem: charts for cycle} below, this ends the construction of $\beta$.
\medskip

To decompose $\beta$ in the form $\sum_i \beta_i$ demanded in
Construction~\ref{Constr: finite order periods}, take for the constituents
$\beta_i$ one of the following.
\begin{enumerate}
\item
$\beta_e$ with $e$ disjoint from
$(n-1)$-cells;
\item
$\Gamma_v$ for $v$ a vertex of valency at least three;
\item
The sum $\beta_e+\beta_{e'}$ for the two edges $e,e'$ adjacent to a vertex $v$
contained in an $(n-1)$-cell;
\item
A slab add-in $\Gamma_{e'}$ whenever this chain is
non-zero.
\end{enumerate}
\end{construction}

\begin{lemma}
\label{Lem: charts for cycle}
For a cycle $\beta=\sum_i\beta_i$ from Construction~\ref{Constr: beta}
there exist charts $\Phi_i:\tilde U_i\to X_k^\circ$ and a choice of slab
add-ins $\Gamma_{e'}$, such that the charts $\Phi_i$ and chains $\beta_i$
fulfill \eqref{ChI}, \eqref{ChII} and \emph{(Cy~I), (Cy~II)} of
Construction~\ref{Constr: finite order periods}, respectively.
\end{lemma}

\noindent
\emph{Proof.} 
\underline{\emph{Step~I: Construction of adapted charts.}}
By ($\beta_{\mathrm{I}}$) of Assumption~\ref{Ass: beta} for the tropical cycle
$\beta_\trop$, the constituents of the form $\beta_i=\beta_e$ or
$\beta_i=\Gamma_v$ are contained in a single chamber $\fou$. Denote by $\sigma$
the maximal cell containing $\fou$. As explained in \S\ref{subsec-make-Xk}, the
chart of $X_k^\circ$ defined by $\fou$ provides an open embedding 
\begin{equation}
\label{Eqn: chart of type I}
\Spec R^k_\fou =\Spec\big(\CC[\Lambda_{\sigma}]\big)\times O_k
\lra X_k^\circ,
\end{equation}
which, viewed as a morphism of analytic log spaces, we take for $\Phi_i$. Thus
$\tilde U_i=\Spec\big(\CC[\Lambda_{\sigma}]\big)\times O_k$. Here we write $O_k=\Spec
\CC[t]/(t^{k+1})$ as in Appendix~\ref{Sect: Period integrals}. The reduction
$U_i$ of $\tilde U_i$ is isomorphic to $(\CC^*)^n$. In the notation of
Construction~\ref{Constr: finite order periods}, such $\Phi_i$ are charts of
type~I.

In the third instance of two edges $e,e'$ with $\beta_i=\eta_e+\beta_{e'}$ and
$\beta_\trop$ oriented from $e$ to $e'$ and meeting in a vertex $v$ on a slab
$\fob\subseteq\ul\rho$, by \eqref{Eqn: ring in codim one} we similarly have an
open embedding
\[
\Spec R^k_\fob \lra X_k^\circ,
\]
with $R^k_\fob= CC[\Lambda_\fob][\tilde Z_+,\tilde Z_-,t]/ (\tilde Z_+\tilde Z_-
- z^{-m_v}f_\fob t^{\kappa_i}, t^{k+1})$, $\kappa_i=\kappa_{\ul\rho_\fob}$. Here
we use the adapted coordinates $\tilde Z_\pm$ from \eqref{Eqn: Adapted equation
codim 1} with $\tilde\zeta=\pm\xi_e$ in case
$\xi_e\not\in\Lambda_\rho$.\footnote{Note that we use here
condition~($\beta_{\mathrm{III}}$) of Assumption~\ref{Ass: beta} that if
$\xi_e\not\in\Lambda_\rho$ then $\xi_e$ is a generator of
$\Lambda_\sigma/\Lambda_\rho$.} Denote also by $\fou,\fou'$ and $\sigma,\sigma'$
the chambers and maximal cells containing the images of $e,e'$, respectively.

To bring the ring $R^k_\fob$ into the form required by \eqref{ChII} of
Construction~\ref{Constr: finite order periods}, we now set
\begin{equation}
\label{Eqn: z,w}
z=\tilde Z_+,\quad w=\tilde Z_-/(z^{-m_v}f_\fob),
\end{equation}
to obtain an open embedding of the open neighborhood $\Spec
(R^k_\fob)_{f_{\fob}}\subseteq \Spec R^k_\fob$ of $\mu^{-1}(v)$ into
\begin{equation}
\label{Eqn: our chart}
\Spec\big(\CC[\Lambda_\rho][z,w,t]/(zw-t^{\kappa_i}, t^{k+1})\big).
\end{equation}
A further shrinking of neighborhood leads to the desired chart of the form
$\tilde U_i= V_i\times H_{\kappa_i}$ with $V_i\subset
\Hom(\Lambda_\fob,\CC^*) \simeq (\CC^*)^{n-1}$ open and
$H_{\kappa_i}$ the base change to $O_k$ of an appropriate bounded
open subset of $\big\{ (z,w,t)\in \CC^3 \big|\,
zw=t^{\kappa_i}\big\}$. With the possible rescaling of $z,w$ from
Remark~\ref{Rem: adjustments for non-standard starting points} understood, this
is a chart of type~II with $\kappa_i=\kappa_i$. Denote by $\sigma$
the maximal cell containing $e$, and by $\sigma'$ the other maximal cell
adjacent to $\fob$.

Note also that the projection $V_i\times H_{\kappa_i}\to V_i$ is a restriction
of the map
\begin{equation}
\label{Eqn: projection to codim one}
\Spec R^k_\fob \lra \Spec\CC[\Lambda_\rho]
\end{equation}
induced by the inclusion $\Lambda_\rho\subset \Lambda_\sigma$, and this map is
equivariant for the homomorphism of tori\footnote{We have $T_e=\tilde
T_e$ by the assumption of primitivity of $\xi_e$ according to
($\beta_{\mathrm{III}}$) in Assumption~\ref{Ass: beta}.} $T_e\to T_{e,v}$
discussed in Construction~\ref{Constr: beta}. By construction, $S(e)$ lies in
the fiber of this projection since any monomial $z^m$ with $m\in\Lambda_\rho$ is
constant on $S(e)$.

\noindent
\underline{\emph{Step II: Checking (Cy~I),(Cy~II) for $\beta_e=\bar\beta_e$ and for $\Gamma_v$.}}
We need to check that $\beta_i$ is of the form specified in (Cy~II) of
Construction~\ref{Constr: finite order periods}. This discussion is
entirely on the central fiber $X_0^\circ$, with the toric local model
$\CC[\Lambda_{\sigma'}]=\CC[\Lambda_\rho][\tilde Z_-^{\pm1}]$ and the non-toric one
$\CC[\Lambda_\rho][z,w]/(zw)$.

Condition~(Cy~I) is readily fulfilled if $\beta_i=\beta_e$ and $e$ is disjoint
from all slabs and for $\beta_i=\Gamma_v$. It remains to consider the case
$\beta_i=\beta_e$ and $e\cap\rho\neq\emptyset$ for some $\rho$.

The part $\beta_e$ of $\beta_i$ lying over $\sigma$ is again easily seen to be
of the required form, with the added flexibility of Remark~\ref{Rem: adjustments
for non-standard starting points} understood:

(i)~If $\xi_e\in\Lambda_\rho$ then the action of the $(n-1)$-torus $\tilde
T_e=T_e$ defined in \eqref{tilde T_e} on $\CC[\Lambda_\rho]$ has a
one-dimensional kernel, while the action on $z=\tilde Z_+$ is non-trivial.
Hence, in the chart \eqref{Eqn: our chart}, we have
$\beta_e=\gamma_i\times\big\{ |z|\le \varepsilon_i\big\}$ with
$\gamma_i$ an $(n-2)$-dimensional orbits of the action of $T_{e,v}\simeq
U(1)^{n-1}$ on $\Spec\CC[\Lambda_\rho]\simeq (\CC^*)^{n-1}$ and some
$\varepsilon_i\in\RR_{>0}$. We do not bother to compute $\gamma_i$ explicitly
because our integral over such chains vanishes in any case.

(ii)~If $\xi_e\not \in\Lambda_\rho$ then $\xi_e=\pm\tilde\zeta$. Hence
the action of $\tilde T_e= T_e$ is trivial on $z=\tilde Z_+$ and the restriction map
$T_e\to T_{e,v}$ of \eqref{Eqn: restriction map of tori} is an isomorphism.
Thus $\beta_e= \gamma_i\times z(S(e))$ with $\gamma_i$ a $T_e$-orbit and $z(S(e))$
a curve inside $\CC$ connecting $z\big(S(v_-)\big)$ to $0$ for $v_-$ the
other vertex of $e$.

\noindent
\underline{\emph{Step III: Construction of $\beta_{e'}$.}} 
The situation for the other constituent $\beta_{e'}$ of $\beta_i$ is less
straightforward. Recall that $\beta_{e'}=\ol\beta_{e'}-\Gamma_{e'}$ with
$\ol\beta_{e'}$ constructed above via the torus action and the momentum map
$\mu$, while the slab add-in $\Gamma_{e'}$ was still to be determined. Denote by
$v_+$ the vertex of $e'$ mapping to the interior of $\sigma'$, the maximal cell
containing $e'$. By construction, $\ol\beta_{e'}$ has the boundary component
$\partial_+\ol\beta_{e'}$ mapping to $v_+$ by the momentum map. This boundary
component $\partial_+\ol\beta_{e'}$ is the torus orbit $T_{e'}\cdot S(v_+) =
\tilde T_{e'}\cdot S(v_+)$ in $\Hom(\Lambda_{\sigma'},\CC^*)$, the reduction
modulo $t$ of the chart $\Spec R^k_{\fou'}$.

For the following discussion, let $e_1,\ldots,e_n\in\Lambda_{\sigma'}$ be an
oriented basis with $e_1=\tilde\zeta'$ the parallel transport of
$-\tilde\zeta\in \Lambda_\sigma$ through $v$, and $e_2,\ldots,e_n\in\Lambda_\rho$,
and let $\hat z_1,\ldots,\hat z_n\in R^k_{\fou'}$ denote the corresponding
monomials. Then $T_{e'}$ is identified with a subtorus of $U(1)^n$ acting
diagonally on $\hat z_1,\ldots,\hat z_n$. For $\theta\in T_{e'}$ denote by
$(\theta_1,\ldots,\theta_n)$ the corresponding image in $U(1)^n$, with $U(1)=
\big\{z\in\CC\,\big|\, |z|=1\big\}$.

In these coordinates, $\partial_+\ol\beta_{e'}$ has the parametrization
\[
T_{e'}\ni \theta=(\theta_1,\ldots,\theta_n)\longmapsto 
\theta\cdot a=\big(\theta_1 a_1, \theta_2 a_2,\ldots,\theta_n a_n\big),
\]
with $a_\mu=\hat z_\mu\big(S(v_+)\big)$ for $\mu=1,\ldots,n$ and $a=(a_1,\ldots,a_n)$.

On the other hand, Condition~(Cy~II) of Construction~\ref{Constr: finite order
periods} tells us that $\beta_{e'}$ must be homologous relative to its boundary
to the chain $\hat\beta_{e'}$ defined analogously, but using the chart $\Phi_i$
modeled on $\CC[\Lambda_\rho][z,w]/(zw-t^{\kappa_i}, t^{k+1})$ and $w$ replacing $\hat
z_1$. Explicitly, in the coordinates $w,z_2,\ldots,z_n$ of the chart $\Phi_i$, with
$z_2,\ldots,z_n$ defined by $e_2,\ldots,e_n\in\Lambda_\rho$, the chain
$\hat\beta_{e'}$ is defined by the parametrization
\begin{equation}
\label{Eqn: hat beta_{e'}}
S(e')\times T_{e'}\lra (\CC^*)^n,\quad
(p,\theta)\longmapsto \big(\theta_1 w(p),\theta_2 z_2(p),\ldots,\theta_n z_n(p)\big).
\end{equation}
The two sets of coordinates are related by
\begin{equation}
\label{Eqn: hatted versus unhatted coords}
c_1 \hat z_1= (f_\fob/z^{m_v})\,w,\quad
c_2 \hat z_2= z_2,\ldots,c_n \hat z_n= z_n,
\end{equation}
with constants $c_\mu= s_{\sigma'\ul\rho}(e_\mu)\in\CC^*$ given by gluing data.
The action of $U(1)^n$ on $z_2,\ldots,z_n$ is compatible with the action on
$\hat z_2,\ldots,\hat z_n$, but not so on $w$. Let $\hat f_\fob\in \CC[t][\hat
z_2,\ldots,\hat z_n]$ be a Laurent polynomial with the property that the
reduction of $\hat z^{-m_v} \hat f_\fob$ modulo $t^{k+1}$ is the image of
$z^{-m_v} f_\fob$ under the gluing map $R^k_\fob\to R^k_{\fou'}$, and denote by
$\hat f_{\ul\rho}$ the reduction of $\hat f_\fob$ modulo $t$. Then the first
equation in \eqref{Eqn: hatted versus unhatted coords} can be rewritten as
\begin{equation}
\label{Eqn: hat z_1 in terms of w}
\hat z_1= c_1^{-1} (\hat f_\fob/\hat z^{m_v})\, w.
\end{equation}
To describe $\partial_+ \hat\beta_{e'}$ only the reduction modulo $t$ is
relevant and hence $f_\fob$ reduces to $f_{\ul\rho}$. Thus in the coordinates
$\hat z_1,\ldots,\hat z_n$, the boundary $\partial_+ \hat\beta_{e'}$ has the
parametrization
\begin{equation}
\label{Eqn: bdry of hat beta_{e'}}
T_{e'}\ni \theta= (\theta_1,\ldots,\theta_n)\longmapsto 
\Big(c_1^{-1}\big(\hat f_{\ul\rho}/\hat z^{m_v}\big)(\theta a)\cdot \theta_1 b_1,
\theta_2 a_2,\ldots,\theta_n a_n\Big),
\end{equation}
where $b_1=w\big(S(v_+)\big)$. For simplicity of notation we view here $\hat f_{\ul\rho}/\hat
z^{m_v}$ as a Laurent polynomial in $n$ variables by the inclusion
$\Lambda_\rho\subset \Lambda_{\sigma'}$.

\noindent
\underline{\emph{Step IV: Construction of slab add-ins $\Gamma_{e'}$.}}
Now that we have explicit parametrizations of both $\partial\hat\beta_{e'}$ and
$\partial\bar\beta_{e'}$, we are ready to construct the slab add-in
$\Gamma_{e'}$ as a chain connecting these boundary cycles. Note that the factor
in the first entry of the right-hand side of \eqref{Eqn: bdry of hat beta_{e'}}
agrees with the restriction of $f_{\ul\rho}/z^{m_v}$ on the fiber of the
momentum map $\mu:X_0^\circ\to B$ over $v$:
\[
\big(\hat f_{\ul\rho}/\hat z^{m_v}\big)(\theta a)
= \big(\hat f_{\ul\rho}/\hat z^{m_v}\big)(\theta_2 a_2,\ldots,\theta_n a_n)
=\big(f_{\ul\rho}/z^{m_v}\big)\big(\theta_2 b_2,\ldots,\theta_n b_n\big),
\]
with $b_2=z_2(S(v)),\ldots, b_n=z_n(S(v))$. The point is that, by the definition
of $m_v$ in Construction~\ref{Affine structure on B minus A}, this map is
homotopically trivial as a map $T_{e'}\to \CC^*$. Thus there is a
differentiable homotopy $\gamma: [0,1]\times T_{e'}\lra \CC^*$ with\\[-6ex]
\begin{wrapfigure}[11]{r}{0.38\textwidth}
\captionsetup{width=.75\linewidth}
\begin{center}
{\includegraphics[width=0.34\textwidth]{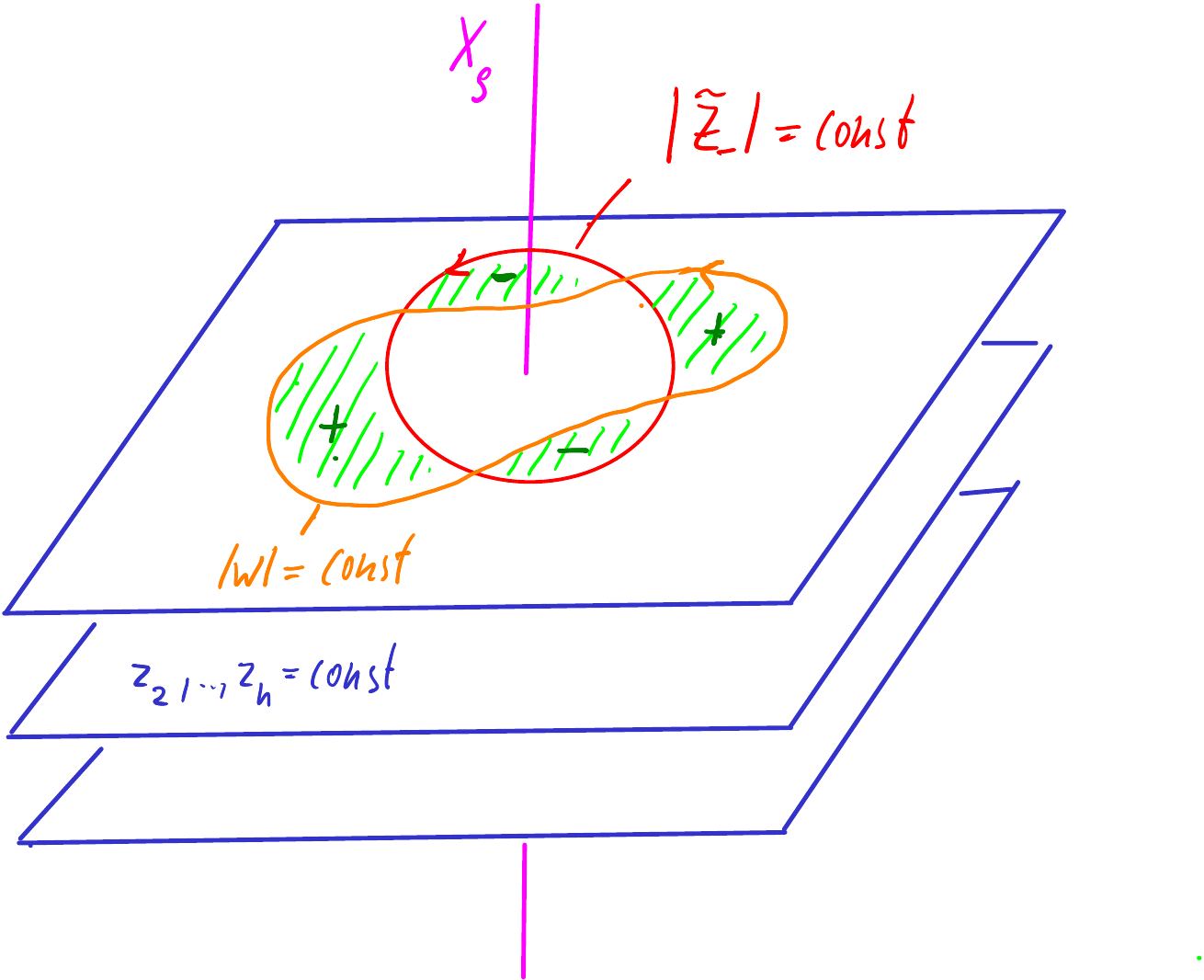}}
\end{center}
\caption{A slab add-in (case $\xi_{e'}\in\Lambda_\rho$)}
\label{Fig: slab add-in}
\end{wrapfigure}
\begin{eqnarray}
\label{Eqn: homotopy gamma}
\lefteqn{\gamma(0,\theta)= \big(\hat f_{\ul\rho}/\hat z^{m_v}\big)(\theta a),\quad
\gamma(1,\theta)= a_1 c_1/b_1,}\hspace{65ex}
\end{eqnarray}
where $a_1=\hat z_1\big(S(v_+)\big)$, $b_1=w\big(S(v_+)\big)$ and $c_1=s_{\sigma'\ul\rho}(\tilde\zeta')$.
We now define the slab add-in $\Gamma_{e'}$ in the coordinates $\hat
z_1,\ldots,\hat z_n$ of $R^k_{\fou'}$ by the parametrization\\
\begin{minipage}{0.6\textwidth}
\begin{eqnarray}
\label{Eqn: parametrization of slab add-in}
\lefteqn{\hspace{8ex}[0,1]\times T_{e'} \lra \Hom(\Lambda_{\sigma'},\CC^*)}\\
&&(s,\theta)\longmapsto \big(\gamma(s,\theta)\cdot
\theta_1 b_1/c_1, \theta_2 a_2,\ldots,\theta_n
a_n\big),\nonumber
\end{eqnarray}
\end{minipage}\\[2ex]
with the given orientation of the domain. Figure~\ref{Fig: slab add-in} provides
a sketch for the case $\xi_{e'}\in\Lambda_\rho$. The horizontal planes indicate
the level sets of $\hat z_2,\ldots,\hat z_n$, which for both $\ol\beta_{e'}$ and
$\hat\beta_{e'}$ vary in the same $U(1)^{n-2}$-orbit, the orbit
containing $(\hat z_2,\ldots,\hat z_n)\big(S(v_+)\big)=(\hat z_2,\ldots,\hat
z_n)(S(e'))$. The shaded region is the part of $\Gamma_e$ in this level set. The
circle and curve show the intersection of $\partial_+\ol\beta_{e'}$ and
$\partial_+\hat\beta_{e'}$ with one of the level sets in this
$U(1)^{n-2}$-orbit. In the other case $\xi_{e'}\not\in\Lambda_\rho$ there is a
$U(1)^{n-1}$-orbit of level sets and on each level set in this orbit,
$\partial_+\ol\beta_{e'}$, $\partial_+\hat\beta_{e'}$ define two points in the
$\hat z_1$-plane, connected by the homotopy $\gamma$.

By construction, $\partial\Gamma_{e'}= \partial_+\ol\beta_{e'}
-\partial_+\hat\beta_{e'}$ and hence $\beta_{e'}= \ol\beta_{e'}-\Gamma_{e'}$ has
boundary $\partial_+\hat\beta_{e'}-\partial_-\ol\beta_{e'}$. Letting the
endpoint $S(v_+)$ vary as $s\in S(e')$, $\gamma$ can be extended to a continuous
family $\gamma_s$ of homotopies between the $T_{e'}$-orbits in
$\ol\beta_{e'}$ and in $\hat\beta_{e'}$ containing $s$. The corresponding family of
$n$-chains sweeps out an $(n+1)$-chain $\tilde\Gamma$ with
\[
\partial\tilde\Gamma= \ol\beta_{e'}-\Gamma_{e'}-\hat\beta_{e'}.
\]
Thus $\beta_{e'}$ and $\hat\beta_{e'}$ are homologous relative to their
boundaries as needed in (Cy~II).

We add the slab add-in $\Gamma_{e'}$ as an additional chain, taken inside the
chart \eqref{Eqn: chart of type I} for the chamber $\fou'$ containing $e'$. With
this definition for the slab add-ins, we have verified the requirements of
Construction~\ref{Constr: finite order periods} for all constituents $\beta_i$
of $\beta$.
\qed

\begin{proposition}
\label{Lem: monodromy action}
Let $\shX\ra \DD$ be a family with $X_0^\circ$ as central fiber and locally
analytically isomorphic to \eqref{Eqn: adapted local eqn}, together with a
family of $n$-cycles $\beta(t)$ with $\beta(0)=\beta$ as in
Proposition~\ref{holomorphic-periods-induce-formal-periods}. Then the
Picard-Lefschetz monodromy $T$ along a counter-clockwise loop in $\DD$ based at
$t_0\neq 0$ acting on $n$-homology classes is given by
\[
(T-\id)(\beta_{t_0})=\langle c_1(\varphi),\beta_\trop\rangle \cdot[\alpha]
=\sum_i \kappa_i\langle \check d_{\rho_i},\xi_i\rangle [\alpha].
\]
Here $\alpha$ denotes the vanishing cycle and the sum is over the
points $v$ of intersection of $\beta_\trop$ with codimension one cells
$\rho=\rho_i$ as explained in \S\ref{Par: PL trf}.
\end{proposition}

\begin{proof}
We apply Lemma~\ref{lemma-on-monodromy} to the specific situation laid out in
the proof of Lemma~\ref{Lem: charts for cycle}. This already justifies what we
are summing over. For each summand, we are in the situation (ii) stated in
Step~II of the proof, where $\xi_e\not \in\Lambda_\rho$ and
$\beta_e=\gamma_i\times z(S(e))$. A clockwise loop in the $w$-plane gives a
counter-clockwise loop in the $z$-plane, call this $S^1$. Hence, by
Lemma~\ref{lemma-on-monodromy}, $(T-\id)(\beta_e+\beta_{e'})=
\kappa_i[\gamma_i\times S^1]$, on the level of chains up to homology. The
vanishing cycle $\alpha$ is represented by an orbit of the diagonal action of
$U(1)^n$ in the coordinates $z^{\xi_e},z_2,\ldots,z_n$ obtained from an oriented
basis $\xi_e,e_2,\ldots,e_n$ with $e_j\in\Lambda_\rho$. Relating this orbit to
$\gamma_i\times S^1$, we have $z=z^{\eps\xi_e}$ for $\eps=\langle \check
d_\rho,\xi_e\rangle=\pm 1$. Recall that $\gamma_i$ is a $T_e$-orbit, so taken
together with $S^1$, the circle in the $z$-plane, this cycle is indeed
homologous to $\alpha$ up to sign. The signs work out as stated.
\end{proof}

%
%

\section{Computation of the period integrals}
\label{section-compute-periods}
The purpose of this section is to prove Theorem~\ref{Thm: period thm}. We
continue to use the setup from Section~\ref{section-tropical-to-homology}. In
particular, we have $X_k^\circ$, an $n$-dimensional log scheme over $O_k= \Spec
\CC[t]/(t^{k+1})$ with restriction to $X_k^\circ\setminus Z$ log smooth over
$O_k$. Here $O_k$ is given the log structure induced from the toric log
structure on $\Spec\CC[t]$. Denote by $\Omega^n_{X_k^\circ/O_k}$ the sheaf of
relative log differentials of degree $n$, which is locally free away from $Z$.
The construction of $X_k^\circ$ comes with a canonical relative logarithmic
$n$-form $\Omega\in \Gamma(X_k^\circ, \Omega^n_{X_k^\circ/O_k})$. If $\sigma$ is
a maximal cell and $e_1,\ldots,e_n$ is an oriented lattice basis of
$\Lambda_\sigma$, then, in the corresponding local coordinates
$z_1=z^{e_1},\ldots, z_n=z^{e_n}$ of $\Spec (R^k_\sigma)= \Spec \big(\CC[t,z_1,
\ldots,z_n]/ (t^{k+1})\big)$, it holds
\begin{equation}\label{Eqn: Omega}
\Omega= \dlog z_1\wedge\ldots\wedge \dlog z_n
= z_1^{-1} dz_1\wedge\ldots\wedge z_n^{-1} dz_n.
\end{equation}
We often also work with polar coordinates $z_j= r_j e^{\sqrt{-1}\alpha_j}$.
To avoid cluttering some formulas with exponentials, we work with
$\theta_j=e^{\sqrt{-1}\alpha_j}\in U(1)=S^1$ rather than with $\alpha_j\in
\RR/2\pi$, as already in the proof of Lemma~\ref{Lem: charts for cycle}. In particular, $\dlog z_j= \dlog r_j+\sqrt{-1} d\alpha_j$ now reads
$\dlog z_j= \dlog r_j+\dlog\theta_j$ and it holds
\begin{equation}
\label{Eqn: int dlog theta}
\int_{S^1}\dlog\theta_j = \sqrt{-1}\int_{S^1} d\alpha_j = 2\pi\sqrt{-1}.
\end{equation}
Recall that $Z$ intersected with the interior of a codimension one stratum
$X_\rho\cap X_0^\circ\subset X_0^\circ$ is given by the zero locus of
$f_{\ul\rho}$, the reduction of $f_\fob$ modulo $t$ for any slab
$\fob\subseteq\ul\rho\subset\rho$. In Construction~\ref{Affine structure on B
minus A} we defined an adapted affine structure on
$B\setminus(\Delta_2\cup\shA)$ for $\shA=\mu(Z)$ the image of $Z$ under the
generalized momentum map $\mu: X_0^\circ\to B$. For a tropical cycle
$\beta_\trop$ on $B\setminus(\Delta_2\cup\shA)$ fulfilling Assumption~\ref{Ass:
beta} and $\beta$ the associated $n$-cycle from Construction~\ref{Constr: beta},
we now compute $\int_\beta\Omega$ in the form $h+g\log t$ with $h,g\in
\CC[t]/(t^{k+1})$ following Appendix~\ref{Sect: Period integrals}.

We first compute the period of $\Omega$ over a general fiber of
the momentum map $\mu: X_0^\circ\to B$ of Proposition~\ref{Prop:
momentum map}.

\begin{lemma}
\label{lemma-intalpha}
Let $v\in B$ be contained in the interior of a maximal cell $\sigma$
and $\alpha=\mu^{-1}(v)$, viewed as an $n$-cycle in $X_0^\circ$ with the
natural orientation. Then, in the sense of finite order period
integrals (Construction~\ref{Constr: finite order periods}),
\[
\displaystyle\int_\alpha\Omega=(2\pi \sqrt{-1})^n \in \CC[t]/(t^{k+1}).
\]
\end{lemma}
\begin{proof}
The cycle $\alpha$ is contained in a single chart $\tilde
U_1=\Spec R^k_\fou$ of type~\eqref{ChI}, for any chamber
$\fou\subseteq\sigma$. Using \eqref{Eqn: int dlog theta}, we obtain
\[
\int_\alpha\Omega = \int_{(S^1)^n}
\dlog\theta_1\wedge\ldots\wedge \dlog\theta_n = (2\pi\sqrt{-1})^n.
\]\\[-6ex]
\end{proof}

According to Proposition~\ref{holomorphic-periods-induce-formal-periods},
Lemma~\ref{lemma-intalpha} proves the ambiguity of $\int_\beta\Omega$ up to
multiples of $(2\pi\sqrt{-1})^n$, hence the stated well-definedness of the
exponentiated period integral in Theorem~\ref{Thm: period thm}.

We now turn to the computation of $\int_\beta \Omega$ for $\beta$ as in
\eqref{eqn:beta-sum-of-chains}.


\subsection{Integration over \texorpdfstring{$\beta_i=\beta_e$}{bi=be} with \texorpdfstring{$\Phi_i$}{Pi} a chart of type~I}
\label{section-chamber-integral}

Let $e$ be an edge of $\beta_\trop$ in the interior of a maximal cell $\sigma$,
with vertices $v_\pm$ and $e$ oriented from $v_-$ to $v_+$. As in
Construction~\ref{Constr: beta} write $\xi_e=m_e\cdot\bar\xi_e$ with $m_e\in\NN$
and $\bar\xi_e\in\Lambda_\sigma$ primitive. Complete $\bar\xi_e=e_1$ to an oriented basis
$e_1,\ldots,e_n$ of $\Lambda_\sigma$. Then the inclusion $\ZZ^{n-1}\to
\Lambda_\sigma$ defined by $e_2,\ldots,e_n$ induces an identification of $T_e$
with $U(1)^{n-1}$ acting diagonally on $(\CC^*)^{n-1}$ with coordinates
$z_2=z^{e_2},\ldots,z_n=z^{e_n}$ and acting trivially on $z^{\ol\xi_e}$. Recall also from
Construction~\ref{Constr: beta} that $\beta_e=\ol\beta_e$ is defined as the
orbit of $S(e)$ under $\tilde T_e= T_e\times \ZZ/m_e\ZZ$, with $\ZZ/m_e\ZZ$ acting on
$z^{\ol\xi_e}$ by roots of unity.

According to Definition~\ref{Def: Phi^+} and \eqref{Eqn: Omega}, it holds 
\[
\Phi_i^+(\Omega) = \dlog z^{\bar\xi_e}\wedge \dlog z_2\wedge \ldots\wedge \dlog z_n.
\]
In view of \eqref{Eqn: integrals for Ch I}, we now compute
\begin{eqnarray*}
\int_{\beta_i} \Phi_i^+(\Omega) &=&
\int_{S(e)\times\tilde T_e} \dlog z^{\ol\xi_e}\wedge \dlog z_2\wedge\ldots\wedge \dlog z_n\\
&=& \int_{S(e)\times T_e\times\ZZ/m_e\ZZ}
\dlog z^{\ol\xi_e}\wedge \dlog\theta_2\wedge\ldots\wedge \dlog\theta_n\\
&=& (2\pi\sqrt{-1})^{n-1} \int_{S(e)\times\ZZ/m_e}\dlog z^{\ol\xi_e}\\
&=& (2\pi\sqrt{-1})^{n-1} \sum_{\nu=0}^{m_e-1}
\Big(\log \big(\epsilon^\nu z^{\ol\xi_e}\big(S(v_+)\big)\big)
- \log\big(\epsilon^\nu z^{\ol\xi_e}\big(S(v_-)\big)\big) \Big),
\end{eqnarray*}
where $\epsilon$ denotes a primitive $m_e$-th root of unity. Expanding $\log
(\epsilon^\nu z^{\ol\xi_e}(S(v_\pm)) = \log \epsilon^\nu+\log
z^{\ol\xi_e}(S(v_\pm))$, each term $\log \epsilon^\nu$ in the sum occurs twice
with opposite signs, leaving us with an $m_e$-fold sum of $\log
z^{\ol\xi_e}\big(S(v_+)\big) - \log z^{\ol\xi_e}\big(S(v_-)\big)$. Thus the sum
equals the difference of $\log z^{\xi_e}=m_e\log z^{\bar\xi_e}$ at the two
endpoints of $S(e)$, that is,
\begin{equation}
\label{Eqn: integral over edges of type I}
\int_{\beta_e}\Omega = (2\pi\sqrt{-1})^{n-1} \Big( \log z^{\xi_e}\big(S(v_+)\big)
-\log z^{\xi_e}\big(S(v_-)\big)\Big),
\end{equation}
for $e$ oriented from $v_-$ to $v_+$.


\subsection{Integration over \texorpdfstring{$\Gamma_v$}{Gv}}
\label{section-vertex-integral}

We need the following lemma.
\begin{lemma} 
\label{lemma-alternating-roots}
Let $\mu_k\subset U(1)$ denote the subgroup of $k$-th roots of
unity. For any two positive integers $n$ and $m$, the subsets
\[
A=\mu_m\cup \mu_n \quad\hbox{ and }\quad
B=\mu_{m+n}\setminus ((\mu_m\cup \mu_n)\cap \mu_{m+n})
\]
of $U(1)=S^1$ alternate, that is, following the circle, we
alternately cross a point from $A$ and $B$.
\end{lemma}
\begin{proof} 
First assume $m=n$. Then we have $A=\{\exp(2\pi\sqrt{-1}\frac{2k}{2m})\mid
k\in\ZZ\}$ and $B=\{\exp(2\pi\sqrt{-1}\frac{2k+1}{2m})\mid k\in\ZZ\}$ and the
assertion holds. Next assume $m\neq n$. Set $d=\gcd(m,n)$. We may view the
situation as a $d$-fold cover of the case where $m$ and $n$ are coprime. As the
assertion transfers to the cover, we may assume that $\gcd(m,n)=1$ and then
$\mu_m\cap \mu_n =\{1\}$ and $\operatorname{lcm}(m+n,n)=n(m+n)$ and
$\operatorname{lcm}(m+n,m)=m(m+n)$. Hence
\[
\mu_{m+n} \cap (\mu_m\cup \mu_n)= \{1\},
\]
so in particular $A$ and $B$ have the same number of elements,
$m+n-1$. Now assume to the contrary of the assertion that there are consecutive
elements in $B$ with no element of $A$ in between. This means there are integers
$a,b,c$ such that
\[
\frac{a}m, \frac{b}n < \frac{c}{m+n} < \frac{c+1}{m+n} <
\frac{a+1}m, \frac{b+1}n.
\]
Multiplying common denominators yields
\begin{eqnarray*}
&0 < (a+1)(m+n)-(c+1)m,\quad
0 < (b+1)(m+n)-(c+1)n,\\
&a(m+n) < cm,\quad
b(m+n) < cn.
\end{eqnarray*}
Plugging the third and fourth inequalities into the first and second, respectively,
with subsequent summation of the resulting equations yields
\[
0 < (a+b+2)(m+n) - (c+1)(m+n) < m+n
\]
which has no solution with $a,b,c\in\ZZ$.
\end{proof}

Recall the definition of $\Gamma_v\subset \mu^{-1}(v)=
\Hom(\Lambda_v, U(1))$ from Construction~\ref{Constr: beta}.
\begin{lemma} 
\label{lemmaav}
Let $v\in\beta_\trop$ be a vertex of valency $\upnu\ge 3$. Then
\begin{equation}\label{Eqn: Integral over Gamma_v}
\frac{1}{(2\pi\sqrt{-1})^n}\int_{\Gamma_v}\Omega
= \left\{ \begin{array}{ll} 0, & \upnu\hbox{ is even,}\\
1/2, & \upnu\hbox{ is odd.}\end{array} \right.
\end{equation}
up to adding integers.
\end{lemma}
\noindent\emph{Proof.}
By construction, $\Gamma_v$ is a singular $n$-chain on the $n$-torus
$\mu^{-1}(v)$. The restriction of $\Omega$ to this torus is
$\dlog\theta_1\wedge\ldots\wedge \dlog\theta_n$, which agrees with
$(2\pi\sqrt{-1})^n$ times the $U(1)^n$-invariant volume form $d\mathrm{vol}$ of
total volume $1$. Thus the statement concerns the volume of $\Gamma_v$ as a
fraction of the volume of $\mu^{-1}(v)$.

We have $\sum_{e\ni v} \eps_{e,v} \xi_e=0$. Set $\xi_j:=\eps_{e_j,v} \xi_{e_j}$
for $e_1,\ldots,e_r$ an enumeration of the edges containing $v$.
\begin{figure}
\resizebox{0.9\textwidth}{!}{
\includegraphics{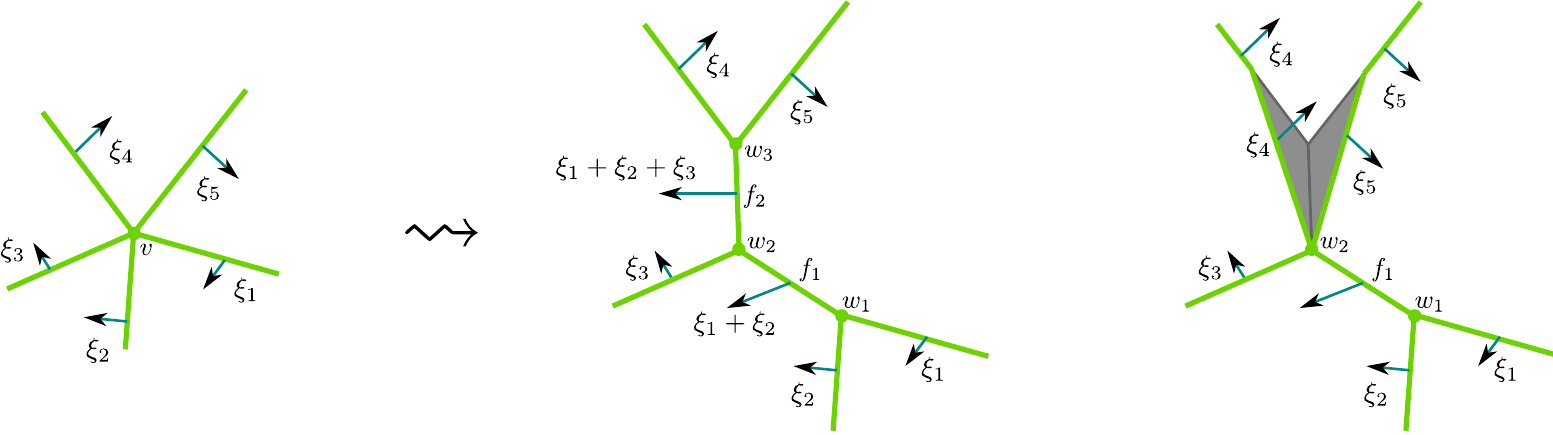}
}
\caption{Making a vertex trivalent by the insertion of new edges and a 2-chain (in grey) that deletes one of the new edges: $f_2$}
\label{Fig: make-trivalent}
\end{figure}
We decompose $v$ into trivalent vertices via insertion of $\upnu-3$ new edges
$f_1,\ldots,f_{\upnu-3}$ meeting the existing edges in the configuration, as
depicted in Figure~\ref{Fig: make-trivalent}. Precisely, we replace $v$ by a
chain of new edges $f_1,\ldots,f_{\upnu-3}$ such that the ending point of $f_j$
is the starting point of $f_{j+1}$. Let $w_1,\ldots,w_{\upnu-2}$ denote the
vertices in this chain, the indices arranged so that $w_1$ meets $e_1,e_2$,
$w_2$ meets $e_3$, $w_3$ meets $e_4$ and so forth, finally $w_{\upnu-2}$ meets
$e_{\upnu-1},e_\upnu$. The edge $f_j$ is decorated with the section
$\xi_1+\ldots+\xi_{j+1}$. One checks that at each vertex $w_j$ the balancing
condition \eqref{balancing} holds. One also checks that the new tropical curve
is homologous to the original one. Indeed, adding boundaries of suitable
2-cycles, we can successively slide down the edges $e_3,e_4,\ldots$ to $w_1$.
In this process the sections along $f_1,\ldots,f_{\upnu-3}$ get modified and
when all $e_j$ have been moved to the first vertex, the sections of the $f_j$
are all trivial and so we end up in the original setup by setting $w_1=v$. Since
there is an injection of groups of chains
\[
C_j(\RR^n,\Lambda) \ra
C_j(\RR^n\times \Hom(\ZZ^n,U(1))),\qquad (c,\xi)\mapsto c\times
\Hom(\ZZ^n/\xi,U(1))
\]
compatible with boundary maps, we conclude that the
associated $n$-cycles to the original and modified $\beta_\trop$ are homologous
as well. Hence
\[
\int_{\Gamma_v}\Omega =  \int_{\Gamma_{w_1}}\Omega +\ldots +
\int_{\Gamma_{w_{\upnu-2}}}\Omega \quad \mod (2\pi\sqrt{-1})^n\ZZ.
\]
We have reduced the assertion to the case where $v$ is trivalent. So
we assume $\upnu=3$ now. As before, set $\xi_j:=\eps_{e_j,v}
\xi_{e_j}$ for $j=1,2,3$.  By the balancing condition \eqref{balancing}, the saturated
integral span $V$ of $\xi_1,\xi_2,\xi_3$ has either rank one or two.
In either case, we have a product situation where we can split
$\Lambda_v\simeq V\oplus W$ which yields a splitting of the torus
\[
\Hom(\Lambda_v,U(1))\simeq \Hom(V,U(1))\times \Hom(W, U(1)),
\]
and $\Gamma_v$ also splits as $\bar\Gamma_v\times \Hom(W, U(1))$. The
integral over the invariant volume form splits similarly with the integral over
$\Hom(W,U(1))$ giving a factor of $1$. It remains to treat the case
$\Lambda_v=V$.

We treat the one-dimensional case first. Let $e$ be a primitive generator of $V$
and $\xi_j=a_j e$. We have $-a_3=a_1+a_2$. Canceling coincidental points (as
these have opposite orientation) between the multi-sets $\hat
A=\exp(2\pi\sqrt{-1}\frac1{a_1}\ZZ)\cup \exp(2\pi\sqrt{-1}\frac1{a_2}\ZZ)$ and
$\hat B=\exp(2\pi \sqrt{-1}\frac1{a_1+a_2}\ZZ)$, we obtain sets $A$
and $B$ as in the setup of Lemma~\ref{lemma-alternating-roots}. The lemma
implies that $\Gamma_v$ up to addition of multiples of the fundamental class is
homologous to a union of non-intersecting intervals with the union of endpoints
being $A\cup B$. This implies that $\Gamma_v$ is homologous to the sum of every
other interval between the pairs of points in $A\cup B$. We claim that
the area of $\Gamma_v$ is half the area of $S^1$. Indeed, the sets $A$
and $B$ are both invariant under conjugation $\kappa: z\mapsto \bar z$.
Moreover, $\kappa$ takes $\Gamma_v$ to the closure of its complement,
so $\Gamma_v$ and $\kappa(\Gamma_v)$ have the same area. Thus
\[
\int_{\Gamma_v}d\mathrm{vol} = \frac12\int_{S^1}d\mathrm{vol} = \frac12
\]
up to adding integers.

We next turn to the case where $V$ is two-dimensional.
In the universal cover $V_\RR^*=\Hom(V,\RR)$ of $\Hom(V,U(1))$, the cycles in
$\Hom(V,U(1))$ given by requiring $\xi_j\mapsto 1$ for $j=1,2,3$, respectively,
pull back to the infinite, discrete union of distinct straight lines
$\bigcup_{j=1}^3 (\xi_j^\perp+\ZZ^2)$. Let $U\subset V_\RR^*$ denote the open
complement of these lines. We claim that the pullback $\tilde\Gamma_v$ of
$\Gamma_v$ to $V_\RR^*$ can be taken as the closure in $V_\RR^*$ of a set of
components of $U$ such that $-\tilde\Gamma_v$ is the closure of
$V_\RR^*\setminus \tilde\Gamma_v$. If this holds then by a similar argument as
in the one-dimensional case we obtain $\int_{\Gamma_v} d\mathrm{vol} = \frac12$
up to integers.

To see the claim, consider the map of lattices $\ZZ^2\to V$ mapping $e_1$
to $\xi_1$ and $e_2$ to $\xi_2$.
\begin{wrapfigure}[11]{r}{0.3\textwidth}
\captionsetup{width=.8\linewidth}
\begin{center}\vspace{-.3cm}
\includegraphics[width=0.2\textwidth]{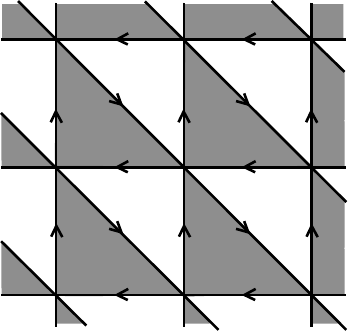}
\end{center}
\caption{Two $\ZZ^2$-invariant sets of lattice triangles in $\ZZ^2\otimes_\ZZ\RR$.}
\label{Fig: triangle-tiles}
\end{wrapfigure}
By the balancing condition, $-e_1-e_2$ then maps
to $\xi_3$. Dually we obtain an inclusion of lattices $V^*\to \ZZ^2$ of the same
index as the sublattice $\ZZ\xi_1+\ZZ\xi_2\subseteq V$. Now $\xi_j^\perp$
maps to the lines in directions $(0,1)$, $(-1,0)$ and $(1,-1)$, respectively,
with the stated orientations. Together with their $\ZZ^2$-translations these
lines subdivide $\RR^2= \ZZ^2\otimes_\ZZ\RR$ into triangular domains, see
Figure~\ref{Fig: triangle-tiles}:
$\ZZ^2$-translations of the two
triangles with vertices $(0,0)$, $(1,0)$, $(0,1)$ and
$(1,0)$, $(1,1)$, $(0,1)$. The first triangle with the natural
orientation of $\RR^2$ and its $\ZZ^2$-translations define a $\ZZ^2$-invariant
chain $A\subset \RR^2$ with the union of lines as its boundary. Moreover,
multiplication by $-1$ leads to the other triangle and its $\ZZ^2$-translations.
Take for $\tilde\Gamma_v\subset V_\RR^*$ the preimage of $A$ under the map
$V_\RR^*\to \RR^2$. Then $\tilde\Gamma_v\cup (-\tilde\Gamma_v)=V_\RR^*$ and
$\tilde\Gamma_v\cap (-\tilde\Gamma_v)$ is the infinite union of lines, as
claimed. 
\qed


\subsection{Integration over \texorpdfstring{$\beta_i=\beta_e+\beta_{e'}$}{bi=be+be'}
with \texorpdfstring{$\Phi_i$}{Pi} a chart of type~II}
\label{section-slab-integral}

Let $v$ be a vertex of $\beta_\trop$ in the interior of a slab
$\fob\subseteq\ul\rho $ with adjacent edges $e$, $e'$ and
$\beta_i=\beta_e+\beta_{e'}$. We use the notation from Construction~\ref{Constr:
beta} and in addition denote $v_-,v_+$ the vertices of $e,e'$ different from
$v$. The chart $\Phi_i$ was defined in the proof of Lemma~\ref{Lem: charts for
cycle} from $R^k_\fob=\CC[\Lambda_\fob][\tilde Z_+,\tilde Z_-,t]/ (\tilde Z_+
\tilde Z_-- z^{-m_v}f_\fob t^{\kappa_{\ul\rho}},t^{k+1})$ by
substituting $z=\tilde Z_+$, $w= \tilde Z_-/(z^{-m_v}f_\fob)$. Denote by
$\tilde\zeta\in\Lambda_\sigma$ the exponent with $\tilde Z_+=c_+
z^{\tilde\zeta}$ for some $c_+\in\CC^*$ as discussed in Construction~\ref{Affine
structure on B minus A} and Remark~\ref{Rem: tilde Z_+ tilde Z_-}. Let
$e_1,\ldots,e_{n-1}\in\Lambda_\rho$ be such that $e_1,\ldots,e_{n-1},
\tilde\zeta\in\Lambda_\sigma$ is an oriented basis. Differing from the choice in
Construction~\ref{Constr: beta} and Lemma~\ref{Lem: charts for cycle}, we now
take $\tilde\zeta$ as the last element of the basis to turn our cycles into the
form required in Appendix~\ref{Sect: Period integrals}. In these coordinates, the
logarithmic $n$-form $\Omega$ reads
\[
\Omega =  \dlog z_1\wedge\ldots\wedge\dlog z_{n-1}\wedge \dlog \tilde Z_+
= -\dlog z_1\wedge\ldots\wedge\dlog z_{n-1}\wedge \dlog\tilde Z_-,
\]
and hence, since $f_\fob$ does not depend on $\tilde Z_+,\tilde Z_-$,
\[
\Phi_i^+(\Omega) = \dlog z_1\wedge\ldots\wedge\dlog z_{n-1}\wedge \dlog z
= -\dlog z_1\wedge\ldots\wedge\dlog z_{n-1}\wedge \dlog w.
\]
In the notation of \eqref{Eqn: Type II contribution}, all the coefficients $g_r, h_r$ of the
Laurent expansion vanish and we have
\[
\Phi_i^+(\Omega)=0,\quad \res_{\Phi_i}(\Omega)= \dlog z_1\wedge\ldots\wedge\dlog z_{n-1}.
\]
Recall from Lemma~\ref{Lem: charts for cycle} that $\beta_e= \ol\beta_e$, while
$\beta_{e'}=\ol\beta_{e'}-\Gamma_{e'}$ is homologous relative to its boundary to
the chain $\hat\beta_{e'}$ defined in \eqref{Eqn: hat beta_{e'}}. Let
us first assume $\xi_e\not\in\Lambda_\rho$. Applying Formula~\eqref{Def:
finite order integrals-beta_i} then gives
\begin{equation}
\label{Eqn: preliminary integral beta_i chart II}
\int_{\beta_i}\Omega = (-1)^{n-1}\Big(\int_{T_e}
\dlog z_1\wedge\ldots\wedge \dlog z_{n-1}\Big)
\big( \kappa_{\ul\rho} \log t -\log b-\log a \big).
\end{equation}
Here the terms with $b=w\big(S(v_+)\big)$ and $a=z\big(S(v_-)\big)$ adjust for
$z(S(e))$ and $w(S(e'))$ to be curves not starting or ending at $1$, see
Remark~\ref{Rem: adjustments for non-standard starting points}. The factor
$(-1)^{n-1}$ comes from the fact that we oriented $\beta_i$ as $S(e)\times T_e$
rather than as $T_e\times S(e)$ as done in the appendix. With
$\xi_e\not\in\Lambda_\rho$ we have $\xi_e=\ol\xi_e=\pm\tilde\zeta$ by
($\beta_\mathrm{III}$) in Assumption~\ref{Ass: beta}. Hence, up to orientation,
$T_e$ acts as the diagonal $U(1)^{n-1}$ on $(z_1,\ldots,z_{n-1})$. Thus by
Lemma~\ref{lemma-intalpha} in dimension $n-1$, the integral over $T_e$ equals
$\pm(2\pi\sqrt{-1})^{n-1}$. To determine the sign recall that we oriented $T_e$
from an adapted oriented basis of $\Lambda_\sigma$ with first element
$\ol\xi_e=\xi_e$. Placing $\ol\xi_e$ at the last place rather than the first
changes the orientation of $T_e$ by $(-1)^{n-1}$, canceling the sign factor in
\eqref{Eqn: preliminary integral beta_i chart II}. Finally,
$\xi_e=\pm\tilde\zeta$ with the sign positive iff $\xi_e$ points into the
maximal cell $\sigma$ containing $e$. Thus denoting by
$\check d_e\in\check\Lambda_\sigma$ the generator of $\Lambda_\rho^\perp\simeq
\ZZ$ with $\langle \check d_e,\tilde\zeta\rangle=1$ we have
\begin{equation}
\label{Eqn: tilde zeta versus xi_e}
\xi_e= \ol\xi_e= \langle \check d_e,\ol\xi_e\rangle \cdot\tilde\zeta.
\end{equation}
With this discussion, \eqref{Eqn: preliminary integral beta_i chart II} yields the following:
\begin{equation}
\label{Eqn: integral over edges of type II, v1}
\int_{\beta_i}\Omega= (2\pi\sqrt{-1})^{n-1} \langle\check d_e,\xi_e\rangle
\Big(\kappa_{\ul\rho} \log t -\log w\big(S(v_+)\big) - \log z\big(S(v_-)\big)\Big).
\end{equation}
The coordinate $z=\tilde Z_+$ maps to $s_{\sigma\ul\rho}
(\tilde\zeta)z^{\tilde\zeta}$ under the generization map $R^k_{\fob}\to
R^k_\fou$. Using \eqref{Eqn: tilde zeta versus xi_e} we can thus rewrite \eqref{Eqn: integral
over edges of type II, v1} for later use as
\begin{equation}
\label{Eqn: integral over edges of type II}
\begin{aligned}
\frac{1}{(2\pi\sqrt{-1})^{n-1}}\int_{\beta_i}\Omega
=&  \langle\check d_e,\xi_e\rangle\Big(\kappa_{\ul\rho} \log t -\log w\big(S(v_+)\big)\Big)\\
&- \log s_{\sigma\ul\rho}(\xi_e)-\log z^{\xi_e}\big(S(v_-)\big).
\end{aligned}
\end{equation}
In the other case, $\xi\in\Lambda_\rho$, our chain $\beta_i$ is of type (ii) in
(Cy~II) of Construction~\ref{Constr: finite order periods} and
$\int_{\beta_i}\Omega=0$ by Formula~\eqref{Def: finite order integrals-beta_i}.
Thus \eqref{Eqn: integral over edges of type II} also holds in this case because
$\langle \check d_e,\xi_e\rangle=0$.


\subsection{Integration over a slab add-in \texorpdfstring{$\beta_i=\Gamma_{e'}$}{bi=Ge'}}
\label{section-slab add-in-integral}

Let $v$ be a vertex of $\beta_\trop$ mapping to a slab $\fob\subseteq\ul\rho$,
with adjacent edges $e,e'$. For the following computation we adopt the notation
of the construction of a slab add-in $\Gamma_{e'}$ in Step~IV of the proof of
Lemma~\ref{Lem: charts for cycle}. Formula~\ref{Eqn: parametrization of slab
add-in} gives the parametrization of $\Gamma_{e'}$ with respect to coordinates
$\hat z_1,\ldots,\hat z_n$ of $\Hom(\Lambda_{\sigma'},\CC^*)$, the reduction modulo
$t$ of the relevant chart $\Spec R^k_{\fou'}$, where $\fou'$ is the chamber containing
the image of $e'$:
\[
[0,1]\times T_{e'} \lra \Hom(\Lambda_\sigma,\CC^*),\quad
(\theta,s)\longmapsto \big(\gamma(s,\theta)\cdot \theta_1 b_1/c_1,
\theta_2 a_2,\ldots,\theta_n a_n\big).
\]
The map $\gamma:[0,1]\times T_{e'}\to \CC^*$ is a differentiable homotopy
with
\begin{equation}
\label{Eqn: reminder homotopy}
\gamma(0,\theta)=(\hat f_{\ul\rho}/\hat z^{m_v})(\theta a),\quad
\gamma(1,\theta)=a_1c_1/b_1.
\end{equation}
Since the chart $\Phi_i$ for $\beta_i$ is of type (Ch~I) of
Construction~\ref{Def: finite order integrals-beta_i}, the first case of
Definition~\ref{Def: Phi^+} gives
\[
\Phi_i^+\Omega= \dlog \hat z_1\wedge \ldots\wedge\dlog \hat z_n.
\]
If $\xi_e\in\Lambda_\rho$, the period integral $\int_{[0,1]\times
T_{e'}}\Phi_i^+\Omega$ involves two-dimensional integrals over level sets of
$z_2,\ldots,z_n$ and hence vanishes, see Figure~\ref{Fig: slab add-in}. In the
other case $\xi_e\not\in\Lambda_\rho$, the torus $T_{e'}$ acts trivially on
$z_1$ and the restriction map \[ T_{e'}\lra U(1)^{n-1},\quad \theta\longmapsto
(\theta_2,\ldots,\theta_n)
\]
is an isomorphism. Since $\hat z_1= z^{\tilde\zeta'}$, this isomorphism is
orientation preserving if $\xi_e$ points into the same maximal cell as
$\tilde\zeta'$, and then $\tilde\zeta'=\xi_e$. In terms of $\check
d_e\in\check\Lambda_\sigma$ used in \eqref{Eqn: tilde zeta versus xi_e} this is
the case if and only if $\langle \check d_e,\xi_e\rangle =-1$. We can now
compute
\begin{equation}
\label{Eqn: Integral over slab add-in, v1}
\begin{aligned}
\lefteqn{\int_{\Gamma_{e'}} \Omega\ =\ 
\int_{[0,1]\times T_{e'}} \Phi_i^+(\Omega)}\hspace{4ex}\\
&= -\langle \check d_e,\xi_e\rangle \int_{[0,1]\times U(1)^{n-1}} \partial_s\log \gamma(s,\theta)\, ds\wedge \dlog\theta_2\wedge\ldots\wedge \dlog\theta_n\\
&= -\langle \check d_e,\xi_e\rangle \int_{ U(1)^{n-1}}
\big(\log \gamma(1,\theta)- \log\gamma(0,\theta)\big)\,\dlog\theta_2\wedge\ldots\wedge \dlog\theta_n\\
&= -\langle \check d_e,\xi_e\rangle(2\pi\sqrt{-1})^{n-1}
\big(\log (a_1c_1/b_1) -\shR\big(z^{-m_v}f_{\ul\rho},v)\big).
\end{aligned}
\end{equation}
As in \eqref{Eqn: integral over edges of type II} the sign $-\langle \check
d_e,\xi_e\rangle$ adjusts the orientation of $T_{e'}$ with the orientation of
$U(1)^{n-1}$. Note that this factor renders the formula also correct in the case
$\xi_e\in\Lambda_\rho$. The last equality follows from \eqref{Eqn: reminder
homotopy} and the definition of the complex Ronkin function in \eqref{Eqn: complex
Ronkin fct}, see \S\ref{Sub: Ronkin fct}. The term $a_1c_1/b_1\in\CC^*$ is the
constant endpoint of the homotopy $\gamma$ defined in \eqref{Eqn: homotopy
gamma}. In the notation used there, provided $\xi_e\not\in\Lambda_\rho$, we have
\[
a_1=\hat z_1(S(v_+))= z^{-\langle \check d_e,\xi_e\rangle \xi_{e'}}\big(S(v_+)\big),\ \ 
b_1=w\big(S(v_+)\big),\ \ 
c_1=s_{\sigma'\ul\rho}(\tilde\zeta')= s_{\sigma'\ul\rho}(\xi_{e'})^{-\langle \check d_e,\xi_e\rangle}.
\]
Thus we can write \eqref{Eqn: Integral over slab add-in, v1} more intrinsically as
\begin{equation}
\label{Eqn: Integral over slab add-in}
\begin{aligned}
\frac{1}{(2\pi\sqrt{-1})^{n-1}}\int_{\Gamma_{e'}} \Omega
=&\langle \check d_e,\xi_e\rangle \Big(\shR(z^{-m_v}f_{\ul\rho},v) +\log w\big(S(v_+)\big)\Big)\\
&+\log z^{\xi_{e'}}\big(S(v_+)\big) +\log s_{\sigma'\ul\rho}(\xi_{e'})
\end{aligned}
\end{equation}
Note that this formula also holds if $\xi_e\in\Lambda_\rho$ and that the term
$w\big(S(v_+)\big)$ appears in \eqref{Eqn: integral over edges of type II} with
opposite sign.


\subsection{Interpolation between charts}
\label{section-wall-integral}

There are two cases where we work with different charts at a vertex $v$ of
$\beta$. First, if $v$ lies on a wall $\fop$ and the two edges $e,e'$ adjacent
to $v$ according to $(\beta_{\rm II})$ are contained in different chambers
$\fou$, $\fou'$. Second, if $v$ is adjacent to an edge intersecting a slab. In
these cases there is a potentially non-trivial contribution of the interpolation
term $\int_{[0,1]\times\gamma_i^\mu} \Phi_{ij}^+(\Omega)$ in \eqref{Def: finite
order integrals}. In all other cases, intersecting chains $\beta_i$ and
$\beta_j$ lie in the interior of the same chamber and hence $\Phi_i=\Phi_j$. We
now determine the contribution of the interpolation term in the remaining cases.

Let us first treat the case that $v$ lies on a wall $\fop$ separating chambers
$\fou$, $\fou'$. Let $\Phi_i :\tilde U_i\to X_k^\circ$, $\Phi_j: \tilde U_j\to X_k^\circ$ be
the charts for the adjacent edges $e\subseteq \fou$, $e'\subseteq\fou'$,
respectively, as defined in the proof of Lemma~\ref{Lem: charts for cycle}. Then
$\tilde U_i=\tilde U_j=\Spec \CC[t]/(t^{k+1})[\Lambda_\sigma]$ and
$\Phi_j=\Phi_i\circ \Psi_{ij}$ with $\Psi_{ij}$ defined by the wall crossing
isomorphism
\begin{equation}
\label{Eqn: Wall crossing formula}
\theta_{\fop}: R^k_{\fou}\lra R^k_{\fou'},\quad
z^m\longmapsto f_\fop^{\langle \check d_\fop, m\rangle} z^m.
\end{equation}
Here $\check d_\fop\in\check\Lambda_\sigma$ is the generator of
$\Lambda_\fop^\perp\simeq\ZZ$ evaluating positively on tangent vectors pointing
from $\fop$ into $\fou$. Writing $f_\fop=1+t\tilde f_\fop$, the homotopy
$\Phi_{ij}: [0,1]\times U_i\times O_k\to X_k^\circ$ between $\Phi_i$ and
$\Phi_j$ of \eqref{Eqn: homotopy Phi_{ij}} can be defined by the family of
$\CC$-algebra homomorphisms
\[
\theta_{\fop}(s): R^k_{\fou}\lra R^k_{\fou'},\quad
z^m\longmapsto (1+st\tilde f_\fop)^{\langle \check d_\fop, m\rangle} z^m,
\]
$s\in[0,1]$. Let $e_1,\ldots,e_n$ be an oriented basis of $\Lambda_\sigma$ with
$\langle\check d_\fop, e_1\rangle =1$ and $e_2,\ldots,e_n$ spanning
$\Lambda_\fop$. Then in the corresponding coordinates $z_1,\ldots,z_n$, the
function $\tilde f_\fop$ does not depend on $z_1$, while
$\theta_\fop(s)(z_1)= (1+st\tilde f_\fop)z_1$ and
$\theta_\fop(s)(z_\mu)=z_\mu$ for $\mu=2,\ldots,n$. Hence
\begin{equation}\label{Eqn: Phi_{ij}^+(Omega)} \Phi_{ij}^+(\Omega) = \big(\dlog
z_1 + \partial_s \log(1+st\tilde f_\fop)ds\big)\wedge \dlog
z_2\wedge\ldots\wedge\dlog z_n.
\end{equation}
With $a_i=z_i(S(v))$ the coordinates of $S(e)$ over $v$, integrating
out $s$, we obtain
\begin{eqnarray}
\label{Eqn: wall interpolation}
\lefteqn{\int_{[0,1]\times T_e} \Phi_{ij}^+(\Omega)\ =\ 
\int_{T_e} \log\big(1+t\tilde f_\fop(t,z_2,\ldots,z_n)\big)
\dlog z_2\wedge\ldots\wedge\dlog z_n}\hspace{2ex}\\ \nonumber
&=& \int_{T_e} \log \big(1+t\tilde f_\fop
(t,\theta_2 a_2,\ldots, \theta_n a_n) \big)
\dlog\theta_2\wedge\ldots\wedge \dlog\theta_n
\end{eqnarray}
Expanding the logarithm yields a finite sum of constant multiples of
$t^\ell\theta_2^{\ell_2}\ldots\theta_n^{\ell _n}$ with
$\ell,\ell_2,\ldots,\ell_n\in\NN$ and $\ell_\mu>0$ for at least one $\mu$.
If $\xi_e\in\Lambda_\fop$, then similar to the situation along codimension one
cells discussed in Construction~\ref{Constr: beta}, the action of $T_e$ on
$(z_2,\ldots,z_n)$ has a kernel and the integral in \eqref{Eqn: wall
interpolation} vanishes for trivial reasons. In the other case
$\xi_e\not\in\Lambda_\rho$, the torus $T_e$ acts on $(z_2,\ldots,z_n)$ via a
finite covering $T_e\to U(1)^{n-1}$ and the integral vanishes because
$\int_{S^1} \theta_\mu^{\ell_\mu}\dlog\theta_\mu=0$ for an index $\mu$ with
$\ell_\mu\neq0$. Hence in any case, there is no interpolation contribution from
changing chambers at walls.
\medskip

In the second case, $\beta_i=\beta_e+\beta_{e'}$ maps to a chart of type~(Ch~II). Let
$e,e'$ map to chambers $\fou,\fou'$ separated by the slab $\fob$ containing the
common vertex $v$ of $e$ and $e'$. As in the construction of $\beta_i$ in
Lemma~\ref{Lem: charts for cycle}, assume $\beta$ is oriented from $e$ to $e'$
and hence $e'$ attaches to the non-trivial slab add-in $\Gamma_{e'}$. Then
$\beta_e$ was constructed with the toric coordinate $z=\tilde Z_+$ and the chart
$\tilde U_i$ is compatible with $R^k_\fou$ in that the localization map
$R^k_\fob\to R^k_\fou$ is toric.\footnote{With non-trivial gluing data the
localization map $R^k_\fob\to R^k_\fou$ identifies monomials only up to scale,
but for this argument it only matters that the map commutes with the torus
action.} In particular, both charts provide the same local product decomposition
with respect to $t$ and hence the change of coordinates map $\Psi_{ij}$ in
Construction~\ref{Constr: finite order periods} is the identity. Thus there is
also no interpolation contribution from this boundary of $\beta_i$.

The interesting change of coordinates happens between $\beta_{e'}$ and the slab
add-in $\Gamma_{e'}$. According to Construction~\ref{Constr: finite order
periods} we need to interpolate between the chart $\Phi_i:\tilde U_i\to
X_k^\circ$, modeled on $\CC[\Lambda_\rho][z,w,t](zw-t^{\kappa_i}, t^{k+1})$ and
used for $\beta_{e'}$, and the chart $\Phi_j:\tilde U_j\to X_k^\circ$, modeled
on $R^k_{\fou'}$ and used for $\Gamma_{e'}$. In \eqref{Eqn: hatted versus
unhatted coords} this change of coordinates has already been made explicit, by
using toric coordinates $\hat z_1,\ldots,\hat z_n$ for $\tilde U_j$ and
$w,z_2,\ldots,z_n$ for $\tilde U_i\setminus (w=0)$. In the notation of
Construction~\ref{Constr: finite order periods} and of \eqref{Eqn: hatted versus
unhatted coords}\eqref{Eqn: hat z_1 in terms of w}, the pull-back by
$\Psi_{ij}:U_{ij}\times O_k\to U_{ij}\times O_k$ is the map
\[
w\longmapsto c_1 \hat z_1/(\hat z^{-m_v}\hat f_\fob) ,\quad
z_2\longmapsto c_2 \hat z_2,\,\ldots,\,z_n\longmapsto c_n \hat z_n,
\]
while the map denoted ``$\id$'' in the appendix has the same form, but with
$\hat f_\fob$ replaced by the reduction $f_{\ul\rho}$ modulo $t$. Indeed,
``$\id$'' is defined as the map $U_{ij}\times O_k\to
U_{ij}\times O_k$ induced by the identity map of $U_{ij}$ as a subset of $X_0^\circ$
and extended by the product structure in the charts $\tilde U_i$ and $\tilde
U_j$, respectively. Writing $\hat f_\fob=\hat f_{\ul\rho}+t g_\fob$, define for
$s\in[0,1]$,
\[
\hat f_\fob(s)= \hat f_{\ul\rho}+st g_\fob.
\]
Then the family of maps $\Psi_{ij}(s)$, $s\in[0,1]$, defined by
\begin{equation}
\label{Eqn: Psi_ij(s) on ring level}
w\longmapsto c_1 \hat z_1/(\hat z^{-m_v}\hat f_\fob(s)) ,\quad
z_2\longmapsto c_2 \hat z_2,\,\ldots,\,z_n\longmapsto c_n \hat z_n
\end{equation}
is a homotopy connecting $\id$ to $\Psi_{ij}$. Thus we can take
$\Phi_{ij}(s)=\Phi_i^{(j)}\circ\Psi_{ij}(s)$ as the homotopy between the two
restrictions of charts $\Phi_i^{(j)}=\Phi_i|_{U_{ij}\times O_k}$ and
$\Phi_j^{(i)}=\Phi_j|_{U_{ij}\times O_k}$. Since $\hat z_1,\ldots,\hat z_n$ was
defined by an oriented basis $e_1=\tilde\zeta',e_2,\ldots,e_n$ of
$\Lambda_{\sigma'}$, replacing $\hat z_1$ by $w= c_1\hat z_1/(\hat z^{-m_v}\hat
f_\fob)$ shows $\Phi_i^+(\Omega)= \dlog
w\wedge\dlog z_2\wedge\ldots\wedge \dlog z_n$. Pulling back by \eqref{Eqn:
Psi_ij(s) on ring level} then gives
\[
\Phi_{ij}^+(\Omega)= \big(\dlog \hat z_1 - \partial_s \log(\hat z^{-m_v} \hat f_\fob(s))\,ds\big)
\wedge \dlog \hat z_2\wedge\ldots\wedge \dlog \hat z_n.
\]
Similar to \eqref{Eqn: Integral over slab add-in}, the interpolation contribution to the
period integral is now computed as
\begin{eqnarray}
\label{Eqn: interpolation contribution}
\hspace{6ex}\lefteqn{\int_{[0,1]\times T_{e'}}\Phi_{ij}^+(\Omega)\ =\ 
\int_{T_{e'}}\big(-\log(\hat z^{-m_v}\hat f_\fob)
+\log(\hat z^{-m_v}\hat f_{\ul\rho} )\big)\,
\dlog \hat z_2\wedge\ldots\wedge\dlog\hat z_n}\nonumber\\ 
&=&\langle \check d_e,\xi_e\rangle \int_{ U(1)^{n-1}}
\big(\log(\hat z^{-m_v}\hat f_\fob) -\log(\hat z^{-m_v}\hat f_{\ul\rho} )\big)\,
\dlog\theta_2\wedge\ldots\wedge \dlog\theta_n \\
&=& \langle \check d_e,\xi_e\rangle (2\pi\sqrt{-1})^{n-1}
\big(\shR(z^{-m_v}f_\fob,v) - \shR(z^{-m_v}f_{\ul\rho},v)\big). \nonumber
\end{eqnarray}
As in \eqref{Eqn: integral over edges of type II}, a factor $-\langle \check
d_e,\xi_e\rangle$ was inserted for the second equality to adjust for the
orientation of $T_{e'}$ and for the non-trivial kernel of the map to
$U(1)^{n-1}$ in case $\xi_e\in\Lambda_\rho$, respectively.

Note that the Ronkin function for $z^{-m_v}f_{\ul\rho}$
in this result cancels with the contribution \eqref{Eqn: Integral over slab
add-in} from the slab add-in, thus only leaving the Ronkin function for
$z^{-m_v}f_\fob$ to contribute to the global period integral.


\subsection{Proof of Theorem~\ref{Thm: period thm}}
\label{Subsect: proof of period thm}

To compute $\frac{1}{(2\pi\sqrt{-1})^{n-1}}\int_\beta\Omega$, it remains to take
the sum over all the computed terms. We had contributions from $\beta_e$ for
edges disjoint from slabs \eqref{Eqn: integral over edges of type I}, from
$\Gamma_v$ for a vertex of higher valency \eqref{Eqn: Integral over Gamma_v},
from $\beta_i=\beta_e+\beta_{e'}$ for pairs of edges crossing a slab \eqref{Eqn:
integral over edges of type II}, from slab add-ins $\Gamma_{e'}$ \eqref{Eqn:
Integral over slab add-in}, and from interpolation terms \eqref{Eqn:
interpolation contribution}. Note that in view of Lemma~\ref{lemma-intalpha} and
Proposition~\ref{Prop: fiber integral well-defined} the result is only
well-defined up to adding integral multiples of $2\pi\sqrt{-1}$.

First, for a vertex of valency $\val(v)\ge 3$ the chain $\Gamma_v$ contributes
$\val(v)\cdot \pi\sqrt{-1}$ up to adding integral multiples of $2\pi\sqrt{-1}$. But a
graph without one-valent vertices can be built inductively by successively
connecting two vertices (possibly equal) by an edge. Each such addition
increases $\sum_v \val(v)$ by $2$. Thus $\sum_v \int_{\Gamma_v}\Omega$ is a
multiple of $(2\pi\sqrt{-1})^n$ and hence can be omitted.

The other terms are easiest to gather according to the types of vertices. For a
vertex $v$ in the interior of a maximal cell $\sigma$ and each edge $e$ with
vertex $v$ we have a contribution $\pm \log z^{\xi_e}(S(v))$ from \eqref{Eqn:
integral over edges of type I}, \eqref{Eqn: integral over edges of type II} or
\eqref{Eqn: Integral over slab add-in}. The sign $\varepsilon_{e,v}=1$ is
positive if $e$ is oriented towards $v$ and $\varepsilon_{e,v}=-1$ otherwise. By
the balancing condition \eqref{balancing}, the sum over all these terms vanishes:
\[
\sum_{e\ni v} \log z^{\eps_{e,v}\xi_e}= \log z^{\sum_{e\ni v} \eps_{e,v}\xi_e} =0.
\]

Collecting the remaining terms now gives
\begin{equation}
\label{Eqn: Putting together}
\begin{aligned}
\lefteqn{\frac{1}{(2\pi\sqrt{-1})^{n-1}} \int_\beta \Omega}\hspace{6ex}\\
&= \sum_v \Big(\langle \check d_e,\xi_e\rangle \shR(z^{-m_v}f_\fob,v)
+\log \frac{s_{\sigma'\ul\rho}(\xi_{e'})}{s_{\sigma\ul\rho}(\xi_{e})}
+\langle \check d_e,\xi_e \rangle\cdot \kappa_{\ul\rho}\log t\Big).
\end{aligned}
\end{equation}
The sum runs over all vertices $v$ of $\beta_\trop$ mapping to a slab, and in
the sum $e,e'$, $\fob$, $\ul\rho$ denote the corresponding incoming and outgoing
edges, the slab containing $v$ and the corresponding
codimension one cell of the barycentric subdivision, respectively. The sum over
the terms containing the gluing data gives $\log \langle s,\beta_\trop\rangle$,
while the sum involving $\kappa_{\ul\rho}$ yields
\[
\sum  \langle \check d_e,\xi_e \rangle\cdot \kappa_{\ul\rho}
=\langle c_1(\varphi),\beta_\trop\rangle.
\]
Thus \eqref{Eqn: Putting together} can be written more intrinsically as
\begin{equation}
\label{Eqn: Final period integral}
\begin{aligned}
\lefteqn{\frac{1}{(2\pi\sqrt{-1})^{n-1}} \int_\beta \Omega}\hspace{6ex}\\
&= \log\langle s,\beta_\trop\rangle +\langle c_1(\varphi),\beta_\trop\rangle\cdot\log t
+\sum_v \langle \check d_e,\xi_e\rangle \shR(z^{-m_v}f_\fob,v).
\end{aligned}
\end{equation}
Exponentiating finally gives the expression for
$\exp\big((2\pi\sqrt{-1})^{-(n-1)}\int_\beta\Omega \big)$ claimed in Theorem~\ref{Thm:
period thm}.

%
%

\section{Analyticity of formal toric degenerations}
\label{Sect: analyticity}

As an application of the period computations we prove analyticity of the
canonical toric degenerations constructed in \cite{affinecomplex} in the case
that $(B,\P)$ as simple singularities. Simple singularities are locally
indecomposable from the affine geometric point of view and they give rise to locally
rigid logarithmic singularities. We won't need any details of simple singularities in this paper and refer to \cite{logmirror1}, Definition~1.60 for the formal
definition and to \cite{logmirror2}, \S2.2, for the local algebraic description
and deformation theory. For $(B,\P)$ with simple singularities and a choice of
multivalued, strictly convex piecewise affine function $\varphi$ on $B$, it has
been shown in \cite{affinecomplex} and \cite{theta}, Theorem~A.2, that there is
a canonical toric degeneration
\begin{equation}\label{Eqn: GS family}
\foX\lra \foS=\Spf \big(A\lfor t\rfor\big).
\end{equation}
Here $A$ is a Laurent polynomial ring, so the base of this family is the product
of an algebraic torus with $\Spf\big(\CC\lfor t\rfor\big)$.\footnote{Assuming
projectivity of the central fiber, Theorem~A.2 of \cite{theta} constructs a
projective scheme over a closed subspace $\Spec \big(A_\PP\lfor
t\rfor\big)\subseteq\Spec\big(A\lfor t\rfor\big)$. Our analyticity holds more
generally in the formal setup, only requiring properness of the map in
\eqref{Eqn: GS family}.} If $\partial B\neq\emptyset$, by \cite[Remark~2.18 and
Remark~4.13]{theta}, the family in \eqref{Eqn: GS family} comes equipped with a
divisor $\foD\subset \foX$ that is flat over $\foS$.

To describe the ring $A$, recall from \cite{logmirror1}, Theorem 5.4, that for
simple singularities, the affine cohomology group\footnote{Note that here we
have $\iota_*\check\Lambda$ rather than $\iota_*\Lambda$ as in loc.cit.\ because
we work in the cone picture rather than in the fan picture, that is, for us a
polyhedron $\tau\in\P$ indexes a closed stratum of $X_0$ isomorphic to the toric
variety with momentum polytope $\tau$.} $H^1(B, \iota_*\check\Lambda
\otimes\CC^*)$ is canonically in bijection with the set of isomorphism classes
of log schemes $(X_0,\M_{X_0})$ over the standard log point with associated
discrete data $(B,\P,\varphi)$. The bijection works by identifying the set of
isomorphism classes of log schemes with the set of equivalence classes of
lifted, normalized gluing data, which in turn can be identified with
the mentioned affine cohomology group. The base ring is
\[
A=\CC[H^1(B,\iota_*\check\Lambda)^*],
\]
the Laurent polynomial ring over $H^1(B,\iota_*\check\Lambda)^*=
\Hom(H^1(B,\iota_*\check\Lambda),\ZZ)$. Thus $\Spec A$ parametrizes choices of
lifted, normalized gluing data.

The construction of the family \eqref{Eqn: GS family} depends on the choice of a
splitting $\sigma_0$ of the quotient map $q_f: H^1(B,\iota_*\check\Lambda)\to
H^1(B,\iota_*\check\Lambda)_f$ by the torsion submodule
$H^1(B,\iota_*\check\Lambda)_t\subseteq H^1(B,\iota_*\check\Lambda)$. Here
$\iota:B_0\to B$ is the inclusion of the regular locus and
$\check\Lambda=\shHom(\Lambda,\ul\ZZ)$ is the sheaf of integral cotangent
vectors on $B_0$.\footnote{In the case with simple singularities, the singular
locus can be taken to be the union of the $(n-2)$-cells of the barycentric
subdivision of $\P$ not containing barycenters of vertices or of maximal cells.}
Such a splitting $\sigma_0$ is unique only up to a homomorphism
$H^1(B,\iota_*\check\Lambda)_f\to H^1(B,\iota_*\check\Lambda)_t$. Thus if
$H^1(B,\iota_*\check\Lambda)$ has non-trivial torsion, there are finitely many
such canonical families. We fix $\sigma_0$ and the resulting canonical family
throughout this Chapter.

If $H^2(B,\iota_*\check\Lambda)$ has torsion, the set of gluing data
$H^1(B,\iota_*\check\Lambda\otimes\CC^*)$ is a disjoint union of torsors for
$H^1(B,\iota_*\check\Lambda)\otimes\CC^*$, with only one of them containing
trivial gluing data. Indeed, the construction of the family also depends on the
choice of a possibly non-trivial element $s_0\in
H^1(B,\iota_*\check\Lambda\otimes\CC^*)$, which selects one of these torsors. If
$H^2(B,\iota_*\check\Lambda)$ is torsion-free, trivial gluing data $s_0=1$ is a
canonical choice. In any case, we fix $s_0$ throughout. 

As a further ingredient in this section, recall from \cite{logmirror1},
Definition~1.45 (using the notation of \cite{theta}, \S A.2) that the short
exact sequence
\begin{equation}\label{Eqn: MPA sequence}
0\lra \iota_*\check\Lambda\lra \breve\shPL(B)\lra \breve\shMPA(B)\lra 0,
\end{equation}
gives rise to the connecting homomorphism
\[
c_1:\breve\MPA(B)\lra H^1(B,\iota_*\check\Lambda).
\]
This homomorphism sends a multivalued piecewise affine function $\varphi$ to
its characteristic class $c_1(\varphi)$. Dually, we have
\begin{equation}\label{Eqn: c_1^*}
c_1^*: H^1(B,\iota_*\check\Lambda)^*\lra \breve\MPA(B)^*.
\end{equation}
The trace homomorphism $\iota_*\Lambda\otimes\iota_*\check\Lambda\to \ul\ZZ$
combined with the sheaf homology-cohomology pairing gives a bilinear map
\begin{equation}
\label{Eqn: Pairing}
\langle\ ,\,\rangle: H^1(B,\iota_*\check\Lambda)\otimes H_1(B,\iota_*\Lambda)\lra \ZZ.
\end{equation}
The induced homomorphism
\begin{equation}\label{Eqn: H_1->(H^1)^*}
H_1(B,\iota_*\Lambda)\lra H^1(B,\iota_*\check\Lambda)^*,\quad
\beta_\trop \longmapsto \beta_\trop^*,
\end{equation}
is an isomorphism over $\QQ$ by the following result from \cite{affinecoh}, Theorem 2.

\begin{theorem} 
\label{Thm: perfect pairing}
Let $(B,\P)$ be an oriented simple tropical manifold. Then \eqref{Eqn: Pairing}
tensored with $\QQ$ is a perfect pairing of $\QQ$-vector spaces.
\end{theorem}

The composition of $c_1^*$ from \eqref{Eqn: c_1^*} with the map from \eqref{Eqn:
H_1->(H^1)^*} and evaluation on $\varphi$ yields the homomorphism
\begin{equation}\label{Eqn: Picard-Lefschetz monodromy}
H_1(B,\iota_*\Lambda)\lra \ZZ,\quad
\beta_\trop\longmapsto\langle c_1(\varphi),\beta_\trop\rangle,
\end{equation}
with $\langle c_1(\varphi),\beta_\trop\rangle$ given explicitly after
\eqref{Eqn: pairing with c_1}. By Proposition~\ref{Lem: monodromy action}, this
map measures the monodromy of the $n$-cycle associated to $\beta_\trop$ in
the base space of the universal family about $t=0$. Denote by
\begin{equation}
\label{Eqn: (H_1)_+}
H_1(B,\iota_*\Lambda)_+\subseteq H_1(B,\iota_*\Lambda)
\end{equation}
the preimage of $\NN\subset \ZZ$ under \eqref{Eqn: Picard-Lefschetz monodromy}.
If $c_1(\varphi)\neq0$, this subset is a half-space and in any case,
$H_1(B,\iota_*\Lambda)_+$ spans $H_1(B,\iota_*\Lambda)$. If $B$ is compact
without boundary, $c_1(\varphi)\neq0$ holds always:

\begin{proposition}
\label{Prop: c_1(varphi) nonzero}
Let $(B,\P,\varphi)$ be a compact polarized affine manifold with singularities of the affine structure disjoint from the vertices of $\P$.
We have $c_1(\varphi)\neq 0$ in each of the following situations
\begin{enumerate}[label=(\roman*)]
  \item \label{c11} $H^1(B,\QQ)=0$ and $\partial B$ is again an affine manifold (including $\partial B=\emptyset$), 
  \item \label{c12} $(B,\P)$ is simple and $\partial B=\emptyset$.
\end{enumerate}
\end{proposition}

\begin{proof}
First assume $\ref{c11}$, so in particular $H^1(B,\ZZ)$ is torsion. For now,
assume that actually $H^1(B,\ZZ)=0$. By chasing the long exact cohomology
sequences for the third row and second column of the diagram in
\cite{logmirror1}, Definition~1.45, and taking into account $H^1(B,\ZZ)=0$, it
follows that (1)~$c_1(\varphi)\in H^1(B,\iota_*\check\Lambda)$ is the image of a
class $\tilde c_1(\varphi)\in H^1(B,\shAff(B,\ZZ))$ with $\shAff(B,\ZZ)$ the
sheaf of integral affine functions on $B$, (2)~$c_1(\varphi)=0$ implies $\tilde
c_1(\varphi)=0$ and (3)~if $\tilde c_1(\varphi)=0$ then $\varphi$ can be
represented by a piecewise-affine function. Thus under the assumptions, if
$c_1(\varphi)=0$ there is a piecewise affine function $\tilde\varphi$
representing $\varphi$. If, more generally, $H^1(B,\ZZ)$ has torsion we can
still run the same argument for some suitable multiple $k\varphi$ with $k>0$
which suffices for the reasoning in the next paragraph.

Since $B$ is compact there is a point in $B$ where $\varphi$ has maximal value,
and $\varphi$ being piecewise affine, this point can be taken to be a vertex. By
assumption, there is an affine chart near this vertex, yielding a strictly
convex, piecewise affine function on the fan defined by $\P$ in this chart. But
such a function cannot have a maximum at the origin since, by assumption
$\ref{c11}$, the origin is contained in the interior of a straight line segment.
Thus $c_1(\varphi)\neq0$.

Now assume $\ref{c12}$. The case $H^1(B,\QQ)=0$ is covered by $\ref{c11}$, so
assume $H^1(B,\QQ)\neq 0$. Since $h^{1,0}=h^{0,1}$ by the Hodge theory of the
formal nearby fiber using simplicity and \cite{logmirror2},
Theorem\,3.22 and Theorem\,4.2, we also get $H^0(B,\iota_*\Lambda)\neq
0$. Let $\xi\in H^0(B,\iota_*\Lambda)$ be non-trivial. Now assume
$c_1(\varphi)=0$. By a similar diagram chase as above, we obtain a section
$\hat\varphi\in H^0(B,\shP\shL/\ZZ)$. For each maximal cell $\sigma$,
denote by $\alpha_\sigma$ the cotangent vector defined by the slope of
$\hat\varphi|_\sigma$. Let $\sigma$ be a maximal cell with
$\nabla_\xi\hat\varphi=\langle \alpha_\sigma,\xi\rangle$ maximal. Then
$\sigma$ has a facet where $\xi$ is outward-pointing to another maximal cell
$\sigma'$ and the convexity of $\hat\varphi$ leads to the contradiction $\langle
\alpha_\sigma,\xi\rangle < \langle \alpha_{\sigma'},\xi\rangle$.
\end{proof}

\begin{remark}
If $(B,\P,\varphi)$ is a regular subdivision of a
lattice polytope, viewed as an integral affine manifold without singularities,
then $H^1(B,\iota_*\check\Lambda)=0$, so in particular $c_1(\varphi)=0$. We may
call this the \emph{purely toric case} and then the resulting family \eqref{Eqn:
GS family} is trivial away from $t=0$, so this case is not very interesting
anyway. However, if one additionally straightens the boundary of $B$ by trading
corners with affine singularities, Case~(i) of Proposition~\ref{Prop:
c_1(varphi) nonzero} then shows $c_1(\varphi)\neq 0$. While the family could
then still be trivial outside $t=0$, we expect that at least the divisor
$\foD\subset\foX$ varies non-trivially. The simplest example here is $\PP^2$
with $\foD$ a toric degeneration of elliptic curves --- the $j$-invariant of the
elliptic curve varies with $t$, see \cite{theta}, Example~6.2.
\end{remark}

From now on, we restrict to the case $c_1(\varphi)\neq0$. Here is the main
result of this section.

\begin{theorem}
\label{Thm: Analyticity of GS}
Let $(B,\P,\varphi)$ be a compact orientable polarized integral affine manifold
with simple singularities and $c_1(\varphi)\neq 0$ and either $\partial
B=\emptyset$ or $\partial B$ itself an affine manifold. Denote by $\foX\to
\foS=\Spf\big(A\lfor t\rfor\big)$ the associated canonical toric degeneration
from \eqref{Eqn: GS family}. Then for every closed point $x=(a,0)\in
\Spec\big(A[t]\big)$ there exists an open neighborhood $U\subset\Spec(A)_\an$
of $a$, a disk $\DD$ and a proper, flat analytic family
\[
\shY\lra U\times \DD,
\]
with completion at $x$ isomorphic over $\foS$ to the completion of $\X\to\foS$ at $x$.
\end{theorem}

\begin{remark}
In the non-orientable case one can take the orientable double cover and study
the $\ZZ/2$-quotient over the $\ZZ/2$-invariant locus in $\foS$.
\end{remark}

Combining this result with Theorem~\ref{Thm: period thm}, we obtain monomial
period integrals. To state this result, denote by $s^m\in A$ the Laurent
monomial associated to $m\in H^1(B,\iota_*\check\Lambda)^*$ as well as the
corresponding holomorphic function on $\Spec(A)_\an$ or on
$\Spec\big(A[t]\big)_\an$.

\begin{corollary}
\label{Cor: monomial analytic periods}
In the situation of Theorem~\ref{Thm: Analyticity of GS}, let $\beta_u\in
H_n(X_u,\ZZ)$, $u\in U\times \DD$, be a family of cycles in the fibers of $\shY\to U\times \DD$
constructed from a tropical cycle $\beta_\trop\in H_1(B,\iota_*\Lambda)$,
well-defined up to homology and up to adding multiples of the family
of vanishing cycles $\alpha_u$. Denote by $\Omega$ the relative holomorphic
$n$-form on $\shY$ with $\int_\alpha\Omega= (2\pi i)^n$. Then
\[
\exp\bigg(\frac{1}{(2\pi \sqrt{-1})^{n-1}}\int_{\beta_u} \Omega\bigg)
= s^{-\beta_\trop^*}\cdot t^{\langle c_1(\varphi),\beta_\trop\rangle},
\]
holds as an equality of meromorphic functions on $U\times \DD$, with $\beta_\trop^*$
introduced in \eqref{Eqn: H_1->(H^1)^*}. If $\beta_\trop\in
H_1(B,\iota_*\Lambda)_+$ then both sides are holomorphic.
\end{corollary}

\begin{proof}
The formula follows readily by applying
Proposition~\ref{holomorphic-periods-induce-formal-periods} and
Theorem~\ref{Thm: period thm} to the reductions modulo $t^{k+1}$ of $\shY\to
U\times \DD$ from Theorem~\ref{Thm: Analyticity of GS} and letting $k\to\infty$.
The term $\shR(\beta_\trop)$ does not appear since the criterion of
Proposition~\ref{Prop: Triviality of complex Ronkin function} holds for all slab
functions thanks to the normalization condition in the smoothing algorithm, see
\cite{affinecomplex}, \S3.6. The sign in $s^{-\beta_\trop^*}$ differs from the
sign in Theorem~\ref{Thm: period thm} due to opposite sign conventions in
\cite{theta} and \cite{affinecomplex}, as discussed in \cite{theta}, \S{A.1}.
\end{proof}

The proof of Theorem~\ref{Thm: Analyticity of GS} requires several preparations
and steps, which will be put together only at the end of this section.


\subsection{The \texorpdfstring{$\GG_m$}{Gm}-action on the canonical family}

Let $\beta_\trop$ be a tropical cycle with $\langle c_1(\varphi),
\beta_\trop\rangle=0$. Then Theorem~\ref{Thm: period thm} applied to the reduction modulo $t^{k+1}$ and taking $k\to\infty$ gives
\begin{equation}
\label{Eqn: Regular monomial periods}
\exp\bigg(\frac{1}{(2\pi \sqrt{-1})^{n-1}} \int_\beta \Omega\bigg)= s^{-\beta_\trop^*}
\end{equation}
for the corresponding period integral of $\foX\to\foS$. Thus such period
integrals produce the pull-back of a Laurent monomial in
$A=\CC[H_1(B,\iota_*\check\Lambda)^*]$ via the projection $\foS=\Spf\big(A\lfor
t\rfor\big)\to \Spec A$. The exponents of monomials thus obtained form the
sublattice
\begin{equation}
\label{Eqn: K^* Sublattice}
K^*=\big\{ \beta_\trop^*\in H^1(B,\iota_*\check\Lambda)^*\,\big|\, \beta_\trop\in
H_1(B,\iota_*\Lambda),\ \langle c_1(\varphi),\beta_\trop\rangle=0\big\}
\end{equation}
of $c_1(\varphi)^\perp\subset H^1(B,\iota_*\check\Lambda)^*$. Theorem~\ref{Thm:
perfect pairing} implies that $K^*\subseteq c_1(\varphi)^\perp$ has full rank.
Hence the period integrals of the form \eqref{Eqn: Regular monomial periods}
generate the coordinate ring of a finite quotient (isogenous) torus
$\Spec\big(\CC[K^*]\big)$ of $\Spec\big(\CC[c_1(\varphi)^\perp]\big)$.
Since $c_1(\varphi)\neq0$ by hypothesis, these tori have dimension one less than
$\dim A$. The explanation for the missing dimension is that the action of the
one-dimensional subtorus $\GG_m\subseteq\Spec (A)$ defined by $c_1(\varphi)\in
H^1(B,\iota_*\check\Lambda)$ extends to an action on $\foX\to\Spf\big(A\lfor
t\rfor\big)$, possibly after a finite base change:

\begin{proposition}
\label{Prop: foX mod TT_L}
There exists a finite index sublattice $H\oplus F\subseteq
H^1(B,\iota_*\check\Lambda)$ containing $c_1(\varphi)$ such that the pull-back
\[
\tilde\foX\to \tilde\foS=\Spf\big(\tilde A\lfor t\rfor\big)
\]
of $\foX\to\foS=\Spf\big(A\lfor t\rfor)$ by the induced isogeny of tori\,
$\Spec \tilde A\to \Spec A$, $\tilde A=\CC[H^*\oplus F^*]$, is equivariant for
a free $\GG_m$-action acting with weight $1$ on $t$. The $\GG_m$-action on
$\Spec \tilde A$ is defined by the $\ZZ$-grading given by evaluation
at $c_1(\varphi)$:
\[
\deg s^m= m\big(c_1(\varphi)\big),\quad m\in H^*\oplus F^*.
\]
\end{proposition}

\begin{proof}
The group action is defined in \cite{theta}, \S{A.3} with a universal choice of
piecewise linear function $\breve\varphi$, taking kinks in a universal monoid
$Q$. The monoid $Q$ is the toric monoid with $\Hom(Q,\ZZ)$ the group
$\breve\MPA(B)$ of multivalued piecewise affine functions on $B$ and such that
$\Hom(Q,\NN)$ is the submonoid of such functions with non-negative kinks. Our
piecewise affine function $\varphi$ is the composition of $\breve\varphi$ with a
homomorphism $h:Q\to\NN$. This universal point of view produces a canonical
family over $\Spf\big(A\lfor Q\rfor\big)$, and our family is
obtained\footnote{The universal construction also involves a choice of splitting
$\sigma_0$ of $q_f: H^1(B,\iota_*\check\Lambda)\to
H^1(B,\iota_*\check\Lambda)_f$. We assume the same $\sigma_0$ as in the construction of $\foX\to\Spf\big(A\lfor t\rfor\big)$ above has been
chosen.} by base change via the homomorphism of $\CC$-algebras
\begin{equation}
\label{Eqn: A[Q]->A[t]}
A\lfor Q\rfor \lra A\lfor t\rfor
\end{equation}
defined by $h$.

The group action in \cite{theta}, Proposition~A.13, has character lattice $L^*$
for $L\subseteq \breve\MPA(B)$ a complement to the kernel of
$c_1:\breve\MPA(B)\to H^1(B,\iota_*\check\Lambda)$, up to finite index. The
lattice $L$ has to be chosen in such a way that the isomorphic image
$H=c_1(L)\subseteq H^1(B,\iota_*\check\Lambda)$ lies in the image of the
splitting $\sigma_0$ (\cite{theta}, Lemma~A.11). Inspecting the construction of
$L$ in \cite{theta}, it is clear that we can also achieve
$\varphi\in L$. Going over to a sublattice, we may also assume that $\varphi$ is
primitive as an element of $L$. The construction in \cite{theta} then provides a
finite index sublattice $H\oplus F$ of $H^1(B,\iota_*\check\Lambda)$, thus a
finite unramified ring extension $ \tilde A=\CC[H^*\oplus F^*]$ of $A$, or
geometrically an isogeny $\Spec \tilde A\to\Spec A$ of tori. By Proposition~A.14
in \cite{theta}, the algebraic torus $\Spec\big(\CC[L^*]\big)$ acts on the
pull-back of the universal family after the corresponding base change by
\begin{equation}
\label{Eqn: Finite base change map}
\tilde\foS=\Spf\big(\tilde A\lfor Q\rfor\big)\lra \foS=\Spf\big(A\lfor Q\rfor\big).
\end{equation}
The action is given by the following $L^*$-grading on monomials. For exponents
in $Q\subset\breve{\MPA^*}$, the grading is the dual of the inclusion $L\to
\breve\MPA$, while for monomials in $\tilde A=\CC[H^*\oplus F^*]$, the grading
is the dual of the composition
\[
L\stackrel{c_1}{\lra} H\lra H\oplus F.
\]
Now since $\varphi\in L$ we can compose these gradings with the dual
$L^*\to\ZZ$ of multiplication with $\varphi$ to obtain an induced $\ZZ$-grading
on $\tilde A\lfor Q\rfor$. The composition
\[
Q\lra L^*\lra \ZZ
\]
defining the $\ZZ$-grading on $Q$ is given by evaluating at
$\varphi\in\breve\MPA$, so agrees with the homomorphism of monoids $h$ inducing
\eqref{Eqn: A[Q]->A[t]}. Combining with \eqref{Eqn: Finite base change map}, we
see that the change of base morphism
\[
\Spf\big(\tilde A\lfor t\rfor\big)\to \Spf\big(\tilde A\lfor Q\rfor\big)
\]
is equivariant with respect to the inclusion of tori $\GG_m\to
\Spec\big(\CC[L^*]\big)$. The induced $\GG_m$-action then also lifts to the
pull-back $\tilde\foX\to \Spf\big(\tilde A\lfor t\rfor\big)$ of our family as
claimed. The statements on the weights of the $\GG_m$-action are immediate from
our construction.
\end{proof}

\begin{remark}
\label{Rem: GG_m-action}
Since $\varphi\in L$ is primitive, so is $c_1(\varphi)$ in the isomorphic image
$H=c_1(L)\subseteq H^1(B,\iota_*\check\Lambda)$ of $L$. Thus there exists a
splitting $H^*=\ZZ\oplus \bar H^*$, with $\bar H^*$ the image of
$c_1(\varphi)^\perp$ under the map $H^1(B,\iota_*\check\Lambda)^*\to H^*$
dual to the inclusion of $H$. Then the $\ZZ$-grading on $H^*\oplus
F^*= \ZZ\oplus \bar H^*\oplus F^*$ is given by projection to the first factor.
This implies that we have a $\GG_m$-equivariant product decomposition
\begin{equation}
\label{Eqn: pull-back family tilde foX}
\tilde\foX= \GG_m\times\bar\foX\lra \tilde\foS=\GG_m\times\bar\foS
\end{equation}
of the family, with $\bar\foS= \Spf\big(\bar A\lfor t\rfor\big)$, $\bar
A=\CC[\bar H^*\oplus F^*]$, and $\GG_m$ acting by multiplication on the first
factor and trivially on $\bar\foX$ and $\bar\foS$. Note that this
product decomposition depends on the splitting $H^*=\ZZ\oplus \bar H^*$, which
is only unique up to changing the embedding of $\ZZ$ by an element of $\bar
H^*$. We fix one such choice from now on and denote the ring epimorphism
induced by the projection $H^*=\ZZ\oplus \bar H^*\to \bar H^*$ to the second
factor by \begin{equation}
\label{Eqn: chi}
\chi: \tilde A=\CC[H^*\oplus F^*]\lra \bar A=\CC[\bar H^*\oplus F^*].
\end{equation}
The corresponding morphism $\Spf\big(\bar A\lfor t\rfor\big)\to \Spf\big(\tilde
A\lfor t\rfor\big)$ identifies $\bar\foS$ with the slice $\{e\}\times
\bar\foS\subset \tilde \foS$ for $e\in\GG_m$ the unit point. By equivariance we
can assume that a lift $\bar a$ of $a$ to the finite cover $\tilde\foS\to \foS$
lies in the slice $\bar\foS$.

To prove Theorem~\ref{Thm: Analyticity of GS} it is therefore enough to prove
the existence of an analytic family $\bar\shY\to \bar U\times \DD$ with
completion at $(\bar a,0)$ isomorphic to the completion of $\bar\foX\to \bar\foS$
at $\bar a$. Indeed, \eqref{Eqn: pull-back family tilde foX} is the base-change of $\bar\foX\to\bar\foS$ by the completion at $t=0$ of the $\CC^*$-invariant map
\begin{equation}
\label{Eqn: base change for bar shY}
\Spec(\bar A)_\an\times \CC^*\times\CC \lra \Spec(\bar A)_\an\times\CC,\quad
(s,\lambda,t)\longmapsto (s,\lambda^{-1} t).
\end{equation}
The preimage of $\bar U\times\DD$ under this map is a neighborhood of the zero
set of $\lambda^{-1}\cdot t$. Thus we can construct $\shY\to U\times \DD$ by
base change of $\bar\shY\to\bar U\times\DD$ with the restriction of \eqref{Eqn:
base change for bar shY} to an appropriate neighborhood of
$(\bar a,1,0)$.
\end{remark}


\subsection{Analytic approximation with monomial period functions}
According to Remark~\ref{Rem: GG_m-action}, it suffices to prove local
analyticity of the slice $\bar\foX\to\bar\foS$ of the discussed $\GG_m$-action
on a finite unramified cover of our family $\foX\to\foS$. Denote by
$\bar\foD\subset\bar\foX$ the restricted divisor defined by $\partial B$.
For $\bar S\subset \Spec\big(\bar A[t]\big)_\an$ an
open neighborhood of $\bar x=(\bar a,0)$, for any $k\ge0$ write $\bar S_{k}$ for the
closed analytic subspace of $\bar S$ given by $(t^{k+1})$. Note that $\bar S_k$
agrees with an open subset of the analytification of the closed subscheme of
$\bar\foS$ given by $(t^{k+1})$. Let $\bar\foX_k$, $\bar\foD_k$ denote the
subschemes of $\bar\foX$ and $\bar\foD$ given by $(t^{k+1})$.

For the following statement recall the notion of \emph{divisorial log
deformations} from \cite{logmirror2}, Definition~2.7, a version of log smooth
deformations appropriate for our particular relatively coherent log structures.

\begin{proposition}
\label{approx-is-toric-degen}
Assume $(B,\P)$ is simple, $B$ compact and either $\partial B=\emptyset$ or
$\partial B$ is again an affine manifold. Then there is an integer $k_0>0$ with
the following property.

Let $\bar\pi: \bar\shY\to \bar S$ be a flat analytic family together with a
Cartier divisor $\bar\shD\subset \bar\shY$ that is also flat over $\bar S$.
Assume that for some $k\ge k_0$ there is an isomorphism $f_k$ over $\bar S_k$ of
the base change of the pair $(\bar\shY,\bar\shD)$ to $\bar S_k$ with
$\big((\bar\foX_k)_\an,(\bar\foD_k)_\an\big)$. Then
$(\bar\shY,\bar\shD)\ra \bar\shS$ with the divisorial log structure defined by
$t=0$ is a \emph{divisorial log deformation}\footnote{The case $\partial
B\neq\emptyset$ ($\Leftrightarrow\bar\foD\neq\emptyset$) wasn't actually covered
in \cite{logmirror2}, Definition 2.7, but its inclusion is straightforward.} in
the sense of \cite{logmirror2}, Definition 2.7.

Furthermore, $f_k$ induces an isomorphism of the fibers over $t=0$ as
log spaces when equipped with the restriction of the divisorial log structures
obtained from the divisors $\{t=0\}\cup \bar\shD\subset \bar\shY$ and
$\{t=0\}\cup \bar\foD \subset \bar\foX$ respectively, compatible with the log
morphism to $\bar S_0$ also given the restriction of the divisorial log
structure via $t$ on $\bar S$.
\end{proposition}

\begin{proof} 
Since $\partial B$ is again an affine manifold, it is also simple and $\bar\shD$
is the corresponding canonical deformation. By simplicity,
\cite{logmirror2},\,Proposition 2.2 and \cite{Al95},\,(6.5)\,Corollary, the
fibers in both $\bar\foX$ and $\bar\foD$ away from $t=0$ are locally rigid.
Hence, by \cite{logness},\,Lemma 2.5 and properness, there is $N>0$ such that
$t^N\shT^1_{\bar\foX/\bar\foS}=0$ and $t^N\shT^1_{\bar\foD/\bar\foS}=0$
where we refer to \cite{logness}\,\S3 for the definition of the sheaves
$\shT^i$. Let $I$ be the ideal sheaf of $\bar\foD$ in $\bar\foX$. By increasing
$N$ if needed and using properness again, we may assume that the kernel of
multiplication by $t^N$ in $\shT^2({\bar\foX/\bar\foS},I)$ and in
$\shT^2_{\bar\foD/\bar\foS}$ are stationary, that is, do not change with larger
$N$ (see \cite{logness}\,\S3.8). Choose $k_0>4N$ and assume
$(\bar\shY,\bar\shD)\ra \bar\shS$ satisfies the assumptions for this $k_0$. If
$\bar\foX_\an\leftarrow V\ra U$ is a local model at a point $y\in \bar\foX_\an$
then by \cite{logness}, Theorem 2.4, possibly after shrinking $V,U$, we find that
also $y\in \bar\shY$ has this local model. The case $y\not\in\bar\foD$ follows
directly $(Z=\emptyset)$. For $y\in\bar\foD$, let $D\subset U$ denote the divisor in
the local model. We first apply \cite{logness}, Theorem 2.4 to $D$ and $\bar\shD$ to
obtain an isomorphism $\varphi:D\ra\bar\shD$ locally at $y$. Then use this
isomorphism $\varphi$ as input in a second application of \cite{logness},
Theorem 2.4, now with $Z=D,Z'=\bar\shD$, to find the pair $(\bar\shY,\bar\shD)$
isomorphic to $(U,D)$ locally at $y$. This implies that $\bar\shY$ has the same
local models as $\bar\foX$ and since the latter is a divisorial deformation, so
is the former. That the log structures on the central fibers agree follows from
\cite{logness},\,Theorem 5.5.
\end{proof}

By Theorem~\ref{Thm: Analytic approximation}, a flat analytic family
$\bar\pi:\bar\shY\to \bar S$ and $\bar\shD\subset \bar\shY$ satisfying the
assumptions in Proposition~\ref{approx-is-toric-degen} exists and we take one.
Without loss of generality we may assume $k_0>\delta$ with $\delta\in\NN$ the
positive generator of the image of the map in \eqref{Eqn: Picard-Lefschetz
monodromy}. Thus $\delta$ is the minimal strictly positive value of $\big\langle
c_1(\varphi),\beta_\trop\big\rangle$ for $\beta_\trop\in H_1(B,\iota_*\Lambda)$.

For both families, $\bar\foX\to\bar\foS$ and $\bar\shY\to \bar S$ we have our
exponentiated period integrals constructed from certain $n$-cycles on
$X_{\bar x}$. In the first case these are formal rational
functions on $\bar\foS$, in the second case germs of meromorphic functions on
$\bar S$ at $\bar x$. To obtain regular functions, we restrict to those
$n$-cycles constructed from tropical cycles $\beta_\trop\in
H_1(B,\iota_*\Lambda)_+$ from \eqref{Eqn: (H_1)_+}, that is, with $\langle
c_1(\varphi),\beta_\trop\rangle\ge 0$. By Theorem~\ref{Thm: period thm}, on
$\bar\foX\to\bar\foS$, the exponentiated period integral for such cycles equals
the monomial
\begin{equation}\label{Eqn: regular period integrals for canonical family}
s^{-\bar\beta_\trop^*}\cdot t^{\langle c_1(\varphi),\beta_\trop\rangle}\in
\bar A\lfor t\rfor.
\end{equation}
Here we write $s^{-\bar\beta_\trop^*}$ for the monomial in $\bar A$
defined by the image of $s^{-\beta_\trop^*}\in \tilde A=\CC[H^*\oplus F^*]$
under the projection $H^*\oplus F^*\to \bar H^*\oplus F^*$ defining $\chi$
\eqref{Eqn: chi}. We now want to apply a holomorphic coordinate change to $\bar
S$ to achieve the same formula for the period integrals on $\bar\shY\to \bar S$.

For the following statement, we assume without loss of generality that the
neighborhood $\bar S$ of $\bar x$ is of the form $\bar U\times \DD$ with
$\bar U\subset \Spec(\bar A)_\an$ an analytic open set and $\DD\subset\CC$ a
small open disc with coordinate~$t$.

\begin{proposition} 
\label{prop-prepare-inv-periods}
After a holomorphic change of coordinates of $\bar S=\bar U\times \DD$ at $\bar
x=(\bar a,0)$ restricting to the identity on $\bar U\times\{0\}$ and leaving $t$
unchanged modulo $t^2$, for every tropical cycle $\beta_\trop\in
H_1(B,\iota_*\Lambda)_+$, the exponentiated period integral on $\bar\shY\to \bar
S$ is the monomial function
\begin{equation} 
\label{period-formula-on-barS}
h_{\beta_\trop}=s^{-\bar\beta^*_\trop}\cdot t^{\langle c_1(\varphi),\beta_\trop\rangle}.
\end{equation}
\end{proposition}

\begin{proof}
The choice of slice $\bar\foS\subset\tilde \foS$ was made in order for the periods of
$\bar\foX\to\bar\foS$ to yield a system of coordinates at $\bar{x}\in\bar\foS$. The
statement will follow by showing that the same is also true for $\bar\shY\to\bar
S$.

By the theory of period integrals developed in Appendix~\ref{Sect: Period
integrals}, the exponentiated period integral $h_{\beta_\trop}$ for a tropical
cycle $\beta_\trop\in H_1(B,\iota_*\Lambda)_+$ on $\bar U\times \DD^*$ extends
holomorphically to the fiber over $t=0$; moreover, the restriction to $t=0$
depends only on the restriction of $\bar\shY\to \bar S$ to $t=0$. Since the
analytic family $\bar\shY\to \bar S$ agrees with the canonical family
$\bar\foX\to \bar\foS$ to order $k_0\ge 1$, the exponentiated period integrals
for $\beta_\trop\in c_1(\varphi)^\perp$ agree with the monomial exponentiated
periods $s^{-\bar\beta_\trop^*}$ for the canonical family modulo $t^{k_0+1}$.
The exponents $-\bar\beta_\trop^*$ thus obtained cover the image under the
projection $H^*\oplus F^*\to \bar H^*\oplus F^*$ of the sublattice $K^*\subset
H^1(B,\iota_*\check\Lambda)^*\subseteq H^*\oplus F^*$ from \eqref{Eqn: K^*
Sublattice}. Now $K^*$ agrees with $c_1(\varphi)^\perp$ up to finite index and
$c_1(\varphi)^\perp$ maps onto a finite index sublattice of $\bar H^*\oplus
F^*$. Thus since $\bar A=\CC[\bar H^*\oplus F^*]$ and $\bar U$ is an open subset of
$\Spec(\bar A)_\an$, period functions $h_{\beta_\trop}$ for $\beta_\trop$ with
$\langle c_1(\varphi), \beta_\trop\rangle=0$ span the relative cotangent space
$T^*_{\bar S/\DD,\bar x}= T^*_{\bar U,\bar a}$.

Let $\beta_\trop^1,\ldots,\beta_\trop^{r}\in H_1(B,\iota_*\Lambda)$ map to a
basis of $K^*$. Then $h_{\beta_\trop^1},\ldots,h_{\beta_\trop^r}$ are a basis of
the relative cotangent space $T^*_{\bar S/\DD,\bar a}$, hence define local coordinates
on $\bar U\times\{0\}$. Additionally pick some $\beta^0_\trop\in
H_1(B,\iota_*\Lambda)_+$ with $\langle c_1(\varphi),\beta_\trop^0
\rangle=\delta$. By \eqref{Eqn: regular period integrals for canonical family}
and since $\bar\shY\to\bar S$ agrees with $\bar\foX\to\bar\foS$ to order
$k_0>\delta$, the exponentiated period integral for $\beta_\trop^0$ on $\bar S$
has the form
\[
h_{\beta_\trop^0}=s\cdot t^\delta,
\]
with $s$ an invertible function on $\bar S$ restricting to
$s^{-(\bar\beta_\trop^0)^*}$ on the fiber over $t=0$. We claim that there exists a local
biholomorphism $\Phi$ of $\bar S$ with
\[
\Phi^*\big(h_{\beta^0_\trop}\big)=s^{-(\bar\beta_\trop^0)^*}\cdot t^\delta,\quad
\Phi^*\big(h_{\beta_\trop^i}\big)=s^{-(\bar\beta_\trop^i)^*},\ i=1,\ldots,r,
\]
restricting to the identity on $t=0$ and leaving $t$ unchanged modulo $t^2$. In
fact, since $h_{\beta^1_\trop},\ldots, h_{\beta^r_\trop}$ restrict to local
coordinates $s_1=s^{-(\beta^1_\trop)^*},\ldots,s_r=s^{-(\beta^r_\trop)^*}$ on
$\bar U\times\{0\}$, the implicit function theorem applied with $t$ as a
parameter produces a local biholomorphism $\Phi_1$ with
\[
\Phi_1^*(t)=t,\quad \Phi_1^*\big( h_{\beta^i_\trop}\big)= s_i=s^{-(\beta^i_\trop)^*},\ i=1,\ldots,r.
\]
Another application of the implicit function theorem with parameters
$s=(s_1,\ldots,s_r)$ finds a local holomorphic function $a(t,s)$ such that the
local biholomorphism $\Phi_2$ defined by
\[
\Phi_2^*(t)=\big(1+a(t,s)t\big)\cdot t,\quad \Phi_2^*(s_i)=s_i,\ i=1,\ldots, r
\]
fulfills
\[
\Phi_2^*\big( h_{\beta^0_\trop}\big) = s^{-(\beta^0_\trop)^*}\cdot t^\delta.
\]
Note that the two sides of the last equation already agree modulo $t^{\delta+1}$; the
equation to solve to find $a$ is the difference of the two sides divided by
$t^\delta$. Then $\Phi=\Phi_2\circ\Phi_1$ defines the sought-after local
biholomorphism.

Finally, the claimed identity \eqref{period-formula-on-barS} holds for all
$\beta_\trop$ since $\beta^0_\trop,\ldots, \beta^r_\trop$ span the image of the
map $\beta_\trop\mapsto\beta_\trop^*$ from \eqref{Eqn: H_1->(H^1)^*}.
\end{proof}

In the proof, we have also shown the following statement.

\begin{corollary} \label{inv-per-span-rel-cotan}
Restricted to $t=0$, the functions $s^{-\bar\beta_\trop^*}$ for $\beta_\trop\in
H_1(B,\iota_*\Lambda)$ with $\langle c_1(\varphi), \beta_\trop\rangle=0$ span
the relative cotangent space of $\bar S$ over $\DD$.
\end{corollary}

We endow all our spaces with the divisorial (analytic or formal) log structures
defined by the divisors given by $t=0$ and $\foD\subset\foX$ or
$\shD\subset\shY$ and write $\shM_\foX$, $\shM_{\shY}$ etc.\ for the respective
monoid sheaves. By Proposition~\ref{approx-is-toric-degen} and Remark~\ref{Rem:
GG_m-action}, the restriction of $(\shY,\shM_\shY)\to (S,\shM_S)$ to $t=0$ is
isomorphic to the restriction of $(\foX,\shM_\foX)\to (\foS,\shM_\foS)$ to $t=0$
as a morphism of log spaces over the standard log point. Note that for
this last statement to be true it is important that in
Proposition~\ref{prop-prepare-inv-periods} we left $t$ unchanged modulo $t^2$.
\medskip

For the final step of the proof we need to restore the $\GG_m$-factor from
Remark~\ref{Rem: GG_m-action} from the logarithmic perspective as follows.

\begin{proposition}
\label{prop-iso-k0-including-C*}
With the log structures defined as in Proposition~\ref{approx-is-toric-degen},
there is an isomorphism of analytic log spaces over the standard log point
between the restrictions to the closed subspace $U\times
\Spec\big(\CC[t]/(t^{k_0+1})\big)_\an$ of the analytic family $\shY\ra U\times
\DD\subseteq \Spec \big(A[t]\big)_\an$ constructed in Remark~\ref{Rem:
GG_m-action} and of $\foX_\an\ra\foS_\an$, respectively.
\end{proposition}

\begin{proof}
In Remark~\ref{Rem: GG_m-action} the family $\shY\to U\times\DD$ was constructed
from $\bar\shY\to \bar U\times\DD$ by a base-change with completion at $t=0$ the
base change producing $\tilde\foX\to \tilde\foS$ out of $\bar\foX\to\bar\foS$.
The statement therefore follows from the corresponding statement for
$\bar\shY\to \bar U\times\DD$ and $\bar\foX\to\bar\foS$.
\end{proof}

\begin{remark}
\label{Rem: different slices and spaces of C[[t]]}
Recall that the base change $\CC[t]\to \CC[\lambda^{\pm1},t]$ from \eqref{Eqn:
base change for bar shY} maps $t$ to $\lambda^{-1} t$. Thus as a log space over the
standard log point, the fiber of $\shY\to U\times\DD$ over $t=0$ is not the
product of the fiber of $\bar \shY\to \bar U\times\DD$ with the standard log
point, the product structure is modified by rescaling the pull-back of $t$ as a
generator of the log structure of the standard log point by the coordinate
$\lambda^{-1}$ of the $\CC^*$-orbit.
\end{remark}


\subsection{Log-versality of the canonical family}
\label{Par: Log versality}

To finish the proof of Theorem~\ref{Thm: Analyticity of GS} and
Corollary~\ref{Cor: monomial analytic periods}, we need to compare two formal
schemes over $\hat \foS=\Spf\big(\hat A\lfor t\rfor\big)$, where $\hat A$ is the
completion of $A=\CC[H^1(B,\iota_*\check\Lambda)^*]$ at the maximal ideal
$\fom_{a}$ defining the closed point $a$. The first is the completion
$\hat\foX\to\hat\foS$ of the formal family $\foX\to \foS=\Spf\big(A\lfor
t\rfor\big)$ at $x\in \Spec A$; the second is the completion $\hat\shY\to
\hat\foS$ of the analytic family $\shY\to U\times \DD\subseteq\Spec\big(A[t]\big)_\an$ at $x$.
To make sense of the comparison in the category of formal schemes, note also that the fiber of
$\shY\to U\times \DD$ over $x$ is the analytification of a scheme, the fiber $X_x$ of
$\foX\to \foS$ over $x$. By GAGA for proper schemes (\cite{SGA1},
Th\'eor\`eme~4.4), the restriction of $\shY\to U\times \DD$ to the $k$-th order thickening
$\Spec\big(\O_{U\times \DD}/\fom_x^{k+1} \big)$ of $x$ is then also the analytification
of a scheme. We can thus also view $\hat\shY\to \Spf\big(\hat A\lfor
t\rfor\big)$ as a morphism of formal schemes, with the same base $\hat\foS$ and
same closed fiber $X_x$ as $\hat\foX\to\hat\foS$.

We do the comparison of the two families by showing that both are a
hull\footnote{The notion of hull arises in deformation situations where the
functor may not be representable, but one has a versal object that produces any
family by pull-back. The hull is a minimal such family. A hull is unique up to
an isomorphism over the base ring $\CC\lfor t\rfor$. The isomorphism may not be
unique, but its differential at the closed point is (\cite{Schlessinger},
Proposition~2.9).} for a certain functor of log deformations of
$(X_x,\shM_{X_x})$ as log spaces over $\CC\lfor t\rfor$, with the log structure
on $\CC\lfor t\rfor$ defined by the chart $\NN\to \CC\lfor t\rfor$ mapping
$1\in\NN$ to $t$. This situation fits into the traditional framework of functors
of ordinary Artinian $\CC\lfor t\rfor$-algebras as treated by Schlessinger
\cite{Schlessinger}, by defining the log structure on an Artinian $\CC\lfor
t\rfor$-algebra by pull-back from $\CC\lfor t\rfor$. Uniqueness of the hull in
\cite{Schlessinger}, Proposition~2.9, then implies that $\hat\foX\to \hat\foS$
and $\hat\shY\to\hat\foS$ are isomorphic as formal schemes over
$\CC\lfor t\rfor$. Moreover, since the period integrals only depend on the
formal family, by Proposition~\ref{prop-prepare-inv-periods} and
Corollary~\ref{inv-per-span-rel-cotan}, this isomorphism turns out to be a
morphism even over $\hat\foS=\Spf\big( \hat A\lfor t\rfor\big)$. We now carry
out the details of this idea of proof.

Recall that $\hat\foX$ and $\hat\foS$ come with log structures and the morphism
$\hat\foX\to\hat\foS$ indeed lifts to a morphism of formal log schemes. For the
closed point $ x\in \Spf\big( \hat A\lfor t\rfor\big)$ and $X_x=\pi^{-1}(x)$,
consider the deformation functor $\shD$ that sends a local Artinian $\CC\lfor
t\rfor$-algebra $R$, viewed as a log ring by the structure homomorphism
$\CC\lfor t\rfor \to R$, to the set $\shD(R)$ of isomorphism classes of flat
divisorial log deformations of $X_x^\dagger=(X_x,\M_{X_x})$, defined in
\cite[Definition 2.7]{logmirror2}. As in \cite{logmirror2} we now use a dagger
superscript to indicate log spaces. We check in Theorem~\ref{Thm:
divisorial hull} in Appendix~\ref{App: log hull} that the deformation functor
$\shD$ has a pro-representable hull. Thus there exists a complete local
$\CC\lfor t\rfor$-algebra $R$ and a divisorial log deformation $\xi\in \shD(R)$
that is a hull for $\shD$.

By the defining property of pro-representable hulls, our two formal divisorial
log deformations $\hat\foX\to\hat\foS$ and $\hat\shY\to\hat\foS$ now arise as
respective pull-backs of $\xi$ by two classifying morphisms
\begin{equation}
\label{Eqn: base change from hull}
h_{\hat\foX}:\hat\foS\lra \Spf R,\quad
h_{\hat\shY}:\hat\foS\lra \Spf R
\end{equation}
of formal schemes over $\CC\lfor t\rfor$. We claim that both $h_{\hat\foX}$ and
$h_{\hat\shY}$ are isomorphisms. Since $\hat\foS$ is smooth, it suffices to
check this statement at the level of cotangent spaces. We provide a proof in the
needed setup for lack of a reference.

\begin{lemma}
\label{Lem: impl fct thm over CC[[t]]}
Let $\varphi:(A,\mathfrak{m})\ra (B,\mathfrak{n})$ be a local map of complete local
Noetherian $\CC\lfor t\rfor$-algebras with residue field $\CC$ and
$B$ flat over $\CC\lfor t\rfor$. Assume that $\varphi$ induces an isomorphism
$\mathfrak{m}/( tA+\mathfrak{m}^2)\ra \mathfrak{n}/(t B+\mathfrak{n}^2)$ of relative Zariski
cotangent spaces. Then $\varphi$ is an isomorphism.
\end{lemma}

\begin{proof}
Surjectivity of $\varphi$ is \cite{Schlessinger},\,Lemma~1.1. By flatness, $B$
is a free $\CC\lfor t\rfor$-module. Let $b_i\in B$, $i\in I$, be a freely
generating set. According to surjectivity, for any $i\in I$ we can choose
$a_i\in A$ with $\varphi(a_i)=b_i$. Then by the assumption on the map of
cotangent spaces and Nakayama's lemma in $A$, the $a_i$ generate $A$ as a
$\CC\lfor t\rfor$-module. But the $a_i$ also cannot fulfill any non-trivial
relation for the $b_i$ do not. Thus $A$ is freely generated by $a_i$, $i\in I$,
and $\varphi$ is an isomorphism.
\end{proof}

To finish the proof of Theorem~\ref{Thm: Analyticity of GS} it essentially remains to
show that the differentials of the maps in \eqref{Eqn: base change from hull}
relative $\CC\lfor t\rfor$ are isomorphisms. Note that by the definition of the
hull, the Zariski-tangent space of $R$ is the tangent space to our functor
\[
t_{\shD}=\shD\big(\CC[\varepsilon]\big)=\Hom_{\CC\lfor t\rfor}\big(R,\CC[\varepsilon])
=\Hom_\CC(\fom_R/\fom_R^2,\CC).
\]
By \cite{logmirror2}, Theorem~2.11,2, we
furthermore have a canonical isomorphism
\[
t_\shD= H^1(X_x,\Theta_{X_x^\ls/\CC^\ls}),
\]
where $\CC^\ls$ denotes the standard log point $(\Spec\CC,\NN\oplus\CC^*)$. With
this identification, the differentials $Dh_{\hat\foX}$, $Dh_{\hat\shY}$ relative
$\CC\lfor t\rfor$ of \ref{Eqn: base change from hull} are the Kodaira-Spencer
maps of our two families:
\begin{equation}
\label{Eqn: Kodaira-Spencer map}
Dh_{\hat\foX}, Dh_{\hat\shY}: T_{\hat\foS/\CC\lfor t\rfor,x}\lra H^1(X_x,\Theta_{X_x^\ls/\CC^\ls}).
\end{equation}

\begin{proposition}
\label{KS-iso}
The relative differential
\[
Dh_{\hat\foX}: T_{\hat\foS/\CC\lfor t\rfor,x}=H^1(B,\iota_*\check\Lambda)\otimes\CC\ra
H^1(X_x,\Theta_{X_x^\dagger/\CC^\dagger})
\]
of the classifying morphism $h_{\hat\foX}: \hat\foS\to \Spf R$ for
$\hat\foX\to\hat\foS$ from \eqref{Eqn: base change from hull} coincides with the
natural map given in Proposition~\ref{open-cover-is-iso}. In particular, this
map is an isomorphism. 
\end{proposition}

\begin{proof}
Since we work with relative tangent spaces only the restriction to $t=0$ is
relevant. For $\bar n\in H^1(B,\iota_*\check\Lambda)\otimes\CC$ denote by
\[
\partial_{\bar n}: A=\CC[H^1(B,\iota_*\check\Lambda)^*]\lra \CC
\]
the associated $\CC$-linear derivation defined by $\partial_{\bar
n}(s^m)=\langle m,\bar n\rangle a(m)$ for $m\in H^1(B,\iota_*\check\Lambda)^*$.
Here $a(m)\in\CC$ is the reduction of $s^m$ modulo $\maxid_a$ defining the given
closed point $a\in \Spec A$. The pair $(a,\bar n)$ is equivalent to the
associated $\CC$-algebra map 
\[
\psi:A \lra\CC[\eps]/(\eps^2),\quad
s^m\longmapsto a(m) + \partial_{\bar n}(s^m)\eps
= a(m)\big(1+\langle m,\bar n\rangle\eps\big).
\]
We are going to describe the pullback $X_\eps\to
\Spec\big(\CC[\eps]/(\eps^2)\big)$ of $\foX\to \foS$ under $\psi$. By
functoriality in the base $S$ of the construction \cite{logmirror1}, Definition
2.28, $X_\eps$ is the toric log CY space constructed from $(B,\P,\varphi)$ for
the image under $\psi$ of the gluing data $(s_0,\sigma_0)$ for $\foX$. In
writing these gluing data as a pair, we used the identification
\[
H^1(B,\iota_*\check\Lambda\otimes A^\times)= H^1(B,\iota_*\check\Lambda\otimes\CC^*)
\oplus \big(H^1(B,\iota_*\check\Lambda)\otimes H^1(B,\iota_*\check\Lambda)^*\big),
\]
observing that $A^\times=\CC^*\oplus H^1(B,\iota_*\check\Lambda)^*$, the set of
monomials with coefficients in $\CC^*$.

For the gluing data describing $X_\eps$, we have
\begin{equation}
\label{Eqn: CC[eps]^x}
\big(\CC[\eps]/(\eps^2)\big)^\times= \CC^*\oplus\CC,
\end{equation}
as an abelian group, mapping the pair $(\lambda,c)\in \CC^*\oplus\CC$
to $\lambda(1+c\eps)$. Thus we have the decomposition
\[
H^1\big(B,\iota_*\check\Lambda\otimes(\CC[\eps]/(\eps^2))^\times\big)
=H^1(B,\iota_*\check\Lambda \otimes \CC^*)\oplus
\big(H^1(B,\iota_*\check\Lambda)\otimes\CC\big),
\]
to describe the gluing data of $X_\eps$ as a pair as well. Since the map on invertibles induced by $\psi$
\[
\CC^*\oplus H^1(B,\iota_*\check\Lambda)^*  \lra \CC^*\oplus\CC,\quad
(\lambda,m)\longmapsto \big( \lambda a(m),\langle m,\bar n\rangle\big),
\]
respects the decompositions as pairs, so does the map on cohomology induced by
$\psi$. The first summand maps $s_0$ to $a s_0$, the translation of $s_0\in
H^1(B,\iota_*\check\Lambda\otimes\CC^*)$ by $a$ as an element of the algebraic
torus $H^1(B,\iota_*\check\Lambda)\otimes\CC^*$ acting on gluing data. This is
expected since $a s_0$ is the gluing data giving rise to the central fiber $X_x$
of $\foX$.

To describe the image of $X_\eps$ in $H^1(X_x,\Theta_{X_x^\ls/\CC^\ls})$ under
the Kodaira-Spencer map, we need to work on the level of cocycles. We use the
coverings by the open sets $W_\tau\subset B$ and $V_\tau\subset X_x$ from
Appendix~\ref{appendix-C}. Let $\bar n$ and $a s_0$ be represented by the
cocycles $n=(n_{\omega\tau})\in \check C^1(\{ W_\tau\}_\tau,\iota_*\check\Lambda
\otimes\CC)$ and $s=(s_{\omega\tau})\in \check C^1(\{
W_\tau\}_\tau,\iota_*\check\Lambda \otimes\CC^*)$, respectively. Write
$s(1+n\eps)$ for the image of $(s,n)$ under the identification
\[
\check C^1(\{ W_\tau\}_\tau,\iota_*\check\Lambda \otimes \CC^*)\oplus
\check C^1(\{ W_\tau\}_\tau,\iota_*\check\Lambda \otimes \CC)=
\check C^1\big(\{ W_\tau\}_\tau,\iota_*\check\Lambda \otimes (\CC[\eps]/(\eps^2))^\times\big)
\]
induced by \eqref{Eqn: CC[eps]^x}. Then $X_\eps$ is canonically isomorphic to
the toric log CY space for $(B,\P,\varphi)$ defined by gluing data $s(1+n\eps)$.
Now for $\omega\subset\tau$ the section $n_{\omega\tau}$ of
$\check\Lambda\otimes \CC$ over $W_{\omega\tau}= W_\omega\cap W_\tau$ defines a
logarithmic vector field $\partial_{n_{\omega\tau}}$ on $V_{\omega\tau}= V_\omega\cap
V_\tau$, and this vector field describes the infinitesimal deformation $X_\eps$ of
$X_x$ on $V_\omega\cup V_\tau$. It thus follows from the \v Cech description of
the Kodaira-Spencer map that the image of $X_\eps$ under the Kodaira-Spencer map
is the cohomology class of the \v Cech 1-cocycle
$\big(\partial_{n_{\omega\tau}}\big)_{\omega,\tau}$. Now the map
\[
\check C^1(\{ W_\tau\}_\tau,\iota_*\check\Lambda \otimes \CC)\lra
\check C^1(\{ V_\tau\}_\tau,\Theta_{X_x^\ls/\CC^\ls}),\quad
(n_{\omega\tau})\longmapsto (\partial_{n_{\omega\tau}})
\]
indeed agrees with the natural isomorphism in
Proposition~\ref{open-cover-is-iso} at the level of cochains, as claimed.
\end{proof}

\begin{proof}[Proof of Theorem~\ref{Thm: Analyticity of GS} and of Corollary~\ref{Cor: monomial analytic periods}.]
Proposition~\ref{KS-iso} and Lemma~\ref{Lem: impl fct thm over CC[[t]]} show
that the classifying map $h_{\hat\foX}: \hat\foS\to \Spf R$ for
$\hat\foX\to\hat\foS$ from \eqref{Eqn: base change from hull} is an isomorphism.
The argument in Proposition~\ref{KS-iso} only required knowing
$\hat\foX\to\hat\foS$ as a divisorial deformation to first order on the fiber
over $t=0$ and, by Proposition~\ref{prop-iso-k0-including-C*}, hence also
applies to $\hat\shY\to\hat\foS$. Thus also the classifying map $h_{\hat\shY}:
\hat\foS\to \Spf R$ for $\hat\shY\to\hat\foS$ is an isomorphism. Taking the
composition $h_{\hat\shY}\circ h_{\hat\foX}^{-1}$, we now obtain an isomorphism
of formal divisorial deformations of $(X_x,\M_{X_x})$, that is, a cartesian
diagram
\begin{equation}
\label{Eqn: iso of formal families}
\begin{CD}
\hat\foX@>>> \hat\shY\\
@VVV@VVV\\
\hat\foS@>h_{\hat\shY}\circ h_{\hat\foX}^{-1}>>\hat\foS
\end{CD}
\end{equation}
over $\CC\lfor t\rfor$ with horizontal maps isomorphisms. But by
Proposition~\ref{prop-prepare-inv-periods} and \eqref{Eqn: regular period
integrals for canonical family}, the exponentiated period functions for
$\hat\foX\to\hat\foS$ and $\hat\shY\to\hat\foS$ agree and contain a system of
coordinate functions on the fiber over $t=0$ of $\bar\foS$. Since furthermore $t$ maps to $t$, in view of Remark~\ref{Rem: different slices and spaces of C[[t]]}, the lower
horizontal arrow $h_{\hat\shY}\circ h_{\hat\foX}^{-1}$ in \eqref{Eqn: iso of
formal families} is the identity. This finishes the proof of both
Theorem~\ref{Thm: Analyticity of GS} and of Corollary~\ref{Cor: monomial
analytic periods}.
\end{proof}


\begin{appendix}

%
%

\section{Finite order period integrals}
\label{Sect: Period integrals}

The main result of this paper computes certain period integrals of a relative
logarithmic holomorphic $n$-form for a flat analytic map $X_k\to \Spec
\CC[t]/(t^{k+1})$ over a family of $n$-cycles. The result is given in the form
$g\log t+ h$ with $g,h\in \CC[t]/(t^{k+1})$. The purpose of this section is to
define such integrals unambiguously despite only working in a finite order
deformation and despite the appearance of the log-pole. It is also
straightforward to incorporate analytic parameters by replacing the ground field
$\CC$ by an analytic $\CC$-algebra $S=\CC\{s_1,\ldots,s_n\}/ (f_1,\ldots,f_k)$.
For the sake of readability all formulas are given over $\CC$.

The log-pole arises by the intersection of the cycle with the singular locus
$(X_0)_\sing\subset X_0$, where locally $X_0$ is assumed to be normal crossings
and $(X_0)_\sing$ smooth. As a preparation, we take a closer look at relative
logarithmic differential forms near a double locus. We work analytically and
denote by $\DD$ the unit disk in $\CC$ and by $\hat \DD$ a slightly larger disk. Let
$\kappa\in\NN\setminus\{0\}$ and denote
\begin{equation}
\label{Eqn: H_kappa}
\hat H_\kappa= \big\{ (z,w,t)\in \hat \DD^2\times \DD\, \big|\, zw=t^\kappa\big\},
\end{equation}
viewed as an analytic log space with log structure
induced by the divisor with normal crossings $t=0$. The function $t$ defines a
log morphism $\hat H_\kappa\to \DD$, for $\DD$ endowed with the divisorial log
structure for $\{0\}$. To not overburden the notation, the log structure is not
made explicit in the notation, but should always be clear from context. A
crucial fact for the following is that a holomorphic function $f$ on $\tilde
H_\kappa$ can be written uniquely as a sum
\begin{equation}
\label{Eqn: decompose hol functions}
f(z,w,t)=z\cdot g(z,t)+ w\cdot h(w,t)+c(t)
\end{equation}
with $g\in \CC\{z,t\}$, $h\in \CC\{w,t\}$ and $c\in\CC\{t\}$, by replacing
mixed terms $zw$ by $t^\kappa$ and then collecting the respective monomials. 

By definition, the sheaf of relative logarithmic $1$-forms $\Omega^1_{\hat
H_\kappa/\DD}$ is the invertible $\O_{\tilde H_\kappa}$-submodule of the sheaf
of relative meromorphic differential forms on $\hat H_\kappa$ generated by
\[
\frac{dz}{z}= -\frac{dw}{w}.
\]
Recall that this relation arises by applying $\dlog$ to the equation
$zw=t^\kappa$ and modding out by $\frac{dt}{t}$. Together with
\begin{align*}
w^{l+1}dz&= -w^{l+1}zw^{-1}dw= -t^\kappa w^{l-1}dw,\\
z^{l+1}dw&= -z^{l+1}wz^{-1}dz= -t^\kappa z^{l-1}dz
\end{align*}
for $l\ge 0$, we see that similarly to \eqref{Eqn: decompose hol functions}, any
$\alpha\in \Gamma(\tilde H_\kappa,\Omega^1_{\hat H_\kappa/\DD})$ can be uniquely
written in the form
\begin{equation}\label{Eqn: decompose 1-form}
\alpha=g(z,t) dz +h(w,t)dw + c(t) \frac{dz}{z}
= g(z,t)dz+h(w,t) dw -c(t)\frac{dw}{w}
\end{equation}
with $g,h$ holomorphic functions on $\hat \DD\times \DD$ and $c$ a holomorphic
function on $\DD$.

A similar statement holds after reduction modulo $t^{k+1}$ and for forms of
higher degree in higher dimensions as follows. Fix $k>0$ throughout this
appendix. Let $O_k$ be the zero-dimensional analytic log space $\Spec
\CC[t]/(t^{k+1})$ with the restriction of the log structure on $\DD$. Let
$H_\kappa$ be the base change of $\hat H_\kappa$ to $O_k$. Then the reduction
of \eqref{Eqn: decompose 1-form} modulo $t^{k+1}$ also yields a unique
decomposition, now for $\alpha\in\Gamma(H_\kappa,\Omega^1_{H_\kappa/O_k})$ and
with $g,h\in \O(\hat \DD)[t]/(t^{k+1})$, $c\in\CC[t]/(t^{k+1})$. 

For the higher dimensional case consider $\tilde U= V\times H_\kappa$ with $V$ a
complex manifold of dimension $n-1$ and let $U$ denote the reduction of $\tilde
U$ by $t$. If $U=U'\cup U''$ is the decomposition of $U$ into the two
irreducible components defined by $w=0$ and $z=0$ respectively, and $\tilde
V=V\times O_k$, various combinations of the functions $z,w,t$ and the product
structure of $\tilde U$ define projections
\begin{align*}
p_V:\tilde U&\lra V&
p_{H_\kappa}: \tilde U&\lra H_\kappa,\\
p_{\tilde V}:\tilde U&\lra \tilde V= V\times O_k,&
p_1: \tilde U&\lra U'\times O_k,&
p_2: \tilde U&\lra U''\times O_k.
\end{align*}
With this notation, the sheaf $\Omega^p_{\tilde U/O_k}$ of relative holomorphic
logarithmic $p$-forms on $\tilde U$ decomposes as a direct sum,
\[
\Omega^p_{V\times H_\kappa/O_k}=
\big(p_V^*\Omega^{p-1}_V \otimes_{\O_{\tilde U}}
p_{H_\kappa}^*\Omega^1_{H_\kappa/O_k} \big)
\oplus p_V^*\Omega^p_V.
\]
Note also that this formula can be rewritten using $p_V^*\Omega^r_V
= p_{\tilde V}^*\Omega^r_{\tilde V/O_k}$ with $r=p-1,p$.
In view of the decomposition of relative (holomorphic) logarithmic $1$-forms of
$H_\kappa/O_k$ arising from~\eqref{Eqn: decompose 1-form}, a logarithmic
$p$-form $\alpha$ on $\tilde U$ can thus be written uniquely as a sum
\begin{equation}\label{Eqn: decompose p-form}
\alpha=  \big(p_1^* \alpha'\big)\wedge dz +
\big(p_2^*\alpha''\big)\wedge dw +
\big(p_{\tilde V}^*\alpha_\restxt\big)\wedge \frac{dz}{z} +\alpha_{\tilde V},
\end{equation}
with $\alpha'\in \Gamma(U'\times O_k, \Omega^{p-1}_{U'\times O_k/O_k})$,
$\alpha''\in \Gamma(U''\times O_k,\Omega^{p-1}_{U''\times O_k/O_k})$,
$\alpha_\restxt\in \Gamma(\tilde V, \Omega^{p-1}_{\tilde V/O_k})$ and
$\alpha_{\tilde V}\in \Gamma(\tilde U, p_V^*\Omega^p_V)$. All these differential
forms can be expanded as polynomials in $t$ by means of the canonical
isomorphisms
\[
\begin{array}{rclrcl}
\Omega^{p-1}_{U'\times O_k/O_k}&=& \Omega^{p-1}_{U'}\otimes\CC[t]/(t^{k+1}),&\quad
\Omega^{p-1}_{U''\times O_k/O_k}&=& \Omega^{p-1}_{U''}\otimes\CC[t]/(t^{k+1})\\
\Omega^{p-1}_{\tilde V/O_k}&=&  \Omega^{p-1}_V\otimes\CC[t]/(t^{k+1}),&\quad
p_V^*\Omega^p_V&=& \O_{\tilde U}\otimes_{p_V^{-1}\O_V} p_V^{-1}\Omega^p_V.
\end{array}
\]
In the last instance, for $\alpha_{\tilde V}$, we use the analogue of
\eqref{Eqn: decompose hol functions} on $\tilde U= V\times H_\kappa$ to write the
coefficient functions as polynomials in $t$.

\begin{definition}\label{Def: Phi^+}
Let $\Phi: \tilde U\to X_k$ be a logarithmic morphism relative $O_k$
with $\tilde U=U\times O_k$ and $U$ non-singular, or $\tilde U=
V \times H_\kappa$ and $V$ a complex manifold of dimension $n-1$. In the
first case define
\[
\Phi^+: \Gamma(X_k,\Omega^p_{X_k/O_k})\lra \Gamma(U,\Omega^p_U)
\otimes_\CC \CC[t]/(t^{k+1} )
\]
by composing $\Phi^*$ with the canonical isomorphism
$\Omega^p_{U\times O_k/O_k}= \Omega^p_U\otimes_\CC
\CC[t]/(t^{k+1})$. In the second case define
\[
{\Phi^+: \Gamma(X_k,\Omega^p_{X_k/O_k})}\lra
\big[\big(\Gamma(U',\Omega^{p-1}_{U'})\oplus \Gamma(U'',\Omega^{p-1}_{U''})\big)
\otimes_\CC \CC[t]/(t^{k+1})\big]
\oplus \Gamma(\tilde U ,p_V^* \Omega^{p-1}_V),
\]
by decomposing $\alpha\in \Gamma(X_k,\Omega^p_{X_k/O_k})$ according
to \eqref{Eqn: decompose p-form} and omitting the term with the simple
pole:
\[
\Phi^+(\alpha):=  \big(\alpha',\alpha'',\alpha_{\tilde V}\big).
\]
We call $\Phi^+(\alpha)$ the \emph{special pull-back of $\alpha$}.

In the second case, the $\alpha_\restxt$-component of $\Phi^*\alpha$ in the
decomposition~\eqref{Eqn: decompose p-form} also provides a homomorphism
\[
\res_\Phi: \Gamma(X_k,\Omega^p_{X_k/O_k}) \lra
\Gamma(V,\Omega^{p-1}_V)\otimes_\CC \CC[t]/(t^{k+1}).
\]
\end{definition}

Note that $\res_\Phi(\alpha)=\alpha_\restxt$ in Definition~\ref{Def: Phi^+} agrees
with the residue of the restriction of $\Phi^*(\alpha)$ to the branch $w=0$.
Restricting to the other branch $z=0$ changes the sign, but up to the choice of
branch, $\res_\Phi(\alpha)$ is well-defined as a $(p-1)$-form on the thickened
double locus $(X_0)_\sing\times O_k$.

\begin{lemma}\label{Lem: Phi^+ commutes with d}
The homomorphism $\Phi^+$ commutes with the exterior differential
$d$.
\end{lemma}

\begin{proof}
This follows easily from the definition.
\end{proof}

With the notion of special pull-back at hand we are now in position
to define our finite order period integrals.

\begin{construction}
\label{Constr: finite order periods}
Let $X_k\to O_k$ be a morphism of analytic log spaces with $O_k$ the
fat log point introduced above. Denote by $X_0$ the central fiber
and let $\beta$ be a singular differentiable $p$-cycle on
$X_0$. Here differentiability is defined on each singular simplex by
locally composing with an embedding of $X_0$ into some $\CC^N$. In a
neighborhood of the image $|\beta|\subset X_0$ of $\beta$ we assume
$X_0$ to be normal crossings and $\pi$ log smooth. Since the
discussion is local around $|\beta|$ we may just as well assume
these conditions to hold everywhere. We assume $\beta=\sum_i
\beta_i$ with each $\beta_i$ a chain mapping into the image of
$\Phi_i: \tilde U_i\to X_k$, a logarithmically strict open embedding over $O_k$
with either
\begin{align} 
\tilde U_i&= U_i\times O_k\hbox{ with }U_i\subset\CC^n\hbox{ open, or} \label{ChI}\tag{Ch I}\\
\tilde U_i&=V_i\times H_{\kappa_i}\hbox{ with }V_i\subset \CC^{n-1}\hbox{ open.}\label{ChII}\tag{Ch II}
\end{align}
We identify the reduction $U_i$ of $\tilde U_i$ with its image in
$X_0$ and we assume the $U_i$ for the second type are mutually
disjoint. The index $i$ runs over a finite subset of $\NN$.

Concerning $\beta$ we assume that
\begin{itemize}[leftmargin=8ex]
\item[(Cy I)\ ]
For either type of chart, $\partial\beta_i=\sum_\mu \gamma_i^\mu$
with $(p-1)$-cycles $\gamma_i^\mu$, the number of summands depending
on $i$. Moreover, for each $(i,\mu)$ there exist exactly one $j\neq
i$ and one $\nu$ with $|\gamma_i^\mu|\cap |\gamma_j^{\nu}|\neq
\emptyset$. For such $(i,\mu),(j,\nu)$, it necessarily holds
$\gamma_i^\mu=- \gamma_{j}^{\nu}$ since $\partial\beta=0$.
\item[(Cy II)]
If $\Phi_i$ is of type~II (i.e., $U_i\cap (X_0)_\sing\neq\emptyset$) then
$\beta_i$ is homologous relative to $\partial\beta_i$ either to
(i)~$\gamma_i\times \Sigma$ with $\Sigma$ the two-chain
$\big[\ol \DD\times\{0\}\big] + \big[\{0\} \times \ol \DD\big]$ in
$H_{\kappa_i},$ or to (ii)~$\gamma_i\times\iota$, with
\[
\iota:[-1,1]\lra H_{\kappa_i},\quad
\lambda\longmapsto \begin{cases}
(-\lambda,0,0),&-1\le \lambda\le0\\
(0,\lambda,0),&0\le \lambda\le 1.
\end{cases}
\]
In the two cases, $\gamma_i$ is a $(p-2)$- and $(p-1)$-cycle in
$V_i$, respectively. In particular, $\partial
\beta_i=\gamma_i^1-\gamma_i^2$ with $\gamma_i^\mu= \gamma_i\times S^1$ homologous to zero in the first case and $\gamma_i^\mu$ homologous to $\gamma_i\times \{0\}$ in the second case.
\end{itemize}
For $i\neq j$ denote $U_{ij}= U_i\cap U_j=U_{ji}$. We then have two open
embeddings $\Phi_i^{(j)},\Phi_j^{(i)}: U_{ij}\times O_k\to X_k$, defined by the
restrictions of $\Phi_i$ and $\Phi_j$, respectively. Note that if
$U_{ij}\neq\emptyset$, at most one of the two charts can be of type~II, say
$U_j$. In this case, $U_j= U'_j\cup U''_j$ decomposes into
two irreducible components with only one of them intersecting $U_{ij}$, say
$U'_j$. In the coordinates $z,w,t$ for $H_{\kappa_j}$ assume that
$U'_j$ is defined by $w=0$. Then for $z\neq0$ we can
eliminate $w$ via $w=z^{-1}t^{\kappa_j}$ to obtain an
identification $\tilde U_j\setminus U''_j= (U'_j\setminus U''_j)\times
O_k$. The map $\Phi_j^{(i)}: U_{ij}\times O_k\to X_k$ is then defined by the
composition
\[
U_{ij}\times O_k\lra  (U'_j\setminus U''_j)\times O_k\lra
V_j\times  H_{\kappa_j}\stackrel{\Phi_j}{\lra} X_k,
\]
with the first two arrows the canonical open embeddings.

In any case, since $\Phi_i^{(j)}$, $\Phi_j^{(i)}$ agree on the reduction
$U_{ij}$, there is a biholomorphism $\Psi_{ij}$ of $U_{ij}\times
O_k=U_{ji}\times O_k$ fulfilling
$\Phi_j^{(i)}=\Phi_i^{(j)}\circ\Psi_{ij}$. Using the linear structure on
$U_i\subset \CC^n$ we may then define a homotopy between $\Phi_i^{(j)}$ and
$\Phi_j^{(i)}$ as follows:
\begin{equation}
\label{Eqn: homotopy Phi_{ij}}
\Phi_{ij}: [0,1]\times U_{ij}\times O_k\lra X_k,\quad
\Phi_{ij}(s,.)= \Phi_i^{(j)}\big((1-s)\cdot\id +s\Psi_{ij}\big). 
\end{equation}
Note that $\Phi_{ij}$ is really a homotopy of the homomorphism between the
structure sheaves, the underlying map of topological spaces stays constant
throughout the homotopy.\footnote{The particular form of homotopy is
not important and can be chosen according to convenience.} For a relative
logarithmic $p$-form $\alpha$ on $X_k$ we define $\Phi_{ij}^+(\alpha)$ by
using the product structure of $U_{ij}\times O_k$.
\medskip

Now let $\alpha$ be a closed relative logarithmic $p$-form on $X_k/O_k$.
If $\Phi_i$ is a chart of type~I, we can easily define
$\int_{\beta_i}\alpha$ by integrating over the first factor in
$\Gamma(U_i,\Omega^p_{U_i})\otimes_\CC \CC[t]/(t^{k+1})$.
Explicitly, expanding $\Phi_i^+(\alpha)=\sum_l \alpha_l t^l$, we 
have
\begin{equation}
\label{Eqn: integrals for Ch I}
\int_{\beta_i}\Phi_i^+(\alpha)=\sum_l \Big(\int_{\beta_i}\alpha_l\Big) t^l. 
\end{equation}
An analogous formula defines $\int_{[0,1]\times\gamma_i^\mu}
\Phi_{ij}^+(\alpha)$ needed for the treatment of $\partial\beta_i$ below.

For charts of type~II we need a different definition of the integral to take into
account the change of topology that $\beta_i$ would undergo under
deformation to $t\neq0$. Expanding the three entries of
$\Phi_i^+(\alpha)$ in power series yields
\begin{equation}
\label{Eqn:dform-powerseries-expansion}
\Phi_i^+(\alpha)= \Big({\textstyle \sum_{r\ge 0} z^r p_{\tilde V_i}^* g_r,
\sum_{r\ge 0} w^r p_{\tilde V_i}^*h_r}, \alpha_{\tilde V_i}\Big),
\end{equation}
with $g_r, h_r\in \Omega_{\tilde V_i}^{p-1}(\tilde V_i)=
\Omega_{V_i}^{p-1}(V_i)\otimes \CC[t]/(t^{k+1})$, $\alpha_{\tilde V_i}\in
\Gamma(\tilde U_i, p_{V_i}^*\Omega_{V_i}^p)$. The two power series are absolutely
and uniformly convergent for $|z|\le 1$ and $|w|\le 1$, respectively. For the
two cases listed in (Cy~II), define now
\begin{equation}
\label{Eqn: Type II contribution}
\int_{\beta_i}\Phi_i^+(\alpha) =
\begin{cases}
0,& \beta_i =\gamma_i\times\Sigma\\
\sum_{r\ge 0}\frac{1-t^{(r+1)\kappa_i}}{r+1}
\int_{\gamma_i}\big(h_r-g_r\big),&\beta_i=\gamma_i\times\iota.
\end{cases}
\end{equation}
The motivation for this definition will become clear in the proof of
Proposition~\ref{holomorphic-periods-induce-formal-periods}. The factor in front
of the integral should be recognized as the integral of $w^r dw$ over a curve in
$\DD$ connecting $t^{\kappa_i}$ and $1$. But note that here $t$ is only defined up
to order $k$, so this interpretation should be taken with care.

Finally define $\int_\beta \alpha$ as a formal linear combination
$g+h\log t$ with coefficients $g,h\in\CC[t]/(t^{k+1})$ as
follows:
\begin{equation}\label{Def: finite order integrals}
\int_\beta\alpha:= \sum_i \int_{\beta_i} \Phi_i^+(\alpha)
+\sum_{i,\mu} \int_{[0,1]\times\gamma_i^\mu} \Phi_{ij}^+\alpha
+\Big(\sum_i \kappa_i \int_{\gamma_i} \res_{\Phi_i}(\alpha)\Big) \log t.
\end{equation}
Here the first sum runs over all $i$. The second sum runs over all $(i,\mu)$
with $\Phi_i$ of type~I; in the summand, $j$ is the unique index with $j\neq i$
and $\gamma_i^\mu$ mapping also to $U_j$ as explained in
Construction~\ref{Constr: finite order periods}, (Cy~I); if also $\Phi_j$ is of
type~I we assume $i<j$. The third sum runs over all $i$ with $\Phi_i$ of
type~II. Note that the integral over the residue vanishes if $\beta_i$ is
a cycle of type~(i), that is, of the form $\gamma_i\times\Sigma$.
\end{construction}

\begin{remark}
\label{Rem: adjustments for non-standard starting points}
Formula~\eqref{Def: finite order integrals} depends on the specific choice of
$\iota$ for chains of type (ii) in (Cy~II) above as a curve connecting $z=1$ to
$w=1$. For curves connecting $z=a$ to $w=b$, the term $\kappa_i\log t$ in
\eqref{Def: finite order integrals} has to be replaced by $\kappa_i\log t- \log
b -\log a$, the result of computing $\int_a^{t^{\kappa_i} b^{-1}} \dlog z$.
Varying $a$ and $b$ implies that the result depends on the choice of a branch of
the logarithm, and hence can only be well-defined up to changing any of the
terms $\kappa_i \int_{\gamma_i}\res_{\Phi_i}(\alpha)$ by integral multiples of
$2\pi\sqrt{-1} \int_{\gamma_i}\res_{\Phi_i}(\alpha)$.

This generalized formula also shows that by replacing $z,w,t$ by
$\varepsilon z,\varepsilon w, \varepsilon^{2/\kappa_i} t$ for a small $\varepsilon\in
\CC^*$, the same Formula~\eqref{Def: finite order integrals} applies if we
replace the unit disks above by disks with any radius.
\end{remark}

For convenient referencing, for charts $\Phi_i$ of type (Ch~II) we also
introduce the notation
\begin{equation}\label{Def: finite order integrals-beta_i}
\int_{\beta_i}\alpha:=\int_{\beta_i} \Phi_i^+(\alpha)
+\kappa_i \Big(\int_{\gamma_i} \res_{\Phi_i}(\alpha)\Big)\log t.
\end{equation}
Note that this definition depends on $\Phi_i$ whereas \eqref{Def: finite order
integrals} does not depend on choices, as we show next.

\begin{proposition}
\label{Prop: fiber integral well-defined}
The integral of the closed logarithmic $p$-form $\alpha$ on $X_k/O_k$ over the
$p$-cycle $\beta$ on $X_0$ defined in Equation~\eqref{Def: finite order
integrals} of Construction~\ref{Constr: finite order periods} as a formal
expression
\[
\int_\beta \alpha = g\log t+h 
\]
with $g,h\in \CC[t]/(t^{k+1})$, does not depend on any choices up to changing
$g$ by adding integral multiples of $2\pi\sqrt{-1}
\int_{\gamma_i}\res_{\Phi_i}(\alpha)$ for any $i$. Moreover, up to this
ambiguity, the result is invariant under changing $\alpha$ by an exact form or
under homotopy of $\beta$ through cycles of the same form.
\end{proposition}

\begin{proof}
First observe that for a given cycle $\beta$, we can make $\beta_i$ with
$\Phi_i$ of type~II arbitrarily small. Indeed, let $\Phi_i$ be of type~II and
$\beta_i$ split into a sum $\beta'_i+\beta^{\mathrm{smaller}}_i$ with $\beta'_i$
mapping to $\breve{U}_i:=U'_i\setminus U''_i\subset X_0$. Viewing $\breve
U_i\times O_k$ as an open subspace of $\tilde U_i=V_i\times H_{\kappa_i}$ by
means of $p_1:\tilde U_i\to U'_i\times O_k$, the restriction
$\Phi'_i:=\Phi_i|_{\breve{U}_i\times O_k}$ defines a chart of type~I. A
straightforward check now shows
\[
\int_{\beta_i}\Phi_i^+\alpha = \int_{\beta'_i}(\Phi'_i)^+\alpha
+\int_{\beta^{\mathrm{smaller}}_i}\Phi_i^+\alpha.
\]
A similar refinement argument holds if we swap the roles of $U_i'$ and $U_i''$
and also for charts of type~I. Thus given two systems of open embeddings
$\Phi_i$, $\hat\Phi_j$ we may go over to a larger indexing set and shrink the
domains of definition to arrive at the situation that the indexing sets and the
open sets $U_i\subset X_0$ agree. 

We use the notation from Construction~\ref{Constr: finite order periods}, with a
hat indicating the use of $\hat\Phi_i$. If $\Phi_i$, $\hat\Phi_i$ are charts of
type~I, the same argument as in the definition of $\Phi_{ij}$ defines a homotopy
\[
\Psi_i:[0,1]\times \tilde U_i\lra X_k,
\]
between $\Phi_i=\Psi_i(0,.)$ and $\hat\Phi_i=\Psi_i(1,.)$. There also exists
such a homotopy $\Psi_i$ for charts $\Phi_i, \hat\Phi_i: V_i\times H_{\kappa_i}
\to X_k$ of type~II, but the construction has to be modified to preserve the
equation $zw=t^{\kappa_i}$ as follows. By composing with $\Phi_i^{-1}$ we may
replace $X_k$ by $ V_i \times H_{\kappa_i}$ and assume $\Phi_i=\id$ for the
construction of the homotopy. Write $\hat\Phi_i: V_i\times H_{\kappa_i}\to
V_i\times H_{\kappa_i}$ component-wise as
\[
\hat\Phi_i(u,z,w,t)= \big( U(u,z,w,t), Z(u,z,w,t), W(u,z,w,t), t\big).
\]
Note that $\hat\Phi_i$ commutes with the map to $O_k$ and reduces to the
identity modulo $t$. Hence $ZW=t^{\kappa_i}$ and $Z, W$ reduce to $z,w$ modulo
$t$. A straightforward induction on the degree in $t$ shows that there
exists an invertible function $h$ on $V_i\times H_{\kappa_i}$ with $Z=z\cdot h$,
$W=w\cdot h^{-1}$ and $h\equiv 1$ modulo $t$. Thus we can define $\log h$
uniquely with $\log h\equiv 0$ modulo $t$, and in turn $h^s=\exp(s\cdot\log h)$
for any $s\in \RR$ is also defined. Then
\[
\Psi_i:[0,1]\times V_i\times H_{\kappa_i}\lra V_i\times H_{\kappa_i},\quad
(s,u,z,w,t)\longmapsto \big((1-s)u+ sU, z\cdot h^s, w\cdot h^{-s},t) \big)
\]
defines the desired homotopy between $\Phi_i=\id$ and $\hat\Phi_i$.

Similarly, there exist homotopies
\[
\Psi_{ij}:[0,1]\times[0,1]\times U_{ij}\times O_k\lra X_k
\]
between $\Phi_{ij}=\Psi_{ij}(0,.\,,.)$ and $\hat\Phi_{ij}= \Psi_{ij}(1,.\,,.)$.
By constructing $\Psi_{ij}$ by linear interpolation between $\Psi_i$ and
$\Psi_j$ as we have done, we can also achieve $\Psi_{ij}(.\,,0,.)= \Psi_i$,
$\Psi_{ij}(.\,,1,.)= \Psi_j$. Since $d\alpha=0$ by assumption, these homotopies
give rise to exact forms in the usual way by integration over the first entry:
\[
\Phi_i^*\alpha-\hat\Phi_i^*\alpha= d\Big(
\int_0^1\Psi_i^*\alpha\Big),\quad
\Phi_{ij}^*\alpha-\hat\Phi_{ij}^*\alpha= d\Big(
\int_0^1\Psi_{ij}^*\alpha\Big).
\]
Taking the respective parts of the product decomposition of
$\tilde U_i$ yields the analogous formulas for special pull-back:
\[
\Phi_i^+(\alpha)-\hat\Phi_i^+(\alpha)= d\Big(
\int_0^1\Psi_i^+(\alpha)\Big),\quad
\Phi_{ij}^+(\alpha)-\hat\Phi_{ij}^+(\alpha)= d\Big(
\int_0^1\Psi_{ij}^+(\alpha)\Big),
\]
where again we view the parameters first complex-valued and then
restrict to $[0,1]\times[0,1]\subset D^2$. Note this computation
requires Lemma~\ref{Lem: Phi^+ commutes with d}.

The difference of the terms appearing in the first sum on the right-hand side of
\eqref{Def: finite order integrals} now can be written as
\[
\int_{\beta_i}\big(\Phi_i^+(\alpha)-\hat\Phi_i^+(\alpha)\big)
=\int_{\beta_i} d\Big(\int_0^1\Psi_i^+(\alpha)\Big)
=\int_{\partial\beta_i} \int_0^1\Psi_i^+(\alpha)
= \sum_\mu \int_{\gamma_i^\mu} \int_0^1\Psi_i^+(\alpha).
\]
For the second term one computes similarly
\[
\int_{[0,1]\times \gamma_i^\mu}
\big(\Phi_{ij}^+(\alpha)- \hat\Phi_{ij}^+(\alpha)\big)=
\int_{[0,1]\times\gamma_i^\mu} d\Big(\int_0^1 \Psi_{ij}^+(\alpha)\Big)=
\int_{\gamma_i^\mu} \int_0^1 \big(\Psi_j^+(\alpha)-
\Psi_i^+(\alpha)\big).
\]
Now each $\gamma_i^\mu$ from $\tilde U_i$ of type~I equals a
unique $-\gamma_j^\nu$ with $j\neq i$. If $\tilde U_j$ is of type~I
the contribution of $\gamma_j^\nu$ occurs with opposite sign in
$\int_{\beta_j}\big(\Phi_i^+(\alpha)-\hat\Phi_i^+(\alpha)\big)$.
If $\tilde U_j$ is of type~II a similar cancellation arises with a
contribution of the second term in \eqref{Def: finite
order integrals}, and each summand in the latter
occurs exactly once. Thus the first two terms in \eqref{Def: finite
order integrals} give the same result for $\Phi_i$ and $\hat\Phi_i$,
while the integral over the residue is already defined independently of choices.

A similar argument shows invariance under homotopies of
$\beta$ and the vanishing of $\int_\beta\alpha$ for exact $\alpha$.
\end{proof}

If $X_k\to O_k$ is the reduction modulo $t^{k+1}$ of an analytic family, our
period integral agrees with the usual period integral, up to order $k$,
assuming $\int_{\gamma_i} \res_{\Phi_i}(\alpha)\in\ZZ$. Otherwise we
have agreement up to integral multiples of $\big(2\pi\sqrt{-1}
\int_{\gamma_i}\res_{\Phi_i}(\alpha)\big)\log(t)$.

\begin{proposition} 
\label{holomorphic-periods-induce-formal-periods}
In the situation of Construction~\ref{Constr: finite order periods},
assume that $X_k\to O_k$ and $\alpha$ are the reductions modulo
$t^{k+1}$ of a holomorphic map $\shX\to \DD$ to the unit disk and of a
closed, relative logarithmic $p$-form $\tilde\alpha$ on $\shX$, respectively.
Let $\beta_t$ be a continuous extension of the $p$-cycle $\beta$ on
$X_0=\pi^{-1}(0)$ to the fibers $\shX_t$ for $t\in \DD\setminus (\RR_{>
0}e^{\sqrt{-1}\zeta})$ for some $\zeta\in [0,2\pi)$. Then possibly after
replacing $\DD$ by a smaller disk, there are holomorphic functions
$\tilde g,\tilde h\in \O(\DD)$ with
\[
\int_{\beta_t}\alpha_t = \tilde g\log(t)+\tilde h, \quad
t\in \DD\setminus \RR_{\ge 0}e^{\sqrt{-1}\zeta},
\]
whose reductions modulo $t^{k+1}$ agree with $g,h\in \CC[t]/(t^{k+1})$ from
\eqref{Def: finite order integrals}, respectively, for some choice of branch of
$\log t$ on $\DD\setminus \RR_{\ge 0}e^{\sqrt{-1}\zeta}$, and up to
changing $g$ by integral multiples of $2\pi\sqrt{-1}
\int_{\gamma_i}\res_{\Phi_i}(\alpha)$ for any $i$.
\end{proposition}

\noindent
\emph{Proof.}
After composing $\shX\to\DD$ with multiplication by $e^{\sqrt{-1}(\pi-\zeta)}$
on $\DD$ we may assume $\zeta=\pi$. The charts $\Phi_i$ from
Construction~\ref{Constr: finite order periods} extend to analytic open
embeddings into $\shX$, possibly after shrinking $U_i\subset X_0$ slightly. To
reduce the amount of notation, we use the same symbols as before, except $O_k$
is replaced by the unit disk $\DD$. Thus $\Phi_i: \tilde U_i\to \shX$ continues
to be a logarithmically strict open embedding, but now $\tilde U_i= U_i\times
\DD$ or $\tilde U_i= V_i\times \hat H_{\kappa_i}$ with $\hat
H_{\kappa_i}=\{(z,w,t)\in \hat \DD^2\times \DD\,|\, zw=t^{\kappa_i}\}$.
Similarly, we have the homotopy $\Phi_{ij}$ between the restrictions of $\Phi_i$
and $\Phi_j$ to a neighborhood of $|\gamma_i^\mu|= |\gamma_j^\nu|$, all assumed
to agree to order $k$ with their respective versions in
Construction~\ref{Constr: finite order periods}.

We now extend $\beta=\sum_i\beta_i$ as a cycle to small $t$ by the sum of the
following three types of singular chains.

(A)~If $\tilde U_i=U_i\times \DD$ is of type~I define
$\beta_i(t)={\Phi_i}_*(\beta_i\times\{t\})$.

(B)~If $\tilde U_i= V_i\times \hat H_{\kappa_i}$ is of type~II, then by
(Cy~II) either $\beta_i= \gamma_i\times \Sigma$ or
$\beta_i=\gamma_i\times \iota$. In the first case, $\Sigma= \tilde
H_{\kappa_i}\cap (\ol \DD\times\ol \DD\times\{0\})$ and $\gamma_i$ is a chain in
$V_i$ of dimension $p-2$. In this case define $\beta_i(t)=
\gamma_i\times\Sigma(t)$ with $\Sigma(t) =\hat H_{\kappa_i}\cap (\ol
\DD\times\ol \DD\times\{t\})$. In the second case, $\iota$ is a union of two line
segments in $\hat H_{\kappa_i}$ in the fiber over $t=0$, while $\gamma_i$ is a
chain in $V_i$ of dimension $p-1$. For $t\in\RR_{>0}$ define $\beta_i(t)=
\gamma_i\times\iota(t)$ with
\begin{wrapfigure}[11]{r}{0.4\textwidth}
\captionsetup{width=.9\linewidth}
\begin{center}
\includegraphics[width=0.38\textwidth]{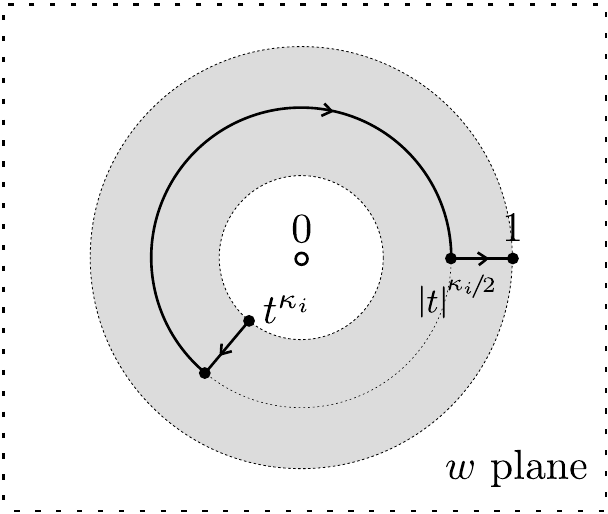}
\end{center}
\caption{The curve $\iota(t)$.}
\label{Fig: w-plane-curves6}
\end{wrapfigure}
\begin{minipage}{0.6\textwidth} 
\[
\iota(t) = \{(t^{\kappa_i}/\lambda, \lambda,t)\in
\hat H_{\kappa_i}\,|\, t^{\kappa_i}\le \lambda\le 1\}.
\]
\end{minipage}\\[2ex]
For $t\in \DD\setminus \RR_{\le 0}$ take the same definition for $\beta_i(t)$, now
with $\iota(t)$ a continuous family of curves in $\hat H_{\kappa_i}$ in the
fiber over $t$ that, projected to the $w$-plane, connects $t^{\kappa_i}$ and $1$
inside the annulus $|t^{\kappa_i}|\le |w|\le 1$. 

For example, writing
$t=|t| e^{\sqrt{-1}\theta}$ with $-\pi< \theta<\pi$, we could take
$\iota(t):[0,1]\to \big\{ w\in\CC\,\big|\, |t|^{\kappa_i}\le |w|\le 1\}$ to map
the three intervals (a) $[0,1/3]$, (b) $[1/3,2/3]$ and (c) $[2/3,1]$ to (a) the
radial line segment connecting $t^{\kappa_i}= |t|^{\kappa_i}
e^{\sqrt{-1}\theta}$ with $|t|^{\kappa_i/2}e^{\sqrt{-1}\theta}$, (b) an arc on a
circle, with endpoint $|t|^{\kappa_i/2}$, (c) the radial line segment from
$|t|^{\kappa_i/2}$ to $1$, respectively, and each interval parametrized with
constant speed. Then indeed $\iota(t)$ varies continuously with
$t\in \CC\setminus\RR_{\le0}$. Moreover, decomposing $\tilde
H_{\kappa_i}\cap \big(\hat \DD^2\times\{|t|\}\big)$ into two annuli of outer
radius $1$ and inner radius $|t|^{1/2}$, we see that $\iota(t)$ for $t\to 0$
converges to the curve $\iota$ in (Cy~II) of Construction~\ref{Constr: finite
order periods}.

(C)~For each $i$ with $\Phi_i$ of type~I and each component
$\gamma_i^\mu=-\gamma_j^\nu$ of $\partial\beta_i$ define the interpolating $p$-chain
\[
\beta_{i,\mu}(t):= {\Phi_{ij}}_*\big([0,1]\times
\gamma_i^\mu\times\{t\}\big).
\]
If also $\Phi_j$ is of type~I we only consider $\beta_{i,\mu}$ if
$i<j$.

Finally, define
\[
\beta(t) =\sum_i\beta_i(t) +\sum_{i,\mu}\beta_{i,\mu}(t).
\]
Note that $\beta(t)$ is a cycle since by construction:
$\partial\beta_i(t)=\sum_\mu \gamma_i^\mu(t)$ with $\gamma_i^\mu(t)$ a
continuous family of cycles converging to $\gamma_i^\mu$ for $t\to 0$, while by
(C) it holds $\partial\beta_{i,\mu}(t)=-\gamma_i^\mu(t)
-\gamma_j^\nu(t)$.

To finish the proof, it remains to compute $\int_{\beta(t)}\alpha$
and to match the various contributions with the terms in \eqref{Def:
finite order integrals}, modulo $t^{k+1}$. For contributions
from~(A) we have
\[
\int_{\beta_i(t)} \alpha = \int_{\beta_i\times\{t\}} \Phi_i^*\alpha.
\]
Developing the integrand in $t$ up to order $k$ yields
$\int_{\beta_i}\Phi_i^+(\alpha)$.

For (B) and $t= |t| e^{\sqrt{-1}\theta}\neq 0$ we may eliminate
$z=t^{\kappa_i}/w$ and work over the $w$-plane. The choice of $w$ over $z$ is
motivated by the fact that the curve $\gamma_i$ moves radially inwards in the
$z$-plane and outwards in the $w$-plane. In this coordinate, using $z^rdz =
-t^{(r+1)\kappa_i}w^{-r-2}dw$ and following \eqref{Eqn: decompose p-form}, \eqref{Eqn:dform-powerseries-expansion}, we can write uniquely
\begin{equation}\label{Eqn: decompose alpha}
\Phi_i^*\alpha= \sum_{r\ge0}  p_{\tilde V_i}^*h_r\wedge w^r dw +
p_{\tilde V_i}^*h_{-1} \wedge\frac{dw}{w}
-\sum_{r\ge 0}  p_{\tilde V_i}^*g_r\wedge\frac{t^{(r+1)\kappa_i}}{w^{r+2}} dw+ \alpha_{\tilde V_i},
\end{equation}
with $ h_r, g_r\in \Gamma(V_i\times \DD,\Omega^{p-1}_{V_i\times \DD/\DD})$ and
$\alpha_{\tilde V_i}\in\Gamma(\tilde U_i, p_{V_i}^*\Omega^p_{V_i})$.
Projected to the $w$-plane, $\iota(t)$ is a curve connecting
$t^{\kappa_i}$ and $1$. If $\beta_i(t)=\gamma_i\times \Sigma(t)$, the integral
over $\Phi_i^*\alpha$ involves integration of a holomorphic one-form over $\Sigma$ and hence
it vanishes identically, in agreement with the first line in \eqref{Eqn:
Type II contribution}.

For the other case, $\beta_i(t)=\gamma_i\times \iota(t)$, we have
\[
\int_{\iota(t)} w^r dw = \int_{t^{\kappa_i}}^1 w^r dw =\begin{cases}
\frac{1}{r+1}\big(1- t^{(r+1)\kappa_i}\big),&r\neq -1\\
-\kappa_i\log t,&r=-1.
\end{cases}
\]
Moreover, $\int_{\gamma_i\times \iota(t)} \alpha_{\tilde V_i}=0$ since
$\alpha_{\tilde V_i}$ vanishes on $\ker ({p_{\tilde V_i}}_*)$. Integration of
\eqref{Eqn: decompose alpha} over $\gamma_i\times \iota(t)$ now gives
\[
\left(-\kappa_i \int_{\gamma_i}  h_{-1} \right)\log t+
\sum_{r\ge 0} \frac{1-t^{(r+1)\kappa_i}}{r+1}\left( \int_{\gamma_i}
h_r- g_r\right).
\]
Since $g_r, h_r$ reduce modulo $t^{k+1}$ to the differential forms with the same
symbols in Construction~\ref{Constr: finite order periods} and since
$g_{-1}=-h_{-1}=\res_{\Phi_i}(\alpha)$, this result agrees to order $k$ with the
contributions to $\int_\beta\alpha$ in~\eqref{Def: finite order integrals}
from~\eqref{Eqn: Type II contribution} and with $\kappa_i\big(\int_{\gamma_i}
\res_{\Phi_i}(\alpha)\big)\log t$.

For the interpolation integrals~(C) it holds
\[
\int_{\beta_{i,\mu}(t)}\alpha =
\int_{[0,1]\times\gamma_i^\mu\times\{t\}}\Phi_{ij}^*\alpha,
\]
which agrees with $\int_{[0,1]\times \gamma_i^\mu}
\Phi_{ij}^+(\alpha)$ in \eqref{Def: finite order integrals} upon
reduction modulo $t^{k+1}$ by the same argument as in (A).

Any other choice of $\beta(t)$ differs from our choice up to homology by a sum
of integrals over vanishing cycles of the form $\int_{{\Phi_i}_*(\gamma_i\times
S^1\times\{t\})} \alpha$ for $\Phi_i$ of type~II. Here
$S^1\times\{t\}\subset\hat H_{\kappa_i}$ is defined by $|z|=|w|$.
Integrating over the $S^1$-factor yields $2\pi\sqrt{-1}\int_{\gamma_i}
\res_{\Phi_i}(\alpha)$, hence only changes the result as stated.
\qed

\begin{lemma} \label{lemma-on-monodromy}
In the situation of Proposition~\ref{holomorphic-periods-induce-formal-periods},
let $T$ denote the monodromy transformation on $n$-cycle classes along a
counter-clockwise simple loop in the base disk $\DD$ of the family $\shX\to \DD$ based
at a fiber $\shX_{t_0}$ for some $t_0\neq 0$. We have
\[
(T-\id)(\beta_{t_0})=\sum \kappa_i[\gamma_i\times S^1]
\]
where, in the notation of the proof of the proposition, the sum is over all
charts of type (B) for which $\beta_{i}=\gamma_i\times\iota$ and $S^1$ denotes a
clockwise simple loop around the origin in the $w$-plane, see Figure~\ref{Fig:
w-plane-curves6}. 
\end{lemma}

\begin{proof}
The cycle $\beta_{t_0}$ decomposes into chains $\beta_i$ according to cases
(A),(B),(C) as in the proof of
Proposition~\ref{holomorphic-periods-induce-formal-periods}. For (A) and (C), it
is straightforward to see that $\beta_i$ is invariant under monodromy because
the family is trivial here. Hence, $(T-\id)$ only yields contributions for case
(B). Note further that the factor $\gamma_i$ is also invariant under monodromy,
so we only need to consider the local situation of the map $H_{\kappa_i}\to \DD$
given by $zw=t^{\kappa_i}$. In the sub-situation where $\beta_i=\gamma_i\times
\Sigma(t)$, we find that $\Sigma(t_0)$ is the fundamental chain of the fiber of
$H_{\kappa_i}\to \DD$ which is also invariant under monodromy. Thus, only the
situation $\beta_i=\gamma_i\times \iota$ contributes, as claimed in the
assertion. Studying how $\iota$ changes when following a simple
counter-clockwise $t_0$-based loop in $\DD$, as illustrated in Figure~\ref{Fig:
w-plane-curves6}, we see that $\iota$ gets mapped to $\iota+\kappa_i[S^1]$ under
$T$. Adding the invariant factor $\gamma_i$ yields the assertion.
\end{proof}

%
%

\section{Analytic approximation of proper formal families}

\begin{theorem}\label{Thm: Analytic approximation}
Let $R=\CC\{t,z_1,\ldots,z_r\}/(g_1,\ldots,g_s)$ be a convergent power series
algebra, $(S,0)\subset (\CC^{r+1},0)$ the corresponding germ of analytic space
and $(\widetilde S,0)$ the completion in the closed subspace defined by $t$. Let
$\widetilde \pi: (\foX,X_0)\to (\widetilde S,0)$ be a proper and flat formal
analytic map and $\pi_k: X_k\to S_k$ its reduction modulo $t^{k+1}$.

Then for any $k\ge 0$ there exists a proper flat analytic map of germs of pairs
$\pi:(\shX,X_0)\to (S,0)$ with reduction modulo $t^{k+1}$ isomorphic to $\pi_k$.

Analogous approximation statements hold for morphisms of complex spaces
$(\foZ,Z_0)\to(\foX,X_0)$, both of which proper and flat over $(\tilde S,0)$,
and for the additional giving of an $(\tilde S,0)$-flat analytic subspace
$(\foD,D_0)$ of $(\foX,X_0)$.
\end{theorem}

\begin{proof}
We first treat the case $(\foX.X_0)\to (\tilde S,0)$. By a result of Douady and
Grauert, the compact complex space $X_0$ admits a versal deformation, a proper
analytic map $h:\shY\to V$ with a point $v\in V$ and an isomorphism
$h^{-1}(v)\simeq X_0$ which is versal for proper flat analytic deformations of
$X_0$ (\cite{Du74}, VII.8, Th\'eor\`eme Principal and \cite{Gr74}, \S5,
Hauptsatz). Possibly by shrinking $V$, we may assume $V$ is an analytic subspace
of an open subset in $\CC^n$ given by some $f_1,\ldots,f_m\in
\CC\{x_1,\ldots,x_n\}$ and $v=0$. Thus for any given $k$ there exists a
cartesian diagram of analytic spaces
\begin{equation}
\label{Eqn: varphi_k}
\begin{aligned}
\xymatrix{
X_k\ar[r]\ar[d]_{\pi_k}&\shY\ar[d]^h\\
S_k \ar^-{\Phi_k}[r]& V.\!}
\end{aligned}
\end{equation}
We are going to construct $(\shX,X_0)\to (S,0)$ by extending $\Phi_k$ to an
analytic map $\Phi: (S,0)\to (V,0)$, first formally and then analytically using Artin
approximation.

To do so, denote by $\widehat\pi:\widehat\foX\lra \widehat S$ and by $\widehat h:\widehat
\shY\to \widehat V$ the completions of $\widetilde\pi$ and $h$ at the
origins, respectively. By results of Schuster and Wavrik \cite{Schuster71},
\cite{Wavrik}, the family $\widehat h$ is formally versal. Hence there
exists a cartesian square
\[
\xymatrix{
\widehat\foX\ar[r]\ar[d]_{\widehat \pi}&\widehat\shY\ar[d]^{\widehat h}\\
\widehat S\ar^-{\widehat\Phi}[r]&\, \widehat V.\!}
\]
We can also achieve that the reduction of $\widehat\Phi$ modulo $t^{k+1}$ agrees
with the completion $\widehat\Phi_k$ of $\Phi_k$ at $0$. Indeed, if $\fom\subset
R$ is the maximal ideal, constructing $\widehat\Phi$ amounts to finding a
compatible system of lifts
\[
\widehat\Phi_l: \O_{\widehat V,0}= \CC\lfor x_1,\ldots,x_n\rfor/(f_1,\ldots,f_m)\lra
R/\fom^{l+1},\quad l\in\NN,
\]
along with a compatible system of isomorphisms of $X_l\to S_l$ with the pull-back of
$\widehat h$ by $\widehat\Phi_l$. Assuming $\widehat\Phi_{l-1}$ given,
the construction of $\widehat\Phi_l$ can be done in two steps: First provide the lift to
\[
R/\big(\fom^{l+1}+(t^{k+1})\cap \fom^l\big) =
R/\fom^l\times_{R/(\fom^l+(t^{k+1}))} R/\big(\fom^{l+1}+(t^{k+1})\big),
\]
by gluing the family over $R/\fom^l$ given by $\widehat\Phi_{l-1}$ with the
reduction modulo $\fom^{l+1}$ of $X_k/S_k$, using the given isomorphism of the
common reductions modulo $\fom^l+ (t^{k+1})$. The fibered sum of analytic spaces
involved in this step exists due to \cite{Schuster70}, Satz~2.7. Then in the
second step use formal versality to extend from $R/\big(\fom^{l+1}+(t^{k+1})\cap
\fom^l\big)$ to $R/\fom^{l+1}$. Thus $\widehat\Phi$ with the requested
properties exists.

Writing $\mathbf{z}=(z_1,\ldots,z_r)$, the map $\widehat\Phi$ is given by
equations $x_i=\widehat\varphi_i(t,\mathbf{z})$ for $1\le i\le n$ with
$\widehat\varphi_i\in \CC\lfor t,\mathbf{z}\rfor$. Since the ideal
$(f_1,\ldots,f_m)$ gets mapped into the ideal $(g_1,\ldots,g_s)$, the
$\widehat\varphi_i$ fulfill the system of equations
\begin{equation}
\label{Eqn: Formal eqn}
f_j\big(\widehat\varphi_1(t,\mathbf{z}),\ldots,\widehat\varphi_n(t,\mathbf{z})\big)
=\sum_{\sigma=1}^s \widehat a_{j\sigma}(t,\mathbf{z}) g_\sigma(t,\mathbf{z}),\quad 1\le j\le m
\end{equation}
for some $\widehat a_{j\sigma}\in\CC\llbracket{t,\bf z}\rrbracket$. Since we
already have the analytic solution $\Phi_k$ on $S_k$, that is, an analytic
solution modulo $t^{k+1}$, we now rewrite
\[
\widehat\varphi_i=\varphi_i+ t^{k+1}\widehat\psi_i,\quad i=1,\ldots,m,
\]
with $\varphi_i\in \CC\{t,\mathbf{z}\}$ the components of $\Phi_k$ and
$\widehat\psi_i\in\CC\lfor t,\mathbf{z}\rfor$. Plugging into \eqref{Eqn: Formal
eqn} we see that $y_i=\widehat\psi_i(t,\mathbf{z})$, $x_{j\sigma}=\widehat
a_{j\sigma}(t,\mathbf{z})$ are a formal solution of the system of analytic equations
\begin{equation}
\label{Eqn: system-of-equations-to-approximate}
f_j\big(\varphi_1(t,\mathbf{z})+t^{k+1} y_1,\ldots,\varphi_n(t,\mathbf{z})+t^{k+1}y_n\big)
=\sum_{\sigma=1}^s x_{j\sigma} g_\sigma(t,\mathbf{z}),\quad 1\le j\le m.
\end{equation}
By Artin's approximation theorem \cite{Ar68}, Theorem~1.2, there exist germs of
analytic functions $\psi_1(t,\mathbf{z}),\ldots,\psi_n(t,\mathbf{z})$ and
$a_{j\sigma}(t,\mathbf{z})$ that solve \eqref{Eqn:
system-of-equations-to-approximate}. Now $\varphi_1+t^{k+1}\psi_1,\ldots,
\varphi_n+t^{k+1}\psi_n$ defines an analytic map $(S,0)\to (V,0)$ with the
property that the reduction modulo $t^{k+1}$ equals $\Phi_k$. The base change
$\shX:=\shY\times_V S$ of $\shY\to V$ by $\Phi$ is the requested analytic
approximation of $\widetilde\pi$. This finishes the proof for the case
$(\foX,X_0)\to (\tilde S,0)$.

The proof for the case of a morphism $(\foZ,Z_0)\to (\foX,X_0)$ is similar,
replacing the versal deformation of $X_0$ by the versal deformation of the
morphism $Z_0\to X_0$, with varying domain and target. This latter versal
deformation space exists by first constructing versal deformations $\shT\to W$
of $Y_0$ and $\shY\to V$ of $X_0$ separately, and then taking the relative hom
space $\Hom_{W\times V} (\shT\times V,W\times\shY)$ from \cite{Du69}, Ch.10, for the
pull-backs to $W\times V$ of the versal deformations of domain and target. The
case of an analytic subspace is a special case, noting that the condition that a
morphism is a closed embedding is an open property.
\end{proof}

%
%

\section{The divisorial log deformation functor has a hull}
\label{App: log hull}
In this section, $X_\CC$ denotes a simple toric log Calabi-Yau space over
$(\Spec\CC,\NN\times\CC^\times)$. We consider divisorial deformations of $X_\CC$
as defined in [GS2], Definition~2.7. Let $\shD:(\hbox{Artinian $\CC\lfor
t\rfor$-algebras})\to (\hbox{Sets})$ be the divisorial log deformation
functor hence that associates to an Artinian $\CC\lfor
t\rfor$-algebra $A$ the set of isomorphism classes of divisorial log
deformations $X_A$ of $X_\CC$ over $\Spec A$ equipped with the divisorial log
structure defined by $t=0$. The definition requires $X_A\to\Spec A$ to be flat
in the ordinary sense, to be log smooth away from $Z$ and to permit local models
of a particular type along $Z$. The last condition requires that each
$\bar x\in Z$ has an \'etale neighborhood $V_A$ with strict \'etale $\Spec
A$-morphisms $X_A\la V_A \to Y_A=\Spec A\times_{\CC\lfor t\rfor} U$ where $U\to
\Spec\CC\lfor t\rfor$ is a particular affine toric variety with monomial
function $t$ uniquely determined by $\bar x$. In the following we call
such an \'etale neighborhood $V_A\to X_A$ a \emph{model neighborhood}. The
only feature of these local models needed for the present discussion is the
following result from \cite{logmirror2}.

\begin{lemma} 
(\cite{logmirror2}, Lemma~2.15.)
\label{lem-local-unique}
For every $\bar x\in Z$, there exists a model neighborhood $V_\CC$ of $\bar x$
in $X_\CC$, so that for every Artinian $\CC\lfor t\rfor$-algebra $A$, any two
divisorial log deformations of $V_\CC$ over $\Spec A$ are isomorphic.
\footnote{The statement in \cite{logmirror2} only asserts the existence of some
\'etale neighborhood, but the proof in fact shows the stronger statement given
here.} 
\end{lemma}

A standard fact about \'etale maps (Remark 2.4 in \cite{logmirror2}) is the following:

\begin{lemma} 
\label{etale-equivalent-categories}
If $Y$ is a log scheme and $Y_0\subset Y$ a closed subscheme defined by a
nilpotent sheaf of ideals with restriction of the log structure from $Y$, then
the category of strict \'etale $Y$-schemes is equivalent to the category of
strict \'etale $Y_0$-schemes by means of $V\mapsto V\times_Y Y_0$.
\end{lemma}

\begin{lemma} 
\label{lemma-prep-pushout}
Assume that $A_1\to A_0\la A_2$ are maps of Artinian $\CC\lfor t\rfor$-algebras and
$X_{A_1}\la X_{A_0}\to X_{A_2}$ maps of divisorial log deformations above these.
Given $\bar x\in Z$, there is the following commutative diagram with all squares
cartesian, rows local models at $\bar x$, left column the given maps of
deformation and the right column the maps induced via pullback by
$U\to\Spec\CC\lfor t\rfor$,
\[ 
\xymatrix@R=1.5em{
X_{A_1} & V_{A_1}\ar[l]\ar[r] &Y_{A_1}\\
X_{A_0}\ar[u]\ar[d] & V_{A_0}\ar[l]\ar[r]\ar[u]\ar[d] &Y_{A_0}\ar[u]\ar[d]\\
X_{A_2} & V_{A_2}\ar[l]\ar[r] &Y_{A_2}.
}
\]
\end{lemma} 

\begin{proof}
Let $X_\CC\la V_\CC\to Y_\CC$ be a model neighborhood of $\bar x\in Z$ in
$X_\CC$ provided by Lemma~\ref{lem-local-unique}. Then Lemma~\ref{etale-equivalent-categories}
implies that for any Artinian $\CC\lfor t\rfor$-algebra $A$ and divisorial log
deformation $X_A\in\shD(A)$ of $X_\CC$, there exists a model neighborhood
$X_A\la V_A\to Y_A$ restricting to $X_\CC\la V_\CC\to Y_\CC$. Moreover, this
model neighborhood is unique up to unique isomorphism. Thus the extension of
the given model neighborhood $X_\CC\la V_\CC\to Y_\CC$ to divisorial log
deformations of $X_\CC$ is functorial, which in particular gives the stated
commutative diagram.
\end{proof}

An important fact implied from the definition is that the log structure on $X_A$
has integral stalks even though it typically is not coherent. Recall that a
morphism $f:(X,\shM_X)\to (Y,\shM_Y)$ of log spaces with integral monoid stalks
is strict if and only the induced map $f^{-1}\overline\shM_Y\to \overline\shM_X$
is an isomorphism.

\begin{lemma}
\label{div-strict}
If $f:X_A\to X_{A'}$ is a map of divisorial log deformations over a
homomorphism of Artinian $\CC\lfor t\rfor$-algebras, then $f$ is strict.
\end{lemma}

\begin{proof}
By strictness of $X_\CC\to X_A$, $X_\CC\to X_{A'}$, the map
$f^{-1}\ol\M_{X_{A'}}\to \ol\M_{X_A}$ induced by $f$ is an isomorphism. The statement now
follows by integrality of stalks.
\end{proof}

\begin{lemma} \label{pushing-out-etale-maps} 
Assume we have a commutative diagram of Noetherian rings
\[
\begin{CD}
B_1@>>>B_0@<b<< B_2\\
@VVV@VVV@VVV\\
C_1@>>> C_0@<c<< C_2
\end{CD}
\]
with $b,c$ surjective with nilpotent kernel, the squares co-cartesian and all
vertical maps flat and unramified, then the natural map $f:B_1\times_{B_0} B_2\to
C_1\times_{C_0} C_2$ is also flat and unramified.
\end{lemma}

\begin{proof} 
The question is local, so we can assume all rings local and $B_i\to C_i$ are
standard, that is, $C_i=B_i[T]/(P_i)$ for $P_i\in B_i[T]$ monic with
simple roots. By co-cartesianness, we may assume $P_0$ is the image of $P_1,P_2$
under $B_i[T]\to B_0[T]$. Thus $P_1, P_2$ define a polynomial $P\in
(B_1\times_{B_0} B_2)[T]$, which is clearly monic. Moreover, $P$ also
has simple roots because any double root would imply a double root also for all
the other $P_i$ using that $\Spec (B_1\times_{B_0} B_2)\to\Spec B_1$ is
bijective by surjectivity of $b:B_2\to B_0$. Finally, we find that
$C_1\times_{C_0} C_2=(B_1\times_{B_0} B_2)[T]/(P)$, which implies the assertion.
\end{proof}

Let $A_1\to A_0\la A_2$ be homomorphisms of Artinian $\CC\lfor t\rfor$-algebras.
Consider the natural map
\begin{equation}
\label{eq-DD-map}
\shD(A_1\times_{A_0}A_2)\lra \shD(A_1)\times_{\shD(A_0)}\shD(A_2).
\end{equation}
The Schlessinger criteria that provide a hull are the following
(\cite{Schlessinger}, Theorem~2.11).
\begin{itemize}
\item[(H1)] The map \eqref{eq-DD-map} is surjective whenever $A_2\to A_0$ is surjective.
\item[(H2)] The map \eqref{eq-DD-map} is bijective whenever $A_0=\CC$ and
$A_2=\CC[\eps]:=\CC[E]/E^2$.
\item[(H3)] $\dim_k (t_{\shD})<\infty$ where $t_{\shD}:=\shD(\CC[\eps])$.
\end{itemize}

\begin{theorem}
\label{Thm: divisorial hull}
The divisorial log deformation functor $\shD$ has a hull.
\end{theorem}

\begin{proof}
The last criterion (H3) is proved in \cite{logmirror2},
Theorem~2.11,(2). It remains to verify (H1) and (H2). We begin with (H1). Let
$A_2\to A_0$ be surjective. Set $A:=A_1\times_{A_0}A_2$ and note this is
naturally an Artinian $\CC\lfor t\rfor$-algebra. Let $(X_i,\shM_i)\to
\Spec{A_i}$ be divisorial log deformations lifting the maps $A_0\to A_i$. Just
as in the proof of (H1) for the log smooth deformation functor in \cite{fkato},
we obtain a glued log space $(X_A,\shN)$ via $\shN:=\shM_1 \times_{\shM_0}
\shM_2 \to \shO_{X_1}\times_{\shO_{X_0}}\shO_{X_2} =:\shO_{X_A}$ with log map to
$(\Spec A,\NN\times A^\times)$ compatible with restrictions to $X_0,X_1,X_2$. In
view of Lemma~\ref{div-strict}, we have $\overline\shM_1=\overline\shM_0=
\overline\shM_2=:\overline\shM$ and there is a natural map
$\alpha:\overline\shN\to \overline\shM_1\times_{\overline\shM_0} \overline\shM_2
= \overline\shM$ that we claim is an isomorphism. Indeed,
since $\ol \shM_2\to\ol\shM_0$ is surjective, $\alpha$ is easily seen to be
surjective. Now assume $(m_1,m_2),(m_1',m_2')\in\ol\shN$ map to the same element
under $\alpha$. Then $m_1=\eps_1m_1', m_2=\eps_2m_2'$ for $\eps_i\in
\shM_i^\times$. The cancellation law in $\shM_0$ gives that $\eps_1$ and
$\eps_2$ map to the same element in $\shM_0^\times$, hence glue to an element of $\shN^\times$. Thus $(m_1,m_2)$ and $(m_1',m_2')$ map to the same element in $\ol\shN$, proving injectivity of $\alpha$. By the same argument as in Lemma~\ref{div-strict}, we now
know that $X_i\to X_A$ are strict.

Away from the incoherent locus $Z$, it was argued in \cite{fkato} that
$(X_A,\shN)$ is a log smooth lifting of $X_0$. It remains to show the existence
of local models along $Z$ (which then also implies the flatness along $Z$). Let
$\bar x\in Z$ be a geometric point. Lemma~\ref{lemma-prep-pushout} provides a
diagram of local models and we use the push-out for each row to obtain the
following commutative diagram
\[ 
\resizebox{11.4cm}{!}{
\xymatrix@C=.6em{
&&&&&V_{0}\ar[ld] \ar[rd] \ar@/_0.2pc/[dllll] \ar@/^0.2pc/[drrrr]\\
&X_{0}\ar[ld] \ar[rd]&&&V_{1}\ar[rd]|(.40)\hole\ar|(.55)\hole@/_0.2pc/[dllll]\ar|(.24)
\hole|(.48)\hole@/^0.2pc/[drrrr]&&V_{2}\ar[ld]\ar@/_0.2pc/[dllll]\ar@/^0.2pc/[drrrr]
&&&Y_{0}\ar|(.5)\hole[ld] \ar[rd]\ar@<2pt>[ddrrrr]\\
X_{1}\ar[rd]&&X_{2}\ar[ld]& & &V_{A}\ar@{-->}@/_0.2pc/[dllll]\ar@{-->}@/^0.2pc/[drrrr]&&&
Y_{1}\ar[rd]\ar|(.4)\hole[drrrrr]&&Y_{2}\ar[ld]\ar[drrr]\\
&X_{A}&&&&&&&&Y_{A}\ar[rrrr]&& && U.
}}
\]
The dashed maps are \'etale by Lemma~\ref{pushing-out-etale-maps} and $Y_A$
agrees with $\Spec A\times_{\Spec\CC\lfor t\rfor}U$. The strictness of all
vertical maps follows from the strictness of $X_i\to X_A$ proved above.
Lemma~\ref{div-strict} then also gives strictness of the dashed maps, using
that $X_1\to X_A$ is a homeomorphism on underlying spaces. We now have
obtained local models for $X_A$, so $X_A$ is a divisorial log deformation of
$X_\CC$ that maps to $(X_1,X_2)$ under the map in \eqref{eq-DD-map}. Thus this
map is surjective, finishing the proof of (H1).

Finally we turn to (H2), for which only injectivity is left to be shown. Let $A_0=\CC$
and $A_2=\CC[\eps]$. Using the same reasoning as in \cite{fkato}, Proof of (H2),
it suffices to prove the following assertion (Lemma 9.2 in \cite{fkato}).

If $(X_A',\shN')\to (\Spec A_i,\NN\times A_i^\times)$ is a divisorial log
deformation that fits in a commutative square
\[
\begin{CD}
(X_{1},\shM_1)@>>> (X_A',\shN')\\
@AAA @AAA\\
(X_{0},\shM_0)@>>> (X_{2},\shM_2)
\end{CD}
\]
so that the restriction maps to $(X_{i},\shM_i)$ for $i=1,2$ induce
isomorphisms, then the natural map $f:(X_A,\shN)\to (X_A',\shN')$ is an isomorphism. The proof in
\cite{fkato} works for us away from $Z$, so it remains to prove $f$ is an
isomorphism along $Z$. Let $\bar x\in Z$ be a point and let $V'_{A}\to
X_A'$ be the strict \'etale neighborhood of $\bar x$ obtained from
the neighborhood $V_\CC$ of $\bar x$ in $X_\CC$ via
Lemma~\ref{etale-equivalent-categories}. Then Lemma~\ref{lem-local-unique}
provides a $\Spec A$-isomorphism $V'_{A}\stackrel{\sim}\to V_A$. Since
the restrictions to $X_\CC\la V_\CC$ are compatible isomorphisms,
Lemma~\ref{etale-equivalent-categories} shows this isomorphism commutes with
$f$. In particular, $f$ is an isomorphism at $\bar x$, completing the proof.
\end{proof}

%
%

\section{Isomorphism of affine and algebraic \texorpdfstring{$H^1(\Theta)$}{H1(Theta)}}
\label{appendix-C}
Let $(B,\P,\varphi)$ be a simple tropical manifold and $x\in
\Spec(\CC[H^1(B,i_*\check\Lambda)^*])$ a closed point. Let $X_x$ denote the fiber of
the canonical family above it. In particular $(B,\P)$ is the \emph{intersection
complex} of $X_0(B,\P)$ and also of $X_x$. Occurrences of $\tau,\sigma$ with
various indices below will always refer to cells in $\P$. Inclusions of closed
strata are covariant: $\tau_0\subset\tau_1\Rightarrow X_{\tau_0}\subset
X_{\tau_1}$. Note that, inconveniently, in order to parse all upcoming
references to \cite{logmirror2}, a mental translation to the \emph{dual
intersection complex} as used in \cite{logmirror2} must be made. The translation is
straightforward, but nonetheless potentially confusing. For $\sigma \in \P$, let
$V_\sigma$ denote the standard open set of $X_x$ that is the open star of the
dense torus of the stratum $X_\sigma$, i.e. the disjoint union of the dense
torus orbits of all $X_{\sigma'}$ for $\sigma'$ containing $\sigma$. Note the
contravariance: $V_{\sigma_1}\subset V_{\sigma_0}$ for ${\sigma_0}\subset
{\sigma_1}$. Refine the partial order $\subseteq$ of $\P$ to a total order $\le$
so that for any sheaf $\shF$ on $X_x$ we obtain a \v{C}ech complex $\check
C^j(\{V_\sigma\}_\sigma,\shF)=\bigoplus_{\sigma_0<\ldots<\sigma_j}\Gamma(V_{\sigma_0}\cap
\ldots\cap V_{\sigma_j},\shF)$ with the usual \v{C}ech differential $\check
C^j\to \check C^{j+1}$. A decoration with $\dagger$ refers to the space with log
structure (given by $t=0$). Following \cite{logmirror2}, let $j:X_x\setminus
Z_x\hra X_x$ denote the open inclusion of the locus where the log structure is
coherent and then we write short
\[
\Omega^r:=j_*\Omega^r_{X_x^\dagger/x^\dagger},\qquad \Theta
:=j_*\Theta_{X_x^\dagger/x^\dagger}.
\]
The main purpose of this section is to prove the following proposition. For the statement, recall that $W_\tau\subset B$ denotes the open set given by the disjoint union of the
relative interiors of all cells in the barycentric subdivision of $\P$ that
contain the barycenter of $\tau$. 

\begin{proposition} 
\label{open-cover-is-iso}
We have a natural isomorphism of \v Cech cohomologies
\[
\textstyle
H^1\big((W_\tau)_\tau ,\iota_*\bigwedge^{n-1}\Lambda\otimes\CC\big)\lra
\check H^1(\{V_\tau\}_\tau,\Omega^{n-1}).
\]
Moreover, if $B$ is orientable, using the global volume forms
$\iota_*\bigwedge^{n}\Lambda\simeq\ul\ZZ$ and $\Omega^{n}\simeq\shO_{X_x}$, this
isomorphism can be rewritten as
\[
\check H^1\big((W_\tau)_\tau,\iota_*\check\Lambda\otimes\CC\big)
\lra \check H^1(\{V_\tau\}_\tau,\Theta).
\]
Explicitly, the image of a cocycle $\big(n_{\omega\tau}\big)_{\omega,\tau}$ with
$n_{\omega\tau}\in \Gamma(W_\omega\cap W_\tau, \iota_*\check\Lambda\otimes\CC)$
is the cocycle $\big(\partial_{n_{\omega\tau}} \big)_{\omega\tau}$ with
$\partial_{n_{\omega\tau}}\in\Gamma(V_\omega\cap V_\tau,\Theta)$ the logarithmic
vector field defined by $n_{\omega\tau}$.
\end{proposition}

For $\tau_0\subseteq\tau_1$, recall from
\cite{logmirror2}, Lemma 3.20,\footnote{\cite{logmirror2} uses the notation
$e:\tau_0\to\tau_1$ but we stick with $\tau_0\subseteq\tau_1$ assuming no
self-intersecting cells, as in \cite{theta}.}
the isomorphism
\begin{equation}
\label{part2-main-technical-iso}
\textstyle
\Gamma(W_{\tau_0}\cap W_{\tau_1},\iota_*\bigwedge^r\Lambda\otimes\CC)= \Gamma(X_{\tau_0},(\Omega^r_{\tau_1}|_{X_{\tau_0}})/{\shT ors})
\end{equation}
where $\Omega^{r}_{\tau}=\kappa_{\tau,*}\kappa_{\tau}^*(\Omega^r|_{X_\tau})$ for $\kappa_{\tau}:X_\tau\setminus Z_\tau \hra X_\tau$ the open embedding defined in loc.cit..
Note that, by the argument in the proof of Theorem 3.21 of \cite{logmirror2}, we have 
\begin{equation}
\label{only-outer-ones-matter}
\textstyle
\Gamma(W_{\tau_0}\cap\ldots\cap W_{\tau_i},\iota_*\bigwedge^r\Lambda\otimes\CC)=
\Gamma(W_{\tau_0}\cap W_{\tau_i},\iota_*\bigwedge^r\Lambda\otimes\CC).
\end{equation}
We will use this identification to define a map of double complexes ${\bf\Lambda}^{i,j}\ra
{\bf\Omega}^{i,j}$ and we only care about $r=n-1$ for $n=\dim B$. The first of
these double complexes arises as a ``doubling'' of the \v{C}ech complex of
$\iota_*\bigwedge^{n-1}\Lambda$ with respect to the cover $\{W_\tau\}_{\tau\in
\P}$, as follows:
\begin{equation}
\label{def-lambdaij}
\textstyle
{\bf\Lambda}^{i,j}=\bigoplus_{\sigma_0\subsetneq\ldots\subsetneq\sigma_j\subseteq
\tau_0\subsetneq\ldots\subsetneq\tau_i} \Gamma((W_{\tau_0}\cap
W_{\tau_i})\cap(W_{\sigma_0}\cap W_{\sigma_j}),\iota_*\bigwedge^{n-1}\Lambda\otimes\CC).
\end{equation}
Note that $W_{\sigma_1}\cap W_{\sigma_2}=\emptyset$ unless
$\sigma_{1}\subseteq\sigma_2$ or $\sigma_{2}\subseteq\sigma_1$, so summing over
$\sigma_0\subsetneq\ldots\subsetneq\sigma_j$ gives the \v{C}ech
complex for the cover $\{W_\tau\}_\tau$:
\[\textstyle
\bigoplus_{j=0}^n \Lambda^{i,j}=
\check C^i\big( (W_\tau)_\tau,\iota_*\bigwedge^{n-1} \Lambda\otimes\CC\big).
\]
Using \eqref{only-outer-ones-matter}, the differential $i\to i+1$ is the usual
alternating sum of the \v{C}ech-differential, and similarly for the differential $j\to
j+1$.

The second double complex is
\begin{equation}
\label{def-omegaij}
{\bf\Omega}^{i,j}=\bigoplus_{\sigma_0\subsetneq\ldots\subsetneq\sigma_j\subseteq \tau_0\subsetneq\ldots\subsetneq\tau_i} 
\Gamma(X_{\tau_0}\cap V_{\sigma_0}\cap\ldots\cap V_{\sigma_j},(\Omega^{n-1}_{\tau_i}|_{X_{\tau_0}\cap V_{\sigma_j}})/{\shT ors}).
\end{equation}
The differential $i\to i+1$ is the differential $d_{\operatorname{bct}}$ given
in \cite[p.736, just before Theorem 3.9]{logmirror2} and the differential $j\to
j+1$ is a \v{C}ech-type-differential analogous to the one in
${\bf\Lambda}^{i,j}$. Note however that, unlike for the cover
$\{W_\sigma\}_\sigma$, we may have $V_{\sigma_1}\cap V_{\sigma_2}\neq \emptyset$
even if none of $\sigma_1,\sigma_2$ is contained in the other. We will later use
Lemma~\ref{partial-to-total-order} to take care of this fact.

\begin{lemma}
\label{dbl-cplx-map}
There is a natural injection of double-complexes $\Phi:{\bf\Lambda}^{i,j}\ra
{\bf\Omega}^{i,j}$.
\end{lemma}

\begin{proof} 
Given $\sigma_0\subseteq\ldots\subseteq \sigma_j\subseteq\tau_0$, the torus-invariant
open subset $X_{\tau_0}\cap V_{\sigma_0}\cap\ldots\cap V_{\sigma_j}=X_{\tau_0}\cap
V_{\sigma_j}$ of the toric variety $X_{\tau_0}$ is of the form $V=\Spec\CC[P]$
for $P$ a toric monoid. By \cite[Lemma 3.12 and Proposition 3.17]{logmirror2},
there is an injection $\Omega^r_{\tau}\hra \Omega^r_{X_{\hat\tau}}(\log \partial
X_{\hat\tau})|_{X_{\tau}}=\shO_{X_\tau}\otimes_\ZZ \bigwedge^r
\Lambda_{\hat\tau}$ for any maximal cell $\hat\tau$ containing $\tau$.
Here, $1\otimes (m_1\wedge\ldots\wedge m_r)$ gets identified with
$\frac{dz^{m_1}}{z^{m_1}}\wedge\ldots\wedge\frac{dz^{m_r}}{z^{m_r}}$. When
changing the choice of $\hat\tau$, identifying $\Lambda_{\hat\tau}$ with another
$\Lambda_{\hat\tau'}$ generally depends on the chosen path in $B\setminus\Delta$, and
furthermore the gluing data rescale the monomials. Both of these won't bother us
for the following reasons. We will only be interested in the subsheaf
$\CC\otimes_\ZZ \bigwedge^r \Lambda_{\hat\tau}$ which is actually invariant
under this torus action, because the scaling operation $z\mapsto \lambda z$
leaves $\frac{dz}z$ invariant. Even better, we will actually only care about the
monodromy invariant part of this subsheaf. With this in mind, in view of
\eqref{def-lambdaij}, it is straightforward to produce the following map
\begin{align*}
\textstyle
\Gamma((W_{\tau_0}\cap W_{\tau_i})\cap(W_{\sigma_0}\cap W_{\sigma_j}),\iota_*\bigwedge^{n-1}\Lambda\otimes\CC)&\stackrel{\eqref{only-outer-ones-matter}}{=}
\textstyle
\CC\otimes\Gamma(W_{\sigma_0}\cap W_{\tau_i},\iota_*\bigwedge^{n-1}\Lambda)\\
&\hra \Gamma(X_{\tau_0}\cap V_{\sigma_j},(\Omega^{n-1}_{\hat\tau}|_{X_{\tau_0}\cap V_{\sigma_j}})/{\shT ors})
\end{align*}
and its image is contained in $\Gamma(X_{\tau_0}\cap
V_{\sigma_j},(\Omega^{n-1}_{\tau_i}|_{X_{\tau_0}\cap V_{\sigma_j}})/{\shT
ors})$. This gives an injection from the sum in \eqref{def-lambdaij} to the one
in \eqref{def-omegaij}. The map respects the differentials by what we said
before and by the functoriality of \v{C}ech-type complexes.
\end{proof}

\begin{remark}
We never used that $r=n-1$, so a similar map as in the previous lemma exists for
any $r$. In fact, the statement of Lemma~\ref{dbl-cplx-map} can be upgraded to
an injection of triple complexes when taking the de Rham differential for
${\bf\Omega}^{i,j}$ and an additional trivial differential $r\to r+1$ on
${\bf\Lambda}^{i,j}$.
\end{remark}

We need a technical lemma before we can prove Proposition~\ref{open-cover-is-iso}. 
For a sheaf $\shF$ on $X_x$, consider the exact sequence of complexes
\begin{equation*} 
\resizebox{\textwidth}{!}{$
0\ra \hspace{1ex} \overbrace{\displaystyle\hspace{-1ex}\bigoplus_{\substack{\sigma_0<\ldots<\sigma_j\\ \exists
k<j:\sigma_k\not\subset \sigma_{k+1}}}\Gamma(V_{\sigma_0}\cap \ldots.\cap
V_{\sigma_j},\shF)}^{K^j:=} \ra\check C^j(\{V_\sigma\}_\sigma,\shF)\stackrel{e}{\ra}
\displaystyle\bigoplus_{\sigma_0\subsetneq\ldots\subsetneq\sigma_j}\Gamma(V_{\sigma_0}\cap
\ldots.\cap V_{\sigma_j},\shF) \ra 0.$}
\end{equation*}

\begin{lemma}
\label{partial-to-total-order}
The surjection $e$ is a quasi-isomorphism. Denoting by $d_2$ the differential in
the second index of ${\bf\Omega}^{i,j}$, we conclude that
\begin{equation}
\label{eq-coho-j}
H^p_{d_2}({\bf\Omega}^{i,\bullet})= \bigoplus_{\tau_0\subsetneq\ldots\subsetneq\tau_i}
H^p(X_{\tau_0},(\Omega^{n-1}_{\tau_i}|_{X_{\tau_0}})/{\shT ors}).
\end{equation}
\end{lemma}

\begin{proof}
We show that $K^\bullet$ is acyclic. Note that $V_{\sigma_0}\cap \ldots\cap
V_{\sigma_j}=\emptyset$ unless there is a $\sigma\in\P$ that contains
$\sigma_0,\ldots,\sigma_j$. Let $\langle\sigma_0,\ldots,\sigma_j\rangle$ denote
the set of minimal elements with respect to $\subseteq$ in the set of all
$\sigma\in\P$ that contain $\sigma_0,\ldots,\sigma_j$ (e.g. for $(B,\P)$ two
intervals glued to form a circle and $\sigma_0,\sigma_1$ being the two vertices,
we have $\langle\sigma_0,\sigma_1\rangle$ is the set containing the two
intervals). The use of this definition is the following observation 
\[
V_{\sigma_0}\cap \ldots.\cap V_{\sigma_j}=\bigsqcup_{\sigma\in \langle\sigma_0,\ldots,\sigma_j\rangle} V_{\sigma}.
\]
Let $K_{\sigma}^\bullet$ be the subcomplex of $K^\bullet$ consisting of the
summands for open sets contained in $V_\sigma$, in particular involving only
summands of terms for $\sigma_0<\ldots<\sigma_j$ where $\sigma$ contains
$\sigma_0,\ldots,\sigma_j$. The subcomplexes $K_{i}^\bullet=\sum_{\dim\sigma=i}
K_{\sigma}^\bullet$ form a filtration of $K^\bullet$. We show that the gradeds
of this filtration are acyclic. Let $\bar K_{\sigma}^\bullet$ be the quotient of
$K_{\sigma}^\bullet$ defined by removing terms where $V_\sigma$ is properly
contained in a component of $V_{\sigma_0}\cap\ldots\cap V_{\sigma_j}$ rather
than equal to a component, precisely
\[
\bar K_{\sigma}^j=\bigoplus_{{\sigma_0<\ldots<\sigma_j}\atop{\sigma \in\langle\sigma_0,\ldots,\sigma_j\rangle}} \Gamma(V_{\sigma},\shF).
\]
We have $K_i/K_{i+1}=\bigoplus_{\dim \sigma=i}\bar K_{\sigma}^\bullet$. We find
that $\bar K_{\sigma}^\bullet$ is isomorphic to the cone of an isomorphism of
complexes and thus acyclic. Indeed, write $\bar K_{\sigma}^\bullet = {}_1\bar
K_{\sigma}^\bullet\ra {}_2\bar K_{\sigma}^\bullet$ where ${}_1\bar
K_{\sigma}^\bullet$ gathers the summands with $\sigma_j\neq\sigma$ and ${}_2\bar
K_{\sigma}^\bullet$ gathers the summands with $\sigma_j=\sigma$. We finished
showing the acyclicity of $K^\bullet$.

Since $V_\sigma$ is affine for each $\sigma$ and so are their intersections,
$\{V_\sigma\}_\sigma$ forms an affine cover of $X_x$. Let $q_\tau:X_\tau\ra X_x$
denote the inclusion of the stratum. We have for
$\shF:=q_{\tau,*}(\Omega^{n-1}_{\tau_i}|_{X_{\tau_0}})/{\shT ors}$ that the
quasi-isomorphism $e$ is a map between the \v{C}ech complex of $\shF$ and a
summand of the complex $({\bf\Omega}^{i,\bullet},d_2)$. Summing the maps $e$ for
all these summands and taking cohomology yields \eqref{eq-coho-j}.
\end{proof}

\begin{proof}[Proof of Proposition~\ref{open-cover-is-iso}]
The main tool is the injection of double complexes $\Phi$ from
Lemma~\ref{dbl-cplx-map}. We consider taking cohomology for the second
differential $j\ra j+1$ for both complexes ${\bf\Lambda}^{i,j}$ and
${\bf\Omega}^{i,j}$. Lemma~\ref{partial-to-total-order} gives the result for the
complex ${\bf\Omega}^{i,j}$. By the proof of Theorem~3.22 and Lemma~3.20 in
\cite{logmirror2}, using that $r=n-1$, we find that
$H^p_{d_2}({\bf\Omega}^{i,\bullet})=0$ for $p>0$. Note that this fails in
general for other degrees $r$ as was found in \cite{Ru10},\,Theorem~1.6. We also
have that $H^p_{d_2}({\bf\Lambda}^{i,\bullet})=0$ for $p>0$ since the cover
$\{W_\tau\}_\tau$ is acyclic for $\iota_*\bigwedge^{r}\Lambda$ for any $r$ by
\cite[Lemma 5.5]{logmirror1}. Thus, taking cohomology by the differential $j\ra
j+1$ on source and target of $\Phi$ simultaneously yields a map induced by
$\Phi$ that is concentrated in degrees $(i,0)$. The resulting complexes for the
remaining differential $i\ra i+1$ the isomorphism of barycentric complexes that
led to the proof of \cite[Theorem 3.22]{logmirror2}. Consequently and relevant
for us is the conclusion that $\Phi:{\bf\Lambda}^{i,j}\ra{\bf\Omega}^{i,j}$ is a
quasi-isomorphism on the total complex of the double complex.

We next consider what happens when we first take cohomology under the first
differential $d_1$, that is, $i\ra i+1$. All cohomology groups at $i>0$ vanish:
for ${\bf\Lambda}^{i,j}$ again because of the acyclicity of the cover
$\{W_\tau\}_\tau$ and for ${\bf\Omega}^{i,j}$ by the exactness of the
barycentric differentials \cite{logmirror1},\,Proposition A.2 and
\cite{logmirror2},\,Theorem 3.5.

Therefore, since $\Phi$ is a quasi-isomorphism, also the induced map on
$d_1$-cohomology that is concentrated in $i=0$,
\begin{equation}
\label{Eqn: tilde Phi}
\textstyle
\tilde\Phi:\bigoplus_{\sigma_0\subsetneq\ldots\subsetneq\sigma_j} \Gamma(W_{\sigma_0}\cap W_{\sigma_j},\iota_*\bigwedge^{n-1}\Lambda\otimes\CC) \lra \bigoplus_{\sigma_0\subsetneq\ldots\subsetneq\sigma_j} \Gamma(V_{\sigma_j},\Omega^{n-1})
\end{equation}
is a quasi-isomorphism under the remaining differential $d_2$, that is, $j\ra j+1$. 
Taking cohomology with respect to $d_2$ and composing with the inverse of the quasi-isomorphism $e$ from Lemma~\ref{partial-to-total-order}, we conclude the assertion.
\end{proof}

\end{appendix}


\sloppy

\end{document}